\documentclass[12pt]{amsart}
\usepackage[utf8]{inputenc}
\usepackage{comment}
\usepackage[toc]{glossaries}
\usepackage[T1]{fontenc}
\usepackage{lmodern}
\usepackage{relsize}
\usepackage{amsfonts}
\usepackage{amssymb}
\usepackage{amsthm}
\usepackage{amsmath}
\usepackage{mathtools}
\usepackage{hyphenat}
\usepackage{enumitem}
\usepackage{microtype}
\usepackage[margin=3.1cm]{geometry}
\usepackage{amsthm}
\usepackage{graphicx}
\graphicspath{ {./images/} }
\usepackage{hyperref}
\usepackage{tikz-cd}
\usepackage{xcolor}
\usetikzlibrary{calc}
\usepackage{etoolbox}
\usepackage{cleveref}
\hypersetup{
    colorlinks=false,
    linkcolor=black,
    filecolor=black,      
    urlcolor=black,
}
\usepackage{blindtext}

 \usepackage[
 backend=biber,
 style=alphabetic,
 ]{biblatex}

\theoremstyle{plain}
\newtheorem{thm}{Theorem}[section]

\newtheorem{prop}[thm]{Proposition}
\AfterEndEnvironment{prop}{\noindent\ignorespaces}
\newtheorem{ex}[thm]{Example}
\AfterEndEnvironment{ex}{\noindent\ignorespaces}
\newtheorem{lemma}[thm]{Lemma}
\AfterEndEnvironment{lemma}{\noindent\ignorespaces}

\AfterEndEnvironment{lem}{\noindent\ignorespaces}

\AfterEndEnvironment{plem}{\noindent\ignorespaces}
\newtheorem{rmk}[thm]{Remark}
\AfterEndEnvironment{rmk}{\noindent\ignorespaces}

\AfterEndEnvironment{rec}{\noindent\ignorespaces}
\newtheorem{df}[thm]{Definition}
\AfterEndEnvironment{df}{\noindent\ignorespaces}
\newtheorem{cor}[thm]{Corollary}
\AfterEndEnvironment{cor}{\noindent\ignorespaces}
\newtheorem{nt}[thm]{Notation}
\AfterEndEnvironment{nt}{\noindent\ignorespaces}

\AfterEndEnvironment{con}{\noindent\ignorespaces}

\AfterEndEnvironment{note}{\noindent\ignorespaces}

\AfterEndEnvironment{as}{\noindent\ignorespaces}

\AfterEndEnvironment{sol}{\noindent\ignorespaces}

\AfterEndEnvironment{conj}{\noindent\ignorespaces}
\newtheorem{axiom}{A}
\newcommand{\aref}[1]{\hyperref[#1]{A~\ref{#1}}}

\newcommand{\Sh}{\operatorname{Sh}}

\newcommand{\Ext}{\operatorname{Ext}}

\newcommand{\length}{\operatorname{length}}

\newcommand{\Hom}{\operatorname{Hom}}

\newcommand{\RHom}{\operatorname{RHom}}

\newcommand{\Cl}{\operatorname{Cl}}

\newcommand{\chr}{\operatorname{char}}

\newcommand{\Image}{\operatorname{Im}}

\newcommand{\Path}{\operatorname{Path}}
\newcommand{\link}{\operatorname{link}}
\newcommand{\Int}{\operatorname{Int}}
\newcommand{\intt}{\operatorname{int}}
\newcommand{\Cell}{\operatorname{\textbf{Cell}}}

\newcommand{\Mor}{\operatorname{Mor}}
\newcommand{\Ob}{\operatorname{Ob}}
\newcommand{\Iso}{\operatorname{Iso}}
\newcommand{\Ent}{\operatorname{Ent}}
\newcommand{\Aut}{\operatorname{Aut}}
\newcommand{\Sets}{\operatorname{Sets}}
\newcommand{\red}{\operatorname{red}}
\newcommand{\sk}{\operatorname{sk}}
\newcommand{\End}{\operatorname{End}}

\newcommand{\pouya}[1]{{\color{red} \sf
$\clubsuit\clubsuit\clubsuit$  [#1]}}

\newcommand{\david}[1]{{\color{blue} \sf
  [#1]}}

\newcommand{\newterm}{\textsf}

\begin{document}
\title{Topological Koszulity for Category Algebras} %%%%%%%%%%%%
\author{David Favero}
\author{Pouya Layeghi}
\date{\today}
%\address{Address}
\newcommand{\Addresses}{{% additional braces for segregating \footnotesize
  \bigskip
  \footnotesize

  \noindent \textsc{David Favero and  Pouya Layeghi}
  \\
  \textsc{School of Mathematics, University of Minnesota } \\
   \textsc{ Vincent Hall, 206 Church Street SE
Minneapolis, MN 55455}
\par\nopagebreak
  \noindent \textit{E-mail address}: \texttt{favero@umn.edu, layeg001@umn.edu}
}}
%\email{favero@umn.edu}
%\email{layeg001@umn.edu}

\begin{abstract}
We give a topological description of Ext groups between simple representations of categories via a nerve type construction. We use it to show that the Koszulity of indiscretely based category algebras is equivalent to the locally bouquet property of this nerve. We also provide a class of functors which preserve the Koszulity of category algebras called almost discrete fibrations.  Specializing from categories to posets, we show that the equivalence relations of V. Reiner and D. Stamate in \cite{rs} are exactly almost discrete fibrations and recover their results.  As an application, we classify when a shifted dual collection to a full strong exceptional collection of line bundles on a toric variety is strong.
%Specializing from categories to posets, we obtain the graded case of a result of \cite{rs} %which generalizes a result of I. Peeva, V. Reiner and B. Sturmfels \cite{prs}, a result of P. Polo \cite{pp}, and a result of D. Woodcock \cite{dw}. 
\end{abstract}

\maketitle

\tableofcontents
\section{Introduction}
In \cite{r}, Reisner combinatiorally classified when the quotient of a polynomial ring by a squarefree monomial ideal is Cohen-Macaulay in terms of the input poset of monomial generators.  It was later shown that the Koszulity of the incidence algebra of the poset is equivalent to a similar condition called the locally Cohen-Macaulay property \cite{iz, pp, ap, rs}. 
%Ultimately, Koszulity can be understood in terms of the topology of the order complex (it is equivalent to the Cohen-Macaulay property of the intervals).

Incidence algebras are special cases of category algebras which are more flexible and appear naturally in many contexts.  This begs the question, can one similarly classify the Koszulity of category algebras?  We answer this question affirmatively for what we call indiscretely based categories:\ graded categories whose degree zero connected components are indiscrete categories.  That is, we define weakening of the Cohen-Macaulay condition for categories (see Definition~\ref{locally bouquet}) which generalizes that of posets to obtain the following:
\begin{thm}
  Let $\mathcal{C}$ be an indiscretely based category.  Then, $k\mathcal{C}$ is Koszul if and only if $\mathcal{C}$ is locally bouquet (cohomologically). %satisfies the equivalent conditions of \Cref{prop: ext1 vanishes}.. 
\end{thm}

 The categorical viewpoint has another advantage - one can study Koszulity in terms of functors. We define a class of such functors which preserve Koszulity called almost discrete fibrations (which specialize to discrete fibrations, see Corollary \ref{cor: almost discrete fibration}). Restricting back to posets, we obtain a categorical description of the equivalence relations axiomatized in \cite{rs}:
 \begin{thm} \label{thm: intro rs}
    Let $P$ be a graded poset regarded as a category. There is a one-to-one correspondence between the equivalence relations on $P$ axiomatized in \cite{rs}  and almost discrete fibrations $F:P\to \mathcal{D}$ where $\mathcal{D}$ is an indiscretely based category which satisfies the property that any morphism in $\mathcal{D}$ lifts to $P$. % up to the equivalences of indiscretely based categories $\sigma:\mathcal{D}\to \mathcal{D}'$ with $F'=\sigma\circ F$.
\end{thm}

One of the motivating factors for this work was a class of category algebras called homotopy path algebras \cite{dj} (HPAs).  Given a stratified topological space $X$, one can consider the category of entrance paths in $X$ up to homotopy.  This category is like a directed fundamental group and its category algebra is an HPA.  In \cite{dj}, a sufficient condition for the Koszulity of an HPA was characterized by shellability of the intervals in the path poset.  Our result  specializes in this case to the necessary and sufficient Cohen-Macaulay property.    
\begin{thm}
A graded HPA is Koszul if and only if %the order complex $K(p,q)$ on the poset $(p,q)$ for any two paths $p,q$ in $Q$
     the path poset is locally Cohen-Macaulay.    
\end{thm}

This result is readily applicable to algebraic geometry.  Namely, the endomorphism algebra of a full strong exceptional collection of line bundles on a toric variety is always an HPA \cite[Lemma 5.2]{dj}. Therefore, we obtain a topological interpretation of Koszulity for such algebras. In the special case of Bondal-Thomsen HPAs \cite[Definition 5.10]{dj}, we get a more concrete topological interpretation of Koszulity. %based on unions of cubes.
\begin{thm}
   A graded Bondal-Thomsen HPA $A_{\Phi}$ is Koszul if and only if 
$$\widetilde{H}^i(\mathbb R^n_{(D,E)})=0$$ for all open intervals $(D,E)$ in the poset of the induced stratification on $\mathbb R^n$ and $i$ less than the maximal chain length in $(D,E)$ where $\mathbb R^n_{(D,E)}$ is the union of strata labeled by elements of $(D,E)$.

%$(D,E)\subset\widetilde{\Image\Phi}$ and $i$ less than the maximal chain length in $(D,E)$ where $\widetilde{\Image\Phi}$ is the poset of the induced stratification on $\mathbb R^{n}$ and $\mathbb R^n_{(D,E)}$ is the union of stratas labeled by elements of $(D,E)$.
   
%$$\widetilde{H}^i(\cup_{F\in (D,E)}(-F+[0,1)^{n+k}))=0$$ 
%for all $(D,E)\subset\Image\widetilde{\Phi}$ and $i<\length[D,E]-2$ where $\widetilde{\Image\Phi}$ is the poset of the induced stratification on $\mathbb R^{n}$. 
\end{thm}
Combining our combinatorial interpretation of these endomorphism algebras  with a result of Bondal \cite[Corollary 7.2, p.\ 8]{b}, we obtain a combinatorial interpretation for strongness of the dual exceptional collection. 
\begin{thm}
Let $\mathcal O_X(D_1), ..., \mathcal O_X(D_n)$ be any collection of line bundles on a toric variety $X$.
 Let $A$ be the endomorphism algebra of $\bigoplus_{i=1}^n \mathcal O_X(D)$.  
 For $1 \leq i \leq n$, consider the poset $P_i$ of all monomials in the Cox ring of degree $D_j - D_i$. Then $A$ is Koszul iff.\ $P_i$ is locally Cohen-Macaulay for all $i$.
Furthermore, if the collection is a full strong exceptional collection such that the partial ordering by morphisms on objects is graded, then the above are equivalent to the existence of integers $d_1, ..., d_n$ such that the shifted dual exceptional collection $S_1[d_1], ..., S_n[d_n]$ is strong. 
\end{thm}

The paper is organized as follows. In section \S\ref{sec: notation}, we fix our notation, conventions, and some essential terminology.
Section \S\ref{sec: top Koszulity} handles category algebras, their Ext groups, and Koszulity in the abstract setting.  Using a nerve type construction, we define a generalization of open intervals in posets to categories. We use this to give a topological description of Ext groups and Koszulity.  In section \S \ref{sec: RS equivalences}, we prove Theorem~\ref{thm: intro rs} which reinterprets the work of Reiner and Stamate in the language of category algebras. Section \S \ref{sec: App. in HPAs} discusses group actions, homotopy path algebras, and applications to Bondal-Thomsen homotopy path algebras and full strong exceptional collections of line bundles in algebraic geometry. 

\subsection{Acknowledgement}
This project was financially supported by the National Science Foundation DMS-2302262. We are very grateful for many helpful discussions and email exchanges with Victor Reiner.  We would also like to thank Peter Webb, Maru Sarazola, Sasha Voronov, Michael K. Brown, Pranjal Dangwal, Jesse Huang, and Mykola Sapronov for other useful input and discussions. 
\subsection{Notations and Conventions}\label{sec: notation}
\begin{itemize}
    \item $R$ is a commutative ring with identity.
    \item $k$ is a field.
    \item For a graded module $M$, $M(n)_k=M_{n+k}$.
    \item A \newterm{semisimple object} (e.g. module, ring, algebra, ...) is a (possibly infinite) direct sum of simple objects. 
    \item The tail of a path/morphism $p$ is denoted by $t(p)$ and  the head is denoted by $h(p)$. 
    \item We write path concatenation from left to right. For example if we have two paths 
    \begin{center}
        $\bullet_u	\xrightarrow{p}\bullet_v$, $\bullet_v\xrightarrow{q}\bullet_w$
    \end{center}
    we would write $pq$ for their composition 
    $$\bullet_u\xrightarrow{p}\bullet_v\xrightarrow{q}\bullet_w$$
    However, when viewed as morphisms, we write $q\circ p$ for their composition. 
    \item All categories are small.
    \item An \newterm{indiscrete category} is a category which has exactly one morphism between any two objects. Note that indiscrete categories are groupoids without automorphisms. 
    \item By a \newterm{skeletal category} we mean a category where any two isomorphic objects are equal. 
    \item A \newterm{saturated subset} $S\subseteq P$ of a poset $P$, is a subset satisfying the property that if $a,b\in S$ and $a\leq c\leq b$, then $c\in S$. 
\end{itemize}
\section{Topological Koszulity}\label{sec: top Koszulity}
In this section, we introduce our main object of study:\ category algebras.  This is a very general class of algebras which includes matrix algebras, free algebras, polynomial rings, path algebras of quivers, incidence algebras, and homotopy path algebras. We begin by establishing all the necessary definitions, conventions, and notation.  After that, we  provide a cellular resolution of the diagonal bimodule of a category algebra from the classifying space. This resolution leads the way to a topological interpretation of Ext groups and hence Koszulity.

\subsection{Category algebras preliminaries}
Category algebras are an extremely general phenomenon.  To any category $\mathcal C$ and any unital commutative ring $R$ we can associate a (possibly non-unital) $R$-algebra as follows. 
\begin{df}
Let $\mathcal{C}$ be a category and $R$ be commutative ring with unit. The \newterm{category algebra} $R\mathcal{C}$ of $\mathcal{C}$ over $R$ is defined to be the free $R$-module with basis given by the morphisms of $\mathcal{C}$ and multiplication defined by composition or zero: $$\alpha\beta:=\begin{cases}
   \alpha\circ \beta & \text{if $\alpha$ and $\beta$ are composable}\\
    0 & \text{otherwise} 
\end{cases}$$  
\end{df}
If $\mathcal{C}$ has finitely many objects, then $\sum_{x \in \text{Ob}(\mathcal C)} \text{Id}_x$ exists and is the unit of $R\mathcal{C}$; if $\mathcal C$ has infinitely many objects $R\mathcal C$ is not unital.
\begin{ex}
    Let $P$ be a poset.  We may regard it as a category where there exists a unique morphism $a \to b$ iff. $a\leq b$. In this case, the category algebra $RP$ is usually called the incidence algebra.
\end{ex}
  \begin{ex}
\label{hpa quiver category}
Homotopy path algebras (see \cite{dj}) generalize the previous example. Let $Q$ be a finite quiver possibly with cycles and loops. A \newterm{homotopy path algebra (HPA)} is a quotient $A=RQ/I_S$ where $Q$ is a finite quiver (possibly with cycles and loops) and 
$$I_S=\langle p-q:(p,q)\in S\rangle $$ where $S$ is any set of pairs of paths such that $h(p)=h(q)$ and $t(p)=t(q)$ and $I_S$ is left and right cancellative. This is the category algebra $R\mathcal C_A$ where $\mathcal C_A$ is the \newterm{path category} of $A$, that is the category whose objects are the vertices of $Q$ and whose morphisms are paths modulo the equivalence relation determined by $I_S$.
\end{ex}
\begin{rmk}
In \cite{dj}, cycles and loops are not allowed but we allow them here.
\end{rmk}
\begin{ex}
\label{matrix algebras are category algebras}
    Let $\mathcal{C}$ be an indiscrete category with $n$ objects $v_1,...,v_n$. Then, $k\mathcal{C}$ is isomorphic to $ M_{n\times n}(R)$ by  sending a morphism $f$ to the matrix $\delta_{t(f)h(f)}$, i.e. the matrix which has $1$ in its $t(f)h(f)$ entry and zero everywhere else. 
\end{ex}
\begin{ex}
\label{ex: R[x] is a category algebra}
 Let $R$ be a ring and $\mathcal{C}$ be a graded category with one object $v$ and  morphisms $\Hom_{\mathcal{C}}(v,v)=\{x^n:n\in \mathbb N\}$. Then, $R\mathcal{C}\cong R[x]$   
\end{ex}
\begin{prop}
\label{prop: tensor product of cat algebras}
  Let $\mathcal{C}$ and $\mathcal D$ be two categories and $R$ be a ring. Then 
  $$R\mathcal{C}\otimes_RR\mathcal{D}\cong R(\mathcal{C}\times\mathcal{D})$$  
\end{prop}
\begin{proof}
    One can easily check
    $$\phi: R\mathcal{C}\otimes_R R\mathcal{D}\to R(\mathcal{C}\times \mathcal{D})$$
    $$f\otimes g\mapsto (f,g)$$
    and 
    $$\psi:R(\mathcal{C}\times \mathcal{D})\to  R\mathcal{C}\otimes_R R\mathcal{D}$$
    $$(f,g)\mapsto f\otimes g$$
    are mutually inverse isomorphisms of $R$-algebras. 
\end{proof}
\begin{ex}
\label{ex: polynomial ring in several variables}
  Let $R$ be a ring and $\mathcal{C}$ be a graded category with one object $v$ and  morphisms $\Hom_{\mathcal{C}}(v,v)=\{x^n:n\in \mathbb N\}$. Then  
\begin{center}
\begin{align*}
    R(\underbrace{\mathcal{C}\times ...\times \mathcal{C}}_{\text{n-times}})&=\underbrace{R\mathcal{C}\otimes_R ...\otimes_R R\mathcal{C}}_{\text{n-times}}\\
    &=\underbrace{R[x]\otimes_R...\otimes_R R[x]}_{\text{n-times}}\\
    &=R[x_1,...x_n]
\end{align*}
\end{center}   
\end{ex}

\begin{thm}[\protect{\cite[Proposition 2.1, p. 5]{p}\cite[Theorem 7.1, p. 33]{MITCHELL19721}}]
\label{category representations vs modules}
Let $\mathcal{C}$ be a category, $(R-mod)^{\mathcal{C}}$ be the category of representations of $\mathcal{C}$ and $R\mathcal{C}-mod$ be the category of $R\mathcal{C}$-modules. Then, there are exact functors 
\begin{center}
	       \begin{tikzcd}
                    (R-mod)^{\mathcal{C}}\arrow[bend left]{r}{\Phi}
		          & R\mathcal{C}-mod \arrow[bend left]{l}{\Psi}
	       \end{tikzcd}     
\end{center} 
such that $\Psi\circ \Phi\cong Id$, $\Phi$ is a full and faithful left adjoint for $\Psi$ and is an equivalence if $\mathcal{C}$ has finitely many objects. 
\end{thm}
\begin{proof}
The details can be found in \cite[Proposition 2.1, p. 5]{p} and \cite[Theorem 7.1, p. 33]{MITCHELL19721}.  We define the functors $\Phi$ and $\Psi$ on objects here for later use.     \begin{enumerate}
        \item $\Phi$:
        \\ Let $F\in \Ob((R-mod)^{\mathcal{C}})$. Define $\Phi(F)=\bigoplus_{x\in \Ob(\mathcal{C})}F(x)$ and if $f:y\to z\in R\mathcal{C}$, then $f$ sends an element $u\in F(x)$ to $F(f)(u)$ if $x=y$ and to $0$ otherwise.
        \item $\Psi$: 
        \\ Let $M\in \Ob(R\mathcal{C}-mod)$. For $x\in \Ob(\mathcal{C})$ define $(\Psi(M))(x)=Id_x \cdot M$ as the $R$-module obtained by acting on $M$ by the identity on $x$. For $f:x\to y\in \Mor(\mathcal{C})$, define $(\Psi(M))(f):(\Psi(M))(x)\to (\Psi(M))(y)$ to be the function which sends $u\in Id_x \cdot M$ to $f \cdot u\in Id_y \cdot M$. 
    \end{enumerate}
The exactness of $\Phi$ and $\Psi$ follows from the statement of this proposition in \cite{MITCHELL19721} and from the fact that $R-mod$ is an abelian category and all small colimits exist in $R-mod$ (or more generally $R-mod$ is a Grothendieck category).
\end{proof}
\begin{rmk}
\label{ext groups can be computed in the functor category}
Theorem \ref{category representations vs modules} allows us to work in $(R-mod)^{\mathcal{C}}$ instead of $R\mathcal{C}-mod$.  In  $(R-mod)^{\mathcal{C}}$ things are more readily comparable to categorical language. For example, since $\Phi$ is an exact left adjoint it preserves projectives and since it is an exact functor, it preserves resolutions. So to compute Ext groups of objects in the image of $\Phi$ in $R\mathcal{C}-mod$, we can use their projective resolutions in $(R-mod)^{\mathcal{C}}$. 
\end{rmk}
\begin{nt}
    Abusing notation, we will not write $\Phi$ when we are working with the image of objects of $(R-mod)^{\mathcal{C}}$ via $\Phi$ in $R\mathcal{C}-mod$. For example, later if $P_v$ (or $S_v$) is an object of $(R-mod)^{\mathcal{C}}$, we will use $P_v$ (or $S_v$) for $\Phi(P_v)$ (or $\Phi(S_v))$ as well. 
\end{nt}

To discuss Koszulity, we require a grading on $R\mathcal{C}$.  Such a grading can naturally be obtained from a grading  on the category itself i.e.\ a functor $l:\mathcal{C}\to \mathbb{N}$ where $\mathbb{N}$ is viewed as a category with one object. Precisely, from the grading $l$ on the category, we obtain a set of morphisms:
 \[
\mathcal{C}_i:=l^{-1}(i).
 \]
 and by setting 
 $R\mathcal{C}_i=\bigoplus_{\mathcal{C}_i}R$
 we get the following grading on the category algebra
\[
R\mathcal{C} = \bigoplus_i R\mathcal C_i.
\]
Category algebras come equipped with some nice classes of modules. For example, projective $R\mathcal{C}$-modules can be obtained as follows.
\begin{df}
     For each object $v\in \mathcal{C}$, the associated \newterm{linearized representable functor} is the functor
     \begin{align*}
       P_v:\mathcal{C} & \to R-mod \\
       w & \mapsto \bigoplus _{\Hom_{\mathcal{C}}(v,w)} R
     \end{align*} i.e. $P_v(w)$ is the free $R$-module with basis $\Hom_{\mathcal{C}}(v,w)$. 
 \end{df}
Note that these are naturally objects in $(R-mod)^{\mathcal{C}}$ but our notational abuse will view $P_v$ as an honest module in $R\mathcal{C}-mod$ as well.
 \begin{prop}[\protect{\cite[Proposition 4.4, p. 11]{p}}]
 \label{Yoneda}
 Let $v$ be an object of $\mathcal{C}$ and $M:\mathcal{C}\to R-mod$ be a representation of $\mathcal{C}$. Then, $\Hom_{R\mathcal{C}}(P_v,M)\cong M(v)$. 
 \end{prop}
As a quick corollary of this proposition we get the following:
\begin{cor}
 Let $v$ be an object of $\mathcal{C}$. Then, $P_v$ is projective in both $(R-mod)^{\mathcal{C}}$ and $R\mathcal{C}-mod$
\end{cor}
\begin{proof} 
    A sequence of functors $F_1 \to F_2 \to F_3$ is exact if and only if it is exact at each object $F_1(v) \to F_2(v) \to F_3(v)$ since limits and colimits are computed point-wise in functor categories. By Proposition \ref{Yoneda}, $\Hom_{(R-mod)^\mathcal{C}}(P_v,M)\cong M(v)$ and hence $\Hom_{(R-mod)^\mathcal{C}}(P_v,- )$ is exact. Therefore $P_v$ is projective. Now by Theorem \ref{category representations vs modules} $\Hom_{R\mathcal{C}}(\Phi(P_v), - ) = \Hom_{(R-mod)^\mathcal{C}}(P_v,- ) \circ \Psi$ and $\Psi$ is exact.  Hence  $\Phi(P_v)$ is also projective.
\end{proof}
% The following proposition shows that equivalences preserve linearized representable functors:
% \david{This seems totally obvious no?}
% \begin{prop}
% \label{equivalence preserves linearized representable functors}
%     Let         
%     \begin{center}
% 	       \begin{tikzcd}
%                     \mathcal{C}\arrow[bend left]{r}{F}
% 		          & \mathcal{D}\arrow[bend left]{l}{G}
% 	       \end{tikzcd}     
%         \end{center} 
%     be an equivalence of categories. Then, $P_v\circ F=P_{G(v)}$ for all $v\in \Ob(\mathcal{D})$.     
% \end{prop}
% \begin{proof}
% Let $v\in \Ob(\mathcal{D})$ and note that for all $x\in \Ob(\mathcal{C})$, $(P_v\circ F)(x)=P_v(F(x))$. Now note that 
% $$\Hom_{\mathcal{D}}(v,F(x))\cong \Hom_{\mathcal{C}}(G(v),G(F(x)))\cong \Hom_{\mathcal{C}}(G(v),x)$$
% and so $(P_v\circ F)(x)=P_{G(v)}(x)$. 
% \end{proof}
We are also interested in simple representations. In general, they are a bit complicated. However assuming that $R$ is a field, we have the following class of simples:
 \begin{df}
        Let $k$ be a field. For each object $v\in \mathcal{C}$, the associated \newterm{simple functor} associated to $v$ is the functor
     \begin{align*}
       S_v:\mathcal{C} & \to k-mod \\
       w & \mapsto \begin{cases} k & \text{ if } w\cong v \\
       0 &  \text{ if } w \ncong v
       \end{cases}
     \end{align*} 
     On morphisms $S_v$ sends any isomorphism between $v$ and another object to the identity and it sends all the other morphisms to $0$. 
 \end{df}
The following proposition shows that $S_v$'s are indeed simple objects in both $(k-mod)^{\mathcal{C}}$ and $R\mathcal{C}-mod$:
\begin{prop}
 Let $v$ be an object of $\mathcal{C}$. Then, $S_v$ is simple in both $(k-mod)^{\mathcal{C}}$ and $R\mathcal{C}-mod$. 
\end{prop}
\begin{proof}
First $S_v$ is a simple object in the category of functors since it can not have any non-trivial sub-functor. 
To check $S_v$ is also simple in $k\mathcal{C}-mod$, suppose $M \subseteq S_v$ is a submodule with $0 \neq m \in M$.  Then, $0 \neq Id_x \cdot m \in M$ for some $x \cong v$.  Now, by rescaling $Id_x \cdot m$ we get all of $Id_x \cdot S_v$.  Then,  acting by isomorphisms and summing gives all elements of $S_v$.
% To check $S_v$ is also simple in $k\mathcal{C}-mod$, first note that $S_v$ is generated by $Id_x.S_v$'s where $x\in \Ob(\mathcal{C})$ (where $Id_x.S_v=k$ if $x\cong v$ and zero otherwise). Let $M$ be a non-zero submodule of $S_v$ and we will show $Id_x.S_v\subset M$ for all $x\in \Ob(\mathcal{C})$ which is obvious when $x\ncong v$. So, let $x\cong v$. Since $M\ne 0$, there exists $m\ne 0$ in $M$ and since $M$ is a submodule of $S$, $m\ne 0$ in $S_v$. So, there exists $y\in \Ob(\mathcal{C})$ with $Id_y.m\ne 0\in Id_y.S_v=k$ and hence $y\cong v\cong x$ as well. Now, pick an isomorphism $f:y\to x$ and note that since $Id_y.m\ne 0\in Id_y.S_v=k$, $f.(Id_y.m)\ne 0 \in Id_x.S_v=k$. So, $Id_x.S_v\subset M$. 
\end{proof}
\begin{ex}
    Let $\mathcal{C}$ be a category with $n$ objects $v_1,...,v_n$ where between any two objects there exists a unique morphism (See Example \ref{matrix algebras are category algebras}). Then, $S_{v_i}$ is isomorphic to the $M_{n\times n}(k)$ simple module which as a set has all the $n\times n$ matrices which only have non-zero entries in their $i$-th row and zero everywhere else. 
\end{ex}
 Since our notions are defined categorically, we have the following.
\begin{prop}
\label{equivalence preserves simples}
    Let         
    \begin{center}
	       \begin{tikzcd}
                    \mathcal{C}\arrow[bend left]{r}{F}
		          & \mathcal{D}\arrow[bend left]{l}{G}
	       \end{tikzcd}     
        \end{center} 
    be an equivalence of categories. This induces an equivalence 
        \begin{center}
	       \begin{tikzcd}
                  (R-mod)^{\mathcal{C}} \arrow[bend left]{r}{\circ G}
		          &  (R-mod)^{\mathcal{D}}\arrow[bend left]{l}{\circ F}
	       \end{tikzcd}     
                 \end{center} 
              which satisfies $S_v \circ G \cong S_{F(v)}$ and $P_v \circ G = P_{F(v)}$.

\end{prop}
\begin{proof}
This is immediate from the definitions.
\end{proof}
% \begin{proof}
%     Let $v\in \Ob(\mathcal{D})$ and note that for $x\in \Ob(\mathcal{C})$, if $x\cong G(v)$ we have $S_{G(v)}(x)=k$ and note that $x\cong G(v)$ implies that $F(x)\cong v$ and so, $(S_v\circ F)(x)=S_v(F(x))=k$. On the other hand if $x\ncong G(v)$, we have $S_{G(v)}(x)=0$ and note that $x\ncong G(v)$ implies that $F(x)\ncong v$ and so, $(S_v\circ F)(x)=S_v(F(x))=0$. 
% \end{proof}

\begin{prop}
\label{equivalences preserve ext groups of simples}
     Let         
    \begin{center}
	       \begin{tikzcd}
                    \mathcal{C}\arrow[bend left]{r}{F}
		          & \mathcal{D}\arrow[bend left]{l}{G}
	       \end{tikzcd}     
        \end{center} 
    be an equivalence of graded categories. Then, 
     $$\Ext^i(S_w,S_v)_n=\Ext^i(S_{G(w)},S_{G(v)})_n$$   
\end{prop}
\begin{proof}
Since our simples are objects of the functor category, by Theorem \ref{category representations vs modules} their Ext groups in the module category can be computed using their projective resolutions in the functor category (see Remark \ref{ext groups can be computed in the functor category}). Now since $\mathcal{C}$ and $\mathcal{D}$ are equivalent, $k-mod^{\mathcal{C}}$ and $k-mod^{\mathcal{D}}$ are also equivalent. Via this equivalence our simple objects map to simple objects and our projective objects map to projective objects (see Proposition \ref{equivalence preserves simples}). Moreover since equivalences are exact functors, they preserve resolutions. Hence
    $$\Ext^i(S_w,S_v)_n=\Ext^i(S_{G(w)},S_{G(v)})_n.$$    
\end{proof}
Koszulity requires that $(k\mathcal{C})_0$ is semisimple. 
 Hence, we assume our categories have no non-trivial automorphisms. Otherwise, the representation theory of all automorphism groups seems to enter the story non-trivially. This has the added benefit of allowing us to interpret the normalized standard resolution as a cellular resolution associated to topological spaces defined later on. We codify this in the following definition and proposition.
\begin{df}
    Let $\mathcal{C}$ be an $\mathbb{N}$-graded category. We say that $\mathcal{C}$  is \newterm{indiscretely based} if each connected component of $\mathcal{C}_0$ is an indiscrete category.
\end{df}

\begin{prop}
\label{kC_0 is semisimple}
    Let $\mathcal{C}$ be an indiscretely based category. Then, 
    $$(k\mathcal{C})_0=\bigoplus_{v\in \Ob(\mathcal{C})}S_v$$
    In particular $(k\mathcal{C})_0$ is semisimple. 
\end{prop}
\begin{proof}
    Define $\phi:(k\mathcal{C})_0\to \bigoplus_{v\in \Ob(\mathcal{C})}S_v$ to be the map which sends an isomorphism $f:v\to w$ to $1_k\in S_v.Id_w$. Note that since we do not have non-trivial automorphisms, if two objects are isomorphic, there exists a unique isomorphism between them. This allows us to define the inverse of $\phi$ as a map $\psi:\bigoplus_{v\in \Ob(\mathcal{C})}S_v\to (k\mathcal{C})_0$ which sends $1_k\in S_v.Id_w$ (when $v\cong w$) to the unique isomorphism from $v$ to $w$. 
\end{proof}
\begin{ex}
    Graded HPAs are indiscretely based and $k\mathcal C_0 = k ^{Q_0}$.  As a special case, graded posets can be realized as graded HPAs and so they are indiscretely based. 
\end{ex}
\begin{ex}
We continue \Cref{matrix algebras are category algebras} where $\mathcal{C}$ is a category with $n$ objects $v_1,...,v_n$ and there exists a unique morphism between any two objects. 
 In this case, $M_{n\times n}(k) \cong k\mathcal C$ and Proposition \ref{kC_0 is semisimple} provides the row decomposition of $M_{n\times n}(k)$. 
\end{ex}
\begin{cor}
\label{vanishing of ext groups based on simples}
  Let $\mathcal{C}$ be an indiscretely based category. Then, $\Ext^i((k\mathcal{C})_0,(k\mathcal{C})_0)_n=0$ if and only if $\Ext^i(S_w,S_v)_n=0$ for all $v,w\in \Ob(\mathcal{C})$. 
\end{cor}
\begin{proof}
  Note that by Proposition \ref{kC_0 is semisimple}, we have $(k\mathcal{C})_0$ is semisimple and
    $$(k\mathcal{C})_0=\bigoplus_{v\in \Ob(\mathcal{C})}S_v$$
    So, 
    $$\Ext^i((k\mathcal{C})_0,(k\mathcal{C})_0)_{-n}=\bigoplus_{v,w\in \Ob(\mathcal{C})}\Ext^i(S_w,S_v)_{-n}$$
    So $\Ext^i((k\mathcal{C})_0,(k\mathcal{C})_0)_n=0$ if and only if $\Ext^i(S_w,S_v)_n=0$ for all $v,w\in \Ob(\mathcal{C})$.
\end{proof}
% \david{Is this corollary really needed?  Maybe just cite 4.23 and 4.17 where you cite this.  It's not an interesting statement and just follows from what's already been done.}
% \begin{cor}
% \label{equivalences preserve vanishing of ext groups}
% Let         
%     \begin{center}
% 	       \begin{tikzcd}
%                     \mathcal{C}\arrow[bend left]{r}{F}
% 		          & \mathcal{D}\arrow[bend left]{l}{G}
% 	       \end{tikzcd}     
%         \end{center} 
%     be an equivalence of graded categories. Then, $\ext^i((k\mathcal{C})_0,(k\mathcal{C})_0(n))=0$ if and only if $\ext^i((k\mathcal{D})_0,(k\mathcal{D})_0(n))=0$ for all $i$.  
% \end{cor}
% \begin{proof}
%  First note that by Proposition \ref{equivalences preserve ext groups of simples} we have
%     $$\ext^i(S_w,S_v(-n))=\ext^i(S_{G(w)},S_{G(v)}(-n))$$
%     Then use Corollary \ref{vanishing of ext groups based on simples} and that equivalences are essentially surjective. 
% \end{proof}
The diagonal bimodule for $R\mathcal C$ has a projective resolution  defined as follows.
  Let $\mathcal{C}$ be a category. 
 Let $\widetilde{\mathcal{C}}(x_1,x_2)$ be the free $R$-module generated by $\Hom_{\mathcal{C}}(x_1,x_2)$ when $x_1\ne x_2\in \Ob(\mathcal{C})$ and $\widetilde{\mathcal{C}}(x,x)$ be the free $R$-module generated by $\Hom_{\mathcal{C}}(x,x)\setminus \{Id_x\}$ for $x\in \Mor(\mathcal{C})$. For each $n\geq 0$ set 
$$\widetilde{C}_n=\bigoplus_{(x_1,...,x_{n+1})}P_{x_{n+1}}\otimes_R \widetilde{\mathcal{C}}(x_{n+1},x_n)\otimes_R...\otimes_R \widetilde{\mathcal{C}}(x_1,x_2)\otimes_R P_{x_1}^{op}$$
 with differentials 
 $$d_n(f_0\otimes f_2\otimes...\otimes f_{n+1})=\Sigma_{i=0}^n(-1)^if_0\otimes...\otimes f_i\circ f_{i+1}\otimes ...\otimes f_{n+1}$$
 Then, $(\widetilde{C}_{\bullet},d_{\bullet})$ is a projective resolution of the \newterm{linearized Yoneda functor}
 $$y_{\mathcal{C}}:\mathcal{C}^{op}\times \mathcal{C}\to R-mod$$
 $$(x,y)\mapsto \bigoplus_{\Hom_{\mathcal{C}}(x,y)}R$$
 called the \newterm{normalized standard resolution} \cite[\S17]{MITCHELL19721}. Using our functor $\Phi$ from Theorem \ref{category representations vs modules}, the linearized Yoneda maps to $R\mathcal{C}$ as a $R\mathcal{C}$-bimodule. 
 Hence, the normalized standard resolution becomes a resolution of $R\mathcal{C}$ as a $R\mathcal{C}$-bimodule. 
 We do not distinguish the two and call this the  normalized standard resolution as well. Note, this resolution may have infinitely many terms (See Example \ref{graded projective resolution of k[x]}).

\subsection{Semi-simplicial sets}
Next, we would like to view the normalized standard resolution as a chain complex obtained from a semi-simplicial set. We recall the following background material from \cite{w}:
\begin{df}[\protect{\cite[Definition 8.1.9, p. 258, Exercise 8.1.6, p. 259]{w}}]
Denote by $\Delta_s$ the category whose objects are finite ordered sets $[n]=\{0<1<2<...<n\}$ and whose morphisms are nondecreasing injective maps. A \newterm{semi-simplicial set} is a functor $K:\Delta_s \to \Sets$.    Equivalently, a semi-simplicial set is a sequence of objects $K_0,K_1,...$ (of $\mathcal{C}$) together with face operators $\partial_i:K_n\to K_{n-1}$ ($i=0,...,n$) such that if $i<j$, then $\partial_i\partial_j=\partial_{j-1}\partial_i$. 
\end{df}
\begin{ex}[\protect{\cite[p. 258]{w}}]
\label{ex: abstract simplicial complexes}
    Let $\Delta$ be a (combinatorial) \newterm{ordered abstract simplicial complex}, i.e. a collection of nonempty subsets (faces) of a finite ordered set of vertices $V$ such that if $G\in \Delta$ and $F\subset G$ then $F\in \Delta$. Set $K_n$ to be the set of all faces of $\Delta$ which have $n+1$ elements. Then $K$ is a semi-simplicial set.
\end{ex}
% There is another interpretation of semi-simplicial sets which is as follows: 
% \begin{prop}[\protect{\cite[Exercise 8.1.6, p. 259]{w}}]
% Any semi-simplicial object is the same thing as a sequence of objects $K_0,K_1,...$ (of $\mathcal{C}$) together with face operators $\partial_i:K_n\to K_{n-1}$ ($i=0,...,n$) such that if $i<j$, then $\partial_i\partial_j=\partial_{j-1}\partial_i$.     
% \end{prop}
Now, let $K$ be a semi-simplicial set. From $K$ we can obtain a topological space called the \newterm{geometric realization} defined as follows. 
\begin{df}[\protect{\cite[Geometric realization 8.1.6, p. 257]{w}}]
   For each $n\geq 0$ regard $K_n$ as a topological space with discrete topology and let $\Delta_n$ be the standard $n$-simplex.  The \newterm{geometric realization} is the quotient space $\coprod(K_n\times \Delta^n)/\sim$ where $(x,s)\in K_m\times \Delta^m$ and $(y,t)\in K_n\times \Delta^n$ are equivalent if there is a map $\alpha:[m]\to [n]$ in $\Delta_s$ such that $K(\alpha)(y)=x$ and $\alpha_*(s)=t$ and where $\alpha_*:\Delta^m\to \Delta^n$ is obtained from the linear extension of $\alpha$ on $\Delta^m$ to $\Delta^n$. That is 
$$(K(\alpha)(y),s)\sim (y,\alpha_*(t)).$$ 
\end{df}

\begin{ex}
Let $\mathcal{C}$ be a category. The \newterm{nerve}  $N\mathcal{C}$ of $\mathcal{C}$ is a (semi)-simplicial set with $(N\mathcal{C})_n=\{(f_1, ... , f_n):f_i\in \Mor(\mathcal{C}), \ \text{$f_1\circ ... \circ f_n\in \Mor(\mathcal{C})$ }\}$ with face operators 
    $$\partial_i:N(\mathcal{C})_n\to N(\mathcal{C})_{n-1}$$
    $$(f_1, ... , f_n)\mapsto (f_1, ..., (f_i\circ f_{i+1}), ...,f_n)$$
    for $i\in \{0,...,n\}$. The geometric realization of the nerve is called the \newterm{classifying space} of $\mathcal{C}$.
\end{ex}
\begin{df}
 Let $\mathcal{C}$ be a category and $p\in \Mor(\mathcal{C})$. We say that $p=f_0\circ  ... \circ f_{n+1}$ ($n\geq 0$) is a nontrivial factorization of $p$ if $f_i$ is not an isomorphism for all $0\leq i\leq n+1$. 
\end{df}
\begin{df}
\label{reduced nerve}
 Let $\mathcal{C}$ be an indiscretely based category. We define the \newterm{reduced nerve} of $\mathcal{C}$ to be the semi-simplicial set $\Bar{N}(\mathcal{C})_{\bullet}$ with 
 $$\Bar{N}(\mathcal{C})_n=\{(f_1, ... , f_n): (f_1,...,f_n)\in (N\mathcal{C})_n \ \text{and $f_i$ is not an isomorphism for all $i$ }\}/\sim$$
 where $(f_0, ... , f_{n+1})\sim (f'_0, ..., f'_{n+1})$ if for all $i$ there exists an isomorphism $h_i$ such that $f'_i=f_i\circ h_i$ and $f'_{i+1}=h_i^{-1}\circ f_{i+1}$ with face operators 
    $$\partial_i:\Bar{N}(\mathcal{C})_n\to \Bar{N}(\mathcal{C})_{n-1}$$
 We denote the geometric realization of $\Bar{N}(\mathcal{C})$ by $B\mathcal{C}$. 
\end{df}
The most important simplicial set for our purposes is defined as follows.
\begin{df}
\label{Factorization space}
    Let $\mathcal{C}$ be an indiscretely based category and $p\in \Mor(\mathcal{C})$. For each $n$ define 
    $$\mathcal{C}(p)_n=\{(f_0, ..., f_{n+1}):\text{$f_0\circ ...\circ f_{n+1}$ is a nontrivial factorization of $p$}\}/\sim
    $$
    where  $$(f_0, ... , f_{n+1})\sim (f'_0, ..., f'_{n+1})$$ if for all $i$ there exists an isomorphism $h_i$ such that $f'_i=f_i\circ h_i$ and $f'_{i+1}=h_i^{-1}\circ f_{i+1}$ with face operators 
    $$\partial_i:\mathcal{C}(p)_n\to \mathcal{C}(p)_{n-1}$$
    $$(f_0, ..., f_{n+1})\mapsto (f_0, ..., (f_i\circ f_{i+1}), ..., f_{n+1})$$
    for $i\in \{0,...,n\}$. Then, $\mathcal{C}(p)$ is a semi-simplicial set. We call the geometric realization of $\mathcal{C}(p)$ the \newterm{factorization space} of $p$ and we denote it by $B\mathcal{C}(p)$.
\end{df}
\begin{rmk}
\label{C and Cop}
  In Definition \ref{Factorization space} note that $\mathcal{C}(p)$ and $\mathcal{C}^{op}(p^{op})$ are isomorphic via the natural isomorphism 
    $$\phi_n:\mathcal{C}(p)_n\to \mathcal{C}^{op}(p^{op})_n$$
    $$(f_0, ..., f_{n+1})\mapsto (f_{n+1}^{op}, ..., f_0^{op})$$
    for each $n\geq 0$. This will later imply (by 
  Theorem \ref{Theorem}) that the Koszulity of $k\mathcal{C}$ and $k\mathcal{C}^{op}$ are equivalent.    
\end{rmk}

\begin{ex}
\label{hpa interval}
    Let $A$ be a homotopy path algebra and $\mathcal{C}$ be its quiver with relations regarded as a category. Let $K(e_{t(p)},p)$ denote the geometric realization of the path interval $K(e_{t(p)},p)$ for some path $p$. Then, $K(e_{t(p)},p)$ and $B\mathcal{C}^{op}(p^{op})$ are homeomorphic. Indeed, let $X$ be the semi-simplicial complex of $(e_{t(p)},p)$ obtained by setting $X_n$ to be the set of $n$-chains in $(e_{t(p)},p)$. For each $n$ define 
    $$\phi_n:X_n\to \mathcal{C}^{op}(p)_n$$
    $$q_1<q_2<...<q_{n+1}\mapsto ((q_1/e_{t(p)}), (q_2/q_1), ..., (q_{n+1}/q_n),(p/q_{n+1}))$$
    Then, $\phi_{\bullet}$ is an isomorphism of semi-simplicial sets. By Remark \ref{C and Cop} this implies that $K(e_{t(p)},p)$ and $B\mathcal{C}(p)$ are also homeomorphic. 
\end{ex}
\begin{rmk}
    In Example \ref{hpa interval}, the cancellative property of morphisms is necessary to define  $q_i/q_{i-1}$ and hence  $\phi$.% since  $q_i/q_{i-1}$ may not be well-defined. %Also, note that without the cancellative property,  $(e_{t(p)},p)$ are still posets by using the order $q<p$ if $p=qr$ for some $r$ and $l(p)>l(q)$. So, indeed the main difference for the non-HPA cases is that the division of morphisms is not well-defined and so we need to define $K(e_{t(p)},p)$ more generally, otherwise $K(e_{t(p)},p)$ will not capture enough information that we need to use later to compute $\ext$ groups. 
\end{rmk}
\begin{ex}
\label{ex: open intervals as BC(p)}
    Let $P$ be a poset and $a\leq_p b$ then 
    $$\phi:(e_a,p)\to (a,b)$$
    $$q\mapsto h(q)$$
    is an isomorphism of posets. So in Example \ref{hpa interval} in the special case when $\mathcal{C}$ is a poset $P$, the order complex of $(a,b)$ is homeomorphic to $BP(p)$. 
\end{ex}

\subsection{Topological description of Ext groups}

In this section we provide our topological computation of Ext groups.  We begin the setting of any indiscretely based category, then specialize to homotopy path algebras where the description is even more concrete topologically.

\subsubsection{General topological description}

The purpose of the following proposition is to provide a projective resolution of $R\mathcal{C}$ as an $R\mathcal{C}$-bimodule based on $B\mathcal{C}$ (see Definition \ref{reduced nerve}) using the normalized resolution of $R(\sk(\mathcal{C}))$. This generalizes \cite[Corollary 6.7]{dj}:
\begin{prop}
\label{projective resolution of diagonal}
    Let $\mathcal{C}$ be an indiscretely based category with grading $l:\mathcal C\to \mathbb N$. Define 
    $$C_k:=\bigoplus_{\eta_k\in \Cell_k(B\mathcal{C})}P_{h(\eta_k)}\boxtimes P_{t(\eta_k)}^{op}(-l(\eta_k))$$
with 
$$d_k=\bigoplus_{\eta_k\in \Cell_k(B\mathcal{C})}d_{\eta_k}:C_k\to C_{k-1}$$
where 
$$d_{\eta_k}=\Sigma_{i=0}^k(-1)^i\partial_{i,\eta_k}$$
and 
$$\partial_{i,\eta_k}(1\otimes [\eta_k]\otimes 1):=\begin{cases}
    f_1\otimes[(f_2, ..., f_k)]\otimes 1, & i=0\\
    1\otimes [(f_1, ..., (f_i\circ f_{i+1}), ..., f_k)]\otimes 1, & 0<i<k\\     1\otimes [(f_1, ..., f_{k-1})]\otimes f_k, & i=k  
\end{cases}$$
Then, $(C_{\bullet},d_{\bullet})$ is a projective resolution of $R\mathcal{C}$ as a graded $R\mathcal{C}$-bimodule. 
\end{prop}
\begin{proof}
 Consider a skeletal category $\sk(\mathcal{C})$ and an equivalence  
    \begin{center}
	       \begin{tikzcd}
                    \sk(\mathcal{C})\arrow[bend left]{r}{F}
		          & \mathcal{C}\arrow[bend left]{l}{G}
	       \end{tikzcd}     
        \end{center} 
The resolution written above is just the normalized resolution of $R\sk(\mathcal{C})$ under the identification
$$P_{x_{k+1}}\otimes_R \widetilde{\mathcal{C}}(x_{k+1},x_k)\otimes_R...\otimes_R \widetilde{\mathcal{C}}(x_1,x_2)\otimes_R P_{x_1}^{op}=\bigoplus_{\eta_k\in \Cell_k(B\mathcal{C}),t(\eta_k)=x_1,h(\eta_k)=x_{k+1}}P_{h(\eta_k)}\boxtimes P_{t(\eta_k)}^{op}$$
% So,
%     \begin{center}
%         \begin{align*}
%         \widetilde{C}_k&=\bigoplus_{(x_1,...,x_{k+1})}P_{x_{k+1}}\otimes_R \widetilde{\mathcal{C}}(x_{k+1},x_k)\otimes_R...\otimes_R \widetilde{\mathcal{C}}(x_1,x_2)\otimes_R P_{x_1}^{op}\\
%             &=\bigoplus_{\eta_k\in \Cell_k(B\mathcal{C})}P_{h(\eta_k)}\boxtimes P_{t(\eta_k)}^{op}\\
%             &=C_k
%         \end{align*}
%     \end{center}
% One can easily check that the differentials match as well. 
Then, the equivalence of Proposition~\ref{equivalence preserves simples} is exact sending this projective resolution of $R\sk(\mathcal C)$ to the described projective resolution of $R\mathcal C$.  The grading shift by $-l(\eta_k)$ makes the differential degree preserving.
% sends 
% So it remains to show that it is also a projective resolution of $R\mathcal{C}$ as a $R\mathcal{C}$-bimodule. \david{I'm not sure what you are doing in this proof.  Haven't you already established these kinds of properties abstractly?} To do this first note that by Theorem \ref{category representations vs modules} it is enough to check it in the functor category. Let          
%     \begin{center}
% 	       \begin{tikzcd}
%                     \sk(\mathcal{C})\arrow[bend left]{r}{F}
% 		          & \mathcal{C}\arrow[bend left]{l}{G}
% 	       \end{tikzcd}     
%         \end{center} 
%     be our equivalence of categories. Then note that we have the following equivalence of categories 
%             \begin{center}
% 	       \begin{tikzcd}
%                     R(\sk(\mathcal{C})^{op}\times \sk(\mathcal{C}))-mod\arrow[bend left]{r}{\alpha=-\circ (G^{op}\times G)}
% 		          & R(\mathcal{C}^{op}\times \mathcal{C})-mod\arrow[bend left]{l}{\beta=-\circ (F^{op}\times F)}
% 	       \end{tikzcd}     
%         \end{center} 
%     with $\alpha(y_{\sk(\mathcal{C})})\cong y_{\mathcal{C}}$. Then by Proposition \ref{equivalence preserves linearized representable functors}, we have 
% $$\alpha(\bigoplus_{\eta_k\in \Cell_k(B\mathcal{C})}P_{h(\eta_k)}\boxtimes P_{t(\eta_k)}^{op})=\bigoplus_{\eta_k\in \Cell_k(B\mathcal{C})}P_{G(h(\eta_k))}\boxtimes P_{G(t(\eta_k))}^{op}$$
% Now since equivalences are exact functors, applying $\alpha$ will provide the same resolution for $y_{\mathcal{C}}$. 
\end{proof}

\begin{ex}
\label{graded projective resolution of k[x]}
    Let $k$ be a field and $\mathcal{C}$ be a graded category with one object $v$ and  morphisms $\Hom_{\mathcal{C}}(v,v)=\{x^n:n\in \mathbb N\}$. Then, $k\mathcal{C}\cong k[x]$. By Proposition \ref{projective resolution of diagonal}, we have the following graded projective resolution of $k[x]$:
       $$...\xrightarrow{d_3} \begin{matrix}
       (P_v\boxtimes_{x^2}P_v^{op}(-2))^{\oplus m_{2,2}}\\
       \oplus\\
       (P_v\boxtimes_{x^3} P_v^{op}(-3))^{\oplus m_{2,3}}\\
       \oplus\\
       ...
    \end{matrix}\xrightarrow{d_2} \begin{matrix}
       (P_v\boxtimes_{x}P_v^{op}(-1))^{\oplus m_{1,1}}\\
       \oplus\\
       (P_v\boxtimes_{x^2} P_v^{op}(-2))^{\oplus m_{1,2}}\\
       \oplus\\
       ...
    \end{matrix}\xrightarrow{d_1}
        P_v\boxtimes P_v^{op}
\xrightarrow{d_0} k[x]\to 0$$ 
where $m_{r,s}$ is the number of $r$-dimensional cells of length $s$, i.e. $m_{r,s}=\binom{s-1}{r-1}$.
\end{ex}
\begin{ex}
\label{ex: resolution of Beilinson quiver}
Let $k$ be a field and $\mathcal{C}$ be the Beilinson quiver regraded as a category, i.e. the following quiver with relations $\Bar{x}_i\circ x_j=\Bar{x}_j\circ x_i$:
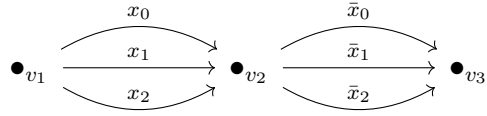
    \begin{center}
	       \begin{tikzcd}
                    \bullet_{v_1}\arrow[bend left]{rr}{x_0}\arrow{rr}{x_1}\arrow[bend right]{rr}{x_2}
		          && \bullet_{v_2}\arrow[bend left]{rr}{\Bar{x}_0}\arrow{rr}{\Bar{x}_1}\arrow[bend right]{rr}{\Bar{x}_2}
                  &&\bullet_{v_3}
	       \end{tikzcd}     
        \end{center} 
By Proposition \ref{projective resolution of diagonal}, we have the following graded projective resolution of $k\mathcal{C}$:
{\small$$0\to
     \begin{matrix}
        P_{v_3}\boxtimes_{\eta_1}P_{v_1}^{op}(-2)\\
        \oplus \\
        ...\\
        \oplus\\
        P_{v_3}\boxtimes_{\eta_9}P_{v_1}^{op}(-2)\end{matrix}
        \xrightarrow{d_2}  \begin{matrix}
       P_{v_3}\boxtimes_{\Bar{x}_0\circ x_0}P_{v_1}^{op}(-2)\\
       \oplus\\
       P_{v_3}\boxtimes_{\Bar{x}_1\circ x_0}P_{v_1}^{op}(-2)\\
       \oplus\\
       P_{v_3}\boxtimes_{\Bar{x}_2\circ x_0}P_{v_1}^{op}(-2)\\
       \oplus\\
       P_{v_3}\boxtimes_{\Bar{x}_2\circ x_2}P_{v_1}^{op}(-2)\\
       \oplus\\
       P_{v_3}\boxtimes_{\Bar{x}_2\circ x_3}P_{v_1}^{op}(-2)\\
       \oplus\\
       P_{v_3}\boxtimes_{\Bar{x}_3\circ x_3}P_{v_1}^{op}(-2)
    \end{matrix}\\\oplus\\\begin{matrix}
       P_{v_2}\boxtimes_{x_0}P_{v_1}^{op}(-1)\\
       \oplus\\
       ...\\
       \oplus\\
       P_{v_3}\boxtimes_{\Bar{x}_2}P_{v_2}^{op}(-1)\end{matrix}
    \xrightarrow{d_1}\begin{matrix}
       P_{v_1}\boxtimes P_{v_1}^{op}\\
       \oplus\\
       P_{v_2}\boxtimes_{}P_{v_2}^{op}\\
       \oplus\\
       P_{v_3}\boxtimes_{}P_{v_3}^{op}
    \end{matrix}\xrightarrow{d_0} k\mathcal{C}\to 0$$}
    where $\eta_1$ is the $2$-cell corresponding to $(\Bar{x}_0, x_0)$, ... and $\eta_9$ is the $2$-cell corresponding to $(\Bar{x}_2, x_2)$.
\end{ex}
\begin{comment}
\begin{ex}
\label{ex: resolution of the A_2 quiver}
    Let $k$ be a field and $\mathcal{C}$ be the following indiscretely based category
    $$\bullet_{v_1}\xrightarrow{f}\bullet_{v_2}\xrightarrow{g}\bullet_{v_3}$$
    By Proposition \ref{projective resolution of diagonal}, we have the following graded projective resolution of $k\mathcal{C}$:
    $$0\to P_{v_3}\boxtimes_{\eta}P_{v_1}^{op}(-2)\xrightarrow{d_2}  \begin{matrix}
       P_{v_2}\boxtimes_{f}P_{v_1}^{op}(-1)\\
       \oplus\\
       P_{v_3}\boxtimes_{g}P_{v_2}^{op}(-1)\\
       \oplus\\
       P_{v_3}\boxtimes_{g\circ f}P_{v_1}^{op}(-2)
    \end{matrix}\xrightarrow{d_1}\begin{matrix}
       P_{v_1}\boxtimes P_{v_1}^{op}\\
       \oplus\\
       P_{v_2}\boxtimes_{}P_{v_2}^{op}\\
       \oplus\\
       P_{v_3}\boxtimes_{}P_{v_3}^{op}
    \end{matrix}\xrightarrow{d_0} k\mathcal{C}\to 0$$
    where $\eta$ is the $2$-cell corresponding to $g\circ f$.
\end{ex}
\end{comment}
To provide a topological description of Ext groups between quiver representations, we also need the following lemma: 
\begin{lemma}
\label{Cells}
Let $\mathcal{C}$ be an indiscretely based category. There is a bijection of sets 
$$\coprod_{p\ne Id}\Cell(B\mathcal{C}(p))\to \Cell_{\geq 2}(B\mathcal{C})$$
$$\alpha\mapsto \alpha$$
\end{lemma}
\begin{proof}
We follow the proof of \cite[Lemma 6.20]{dj} but in our new setting. Here first note that any $n$-cell of $B\mathcal{C}(p)$ is of the form $[(f_0, ..., f_{n+1})]$. But on the other hand, $[(f_0, ..., f_{n+1})]$ can be also seen as a cell of $B\mathcal{C}$ and this is what we mean by our map. It is surjective since any cell of $B\mathcal{C}$ of dimension at least $2$ is of the form $[(f_0, ..., f_{n+1})]$ where $n\geq 0$. It is trivially injective. 
\end{proof} 
We are now ready to state and prove one of our main propositions:
\begin{prop}
\label{Ext groups}
    Let $k$ be a field and $\mathcal{C}$ be an indiscretely based category with grading $l:\mathcal C\to \mathbb N$. Then, 
$$\Ext_{k\mathcal{C}}^i(S_w,S_v)_{-n}=\bigoplus_{\substack{p|l(p)=n,\\t(p)=v,h(p)=w}}\Tilde{H}^{i-2}(B\mathcal{C}(p))$$
\end{prop}
\begin{proof}
%    We follow the proof of \cite[Lemma 6.23]{dj}.
    By Proposition \ref{projective resolution of diagonal}, $S_w \otimes C_{\bullet}$ is a projective resolution of $S_w$.  Therefore we can compute 
    \begin{center}
\begin{align*}
    \Ext^i(S_w,S_v)_{-n}
    &=H^i(\Hom(S_w \otimes C_{\bullet},S_v)_{-n})\\
\end{align*}
\end{center}
\begin{comment}

Observe that for a summand $P_s\boxtimes_k P_t^{op}$ of $C_\bullet$ we have
$$\Hom_{R\mathcal{C}}(S_w \otimes P_s\boxtimes_t P_v^{op}(-n), S_v)_{-n}=\begin{cases}
        k, & w=s, v=t, n=0\\
        0, & \text{else}
    \end{cases} $$
    The terms in $(\Hom(S_w \otimes C_{\bullet},S_v)$ reduce dramatically as   
    $$S_u\otimes_{k\mathcal{C}} (P_w\boxtimes_k P_v^{op}) =\begin{cases}
        P_v^{op}, & w=u\\
        0, & w\ne u.
    \end{cases} $$ and
    Hence $S_w=R_{\bullet}^w$ where  $R_k^w=\bigoplus_{\eta_k\in \Cell_k(B\mathcal{C}),h(\eta_k)=w}P_{t(\eta_k)}^{op}$. 
\end{comment}
Since 
$$S_w\otimes_{k\mathcal{C}} (P_u\boxtimes_k P_v^{op}) =\begin{cases}
        P_v^{op}, & w=u\\
        0, & w\ne u.
    \end{cases} $$
for each $r$, we have
$$(\Hom(S_w \otimes C_{\bullet},S_v)_{-n})_r=\bigoplus_{\substack{\eta_r\in \Cell_r(B\mathcal{C}),\\l(\eta_r)=n,h(\eta_r)=w}}\Hom(P_{t(\eta_r)}^{op}(-l(\eta_r)),S_v)_{-n}$$
Since  $$\Hom_{R\mathcal{C}}(P_u^{op}(-m),S_v)_{-n}=\begin{cases}
        k, & v=u,n=m\\
        0, & \text{else}
    \end{cases} $$
we get
$$(\Hom(S_w \otimes C_{\bullet},S_v)_{-n})_r=\bigoplus_{\substack{\eta_r\in \Cell_r(B\mathcal{C}),\\l(\eta_r)=n, t(\eta_r)=v,h(\eta_r)=w}}k$$
So
$$\Hom(S_w \otimes C_{\bullet},S_v)_{-n}=\bigoplus_{\substack{p|l(p)=n,\\t(p)=v,h(p)=w}}S_{\bullet}^p$$
where $S_{\bullet}^p$ is the summand consisting of all cells which factor a fixed path $p$. By Lemma \ref{Cells}, the complex $S^p_\bullet$ is cellular with respect to $B\mathcal C(p)$ i.e. the terms are indexed by the cells in $B\mathcal C(p)$.  To see that the differentials 
\[
d_k := \sum_{0 \leq i \leq k-1, \eta_k\in\Cell_k(B\mathcal{C}(p))}  (-1)^i \partial_{i,\eta_k}:S^p_k\to S^p_{k-1}
\]
of $S^p_\bullet$ agree with the CW cohomology differential, observe that $\partial_{0,\eta_k}$ vanishes after tensoring with $S_w$ and $\partial_{k,\eta_k}$ vanishes after applying $\Hom(-,S_v)$. The remaining terms $\partial_{i, \eta_k}$ have coefficients $\pm 1$, which agree with the usual CW differential
\[
d_k^{CW} := \sum_{1 \leq i \leq k-1, \eta_k\in\Cell_k(B\mathcal{C}(p))} (-1)^i \partial_{i,\eta_k}:S^p_k\to S^p_{k-1}.
\]
\end{proof}
\begin{ex}
\label{polynomial ring in one variable}
Let $k$ be a field and $\mathcal{C}$ be a graded category with one object $v$ and all powers of a single endomorphism i.e. $k\mathcal C = k[x]$. Then, by tensoring the graded projective resolution of $k[x]$ from Example \ref{graded projective resolution of k[x]} by $S_v$ from the left, we get the following resolution of $k$
       $$...\xrightarrow{d_3} \begin{matrix}
       (P_v^{op}(-2))^{\oplus m_{2,2}}\\
       \oplus\\
       (P_v^{op}(-3))^{\oplus m_{2,3}}\\
       \oplus\\
       ...
    \end{matrix}\xrightarrow{d_2} \begin{matrix}
       (P_v^{op}(-1))^{\oplus m_{1,1}}\\
       \oplus\\
       (P_v^{op}(-2))^{\oplus m_{1,2}}\\
       \oplus\\
       ...
    \end{matrix}\xrightarrow{d_1}
        P_v^{op}
\xrightarrow{d_0} k\to 0$$ 
where $m_{r,s}$ is the number of $r$-dimensional cells of length $s$, i.e. $m_{r,s}=\binom{s-1}{r-1}$. As $P_v\cong k[x]$ and $(P_v)^{op}\cong k[x]$,  this translates to 
       $$...\to\begin{matrix}
       k[x](-r)\\
       \oplus\\
       (k[x](-r-1))^{\oplus\binom{r}{r-1}}\\
       \oplus\\
       ...\\
       \oplus\\
       (k[x](-s))^{\oplus\binom{s-1}{r-1}}\\
       \oplus\\
       ...
    \end{matrix}\to...\to \begin{matrix}
       k[x](-2)\\
       \oplus\\
       (k[x](-3))^{\oplus\binom{2}{1}}\\
       \oplus\\
       ...
    \end{matrix}\to \begin{matrix}
       k[x](-1)\\
       \oplus\\
       k[x](-2)\\
       \oplus\\
       ...
    \end{matrix}\to
        k[x]
\to k\to 0$$ 
Taking $\Hom(-,S_v)$ ($n\geq 1$) from this resolution gives
$$0\to k\to k\to \begin{matrix}
       k(-1)\\
       \oplus\\
       k(-2)\\
       \oplus\\
       ...
    \end{matrix}\to \begin{matrix}
       k(-2)\\
       \oplus\\
       (k(-3))^{\oplus\binom{2}{1}}\\
       \oplus\\
       ...
    \end{matrix}\to ...\to \begin{matrix}
       k(-r)\\
       \oplus\\
       (k(-r-1))^{\oplus\binom{r}{r-1}}\\
       \oplus\\
       ...\\
       \oplus\\
       (k(-s))^{\oplus\binom{s-1}{r-1}}\\
       \oplus\\
       ...
    \end{matrix}\to ...$$
Hence $\RHom(S_v,S_v)_{-n}$ is as follows
$$0\to k^{\binom{n-1}{0}}\to ...\to k^{\binom{n-1}{i-1}}\to ...\to k^{\binom{n-1}{n-1}}\to 0.$$
For $n \geq 2$, this is nothing more than the Koszul complex on the vector $(1, ..., 1) \in k^{n-1}$ or equivalently the reduced homology of a $n-2$-simplex which is of course acyclic.  For $n=0,1$ we get the complex $k$.
\end{ex}
\begin{ex}
Let $k$ be a field and $\mathcal{C}$ be the Beilinson quiver regraded as a category, i.e. the following quiver with relations $\Bar{x}_i\circ x_j=\Bar{x}_j\circ x_i$:
    \begin{center}
	       \begin{tikzcd}
                    \bullet_{v_1}\arrow[bend left]{rr}{x_0}\arrow{rr}{x_1}\arrow[bend right]{rr}{x_2}
		          && \bullet_{v_2}\arrow[bend left]{rr}{\Bar{x}_0}\arrow{rr}{\Bar{x}_1}\arrow[bend right]{rr}{\Bar{x}_2}
                  &&\bullet_{v_3}
	       \end{tikzcd}     
        \end{center} 
By tensoring the graded projective resolution of $k\mathcal{C}$ from Example \ref{ex: resolution of Beilinson quiver} by $S_{v_3}$ from the left, we get the following graded projective resolution of $S_{v_3}$:
$$0\to P_{v_1}^{op}(-2)^{\oplus 9}\to P_{v_1}^{op}(-2)^{\oplus 6}\oplus  P_{v_2}^{op}(-1)^{\oplus 3}\to
       P_{v_3}^{op}\to S_{v_3} \to 0$$
Now taking $\Hom(-,S_{v_1})$ from this resolution gives us the following complex:
    $$0\to 0\to k(-2)^{\oplus 6}\to k(-2)^{\oplus 9}\to 0$$
    Hence $\RHom(S_{v_3},S_{v_1})_{-2}$ is the complex 
$$0\to 0\to k^{\oplus 6}\to k^{\oplus 9}\to 0$$
    Therefore $\Ext^i(S_{v_3},S_{v_1})_{-2}$ can be computed by computing the cohomologies of this complex. So, the only non-zero one is 
    $$\Ext^2(S_{v_3},S_{v_1})_{-2}=k^3$$
    Furthermore we have 
    \begin{align*}
    \bigoplus_{\substack{p|l(p)=2,\\t(p)=v_1,h(p)=v_2}}\widetilde{H}^0(B\mathcal{C}(p))& =\widetilde{H}^0(B\mathcal{C}(\Bar{x}_0\circ x_0))\oplus\widetilde{H}^0(B\mathcal{C}(\Bar{x}_1\circ x_1))\oplus\widetilde{H}^0(B\mathcal{C}(\Bar{x}_2\circ x_2))\\
    &\oplus\widetilde{H}^0(B\mathcal{C}(\Bar{x}_0\circ x_1))\oplus\widetilde{H}^0(B\mathcal{C}(\Bar{x}_0\circ x_2))\oplus\widetilde{H}^0(B\mathcal{C}(\Bar{x}_1\circ x_2))\\
    &= 0\oplus 0\oplus 0 \\
    & \oplus k\oplus k \oplus k\\
    &= k^3
\end{align*}
\end{ex}
\begin{comment}
\begin{ex}
    Let $k$ be a field and $\mathcal{C}$ be the following indiscretely based category
    $$\bullet_{v_1}\xrightarrow{f}\bullet_{v_2}\xrightarrow{g}\bullet_{v_3}$$
    By tensoring the graded projective resolution of $k\mathcal{C}$ from Example \ref{ex: resolution of the A_2 quiver} by $S_{v_3}$ from the left, we get the following graded projective resolution of $S_{v_3}$:
    $$0\to P_{v_1}^{op}(-2)\to \begin{matrix}
       P_{v_2}^{op}(-1)\\
       \oplus\\
       P_{v_1}^{op}(-2)
    \end{matrix}\to
       P_{v_3}^{op}\to S_{v_3} \to 0$$
    Now taking $\Hom(-,S_{v_1})$ from this resolution gives us the following complex:
    $$0\to 0\to k(-2)\to k(-2)\to 0$$
    Hence $\RHom(S_{v_3},S_{v_1})_{-2}$ is the complex 
    $$0\to 0\to k\to k\to 0$$
    Therefore $\Ext^i(S_{v_3},S_{v_1})_{-2}$ can be computed by computing the cohomologies of this complex which they are all zero.
\end{ex}    
\end{comment}

\begin{comment}
\subsubsection{Ext groups of homotopy path algebras}
In the case of homotopy path algebras, ext groups have a more concrete explanation. 
Homs between strata.    
\end{comment}

\subsection{Topological interpretation of Koszulity}
Here we discuss the Koszulity of category algebras. We then define almost discrete fibrations, a class of functors which preserve Koszulity. 

\subsubsection{Koszul Algebras}
We present here the main facts about Koszul algebras we will need and get some quick results.  For a much more complete treatment of Koszul algebras see \cite{bgs}.
\begin{df}[\protect{\cite[Definition 1.2.1, p. 475]{bgs}}]
A Koszul ring is a positively graded ring $A=\bigoplus_{j\geq 0}A_j$ such that 
\begin{enumerate}
    \item $A_0$ is semisimple,
    \item $A_0$ considered as a graded left $A$-module admits a graded projective resolution 
    $$...\to P^2\to P^1\to P^0\twoheadrightarrow A_0$$
    such that $P^i$ is generated by its component of degree $i$, $P^i=AP^i_i$. 
\end{enumerate}
\end{df}
The followings are quick corollaries of this definition:
\begin{cor}
\label{cor: Koszulity of kC and k(C^{op})}
    Let $\mathcal{C}$ be an indiscretely based category. Then, $k\mathcal{C}$ is Koszul if and only if $k(\mathcal{C}^{op})$ is Koszul. 
\end{cor}
\begin{proof}
We have
\begin{center}
\begin{align*}
\Ext^i(k(\mathcal{C})_0,k(\mathcal{C})_0)_{-n}&=\bigoplus_{v,w\in \Ob(\mathcal{C})}\Ext^i(S_w,S_v)_{-n}\\ \text{By Proposition \ref{Ext groups}}
&=\bigoplus_{v,w\in \Ob(\mathcal{C})}\bigoplus_{\substack{p|l(p)=n,\\t(p)=v,h(p)=w}}\Tilde{H}^{i-2}(B\mathcal{C}(p))\\\text{By Remark \ref{C and Cop}}
&=\bigoplus_{w,v\in \Ob(\mathcal{C}^{op})}\bigoplus_{\substack{p^{op}|l(p^{op})=n,\\t(p^{op})=w,h(p^{op})=v}}\Tilde{H}^{i-2}(B\mathcal{C}^{op}(p^{op}))\\\text{By Proposition \ref{Ext groups}}
&=\bigoplus_{w,v\in \Ob(\mathcal{C}^{op})}\Ext^i(S_v,S_w)_{-n}\\ 
&=\Ext^i(k(\mathcal{C}^{op})_0,k(\mathcal{C}^{op})_0)_{-n}
\end{align*}
\end{center}
\end{proof}
\begin{cor}
\label{tensor prodcut of koszul algebras}
    Let $k$ be a field and $A$ and $B$ be Koszul $k$-algebras. Then, $A\otimes_k B$ is Koszul.
\end{cor}
\begin{proof}
    Since $A$ and $B$ are Koszul, $A_0$ and $B_0$ have linear projective resolutions. The tensor product of these linear projective resolutions is a linear projective resolution of $A_0\otimes_k B_0=(A\otimes_k B)_0$.  Hence $A\otimes_k B$ is Koszul.
\end{proof}

\begin{cor}
    Let $\mathcal{C}$ and $\mathcal{D}$ be two indiscretely based categories such that their category algebras $k\mathcal{C}$ and $k\mathcal{D}$ are Koszul. Then, $k(\mathcal{C}\times \mathcal{D})$ is Koszul.
\end{cor}
\begin{proof}
    By Proposition \ref{prop: tensor product of cat algebras}
    $k\mathcal{C}\otimes_k k\mathcal{D}\cong k(\mathcal{C}\times \mathcal{D})$. By Corollary \ref{tensor prodcut of koszul algebras}, since $k\mathcal{C}$ and $k\mathcal{D}$ are Koszul, $k\mathcal{C}\otimes_k k\mathcal{D}$ is Koszul and so $k(\mathcal{C}\times \mathcal{D})$ is Koszul. 
\end{proof}
The topological interpretation of Ext groups from Proposition \ref{Ext groups} allows us to interpret Koszulity topologically using the following proposition.   
\begin{prop}[\protect{\cite[Proposition 2.1.3, p.\ 480]{bgs}}]
\label{Koszul}
Let $A=\bigoplus_{j\geq 0}A_j$ be a positively graded ring and suppose $A_0$ is semisimple. The following conditions are equivalent:
\begin{enumerate}
    \item $A$ is Koszul. 
    \item $\Ext^i_A(A_0,A_0)_{-n}=0$ for $i\ne n$. 
\end{enumerate}
\end{prop}
The following is a quick corollary of this proposition.
\begin{cor}
\label{Koszulity based on simples}
    Let $\mathcal{C}$ be an indiscretely based category. Then, $k\mathcal{C}$ is Koszul if and only if $\Ext^i(S_w,S_v)_{-n}=0$ for $i\ne n$ for all $v,w\in \Ob(\mathcal{C})$. 
\end{cor}
\begin{proof}
    First note that by Proposition \ref{kC_0 is semisimple}, we have $(k\mathcal{C})_0$ is semisimple. By Corollary \ref{vanishing of ext groups based on simples} $\Ext^i((k\mathcal{C})_0,(k\mathcal{C})_0)_n=0$ if and only if $\Ext^i(S_w,S_v)_n=0$ for all $v,w\in \Ob(\mathcal{C})$. Now use Proposition \ref{Koszul}. 
\end{proof}
Finally, we observe that Koszulity is a property of the category.
\begin{prop}
\label{equivalences preserve Koszulity}
Let         
    \begin{center}
	       \begin{tikzcd}
                    \mathcal{C}\arrow[bend left]{r}{F}
		          & \mathcal{D}\arrow[bend left]{l}{G}
	       \end{tikzcd}     
        \end{center} 
    be an equivalence of indiscretely based categories. Then, $k\mathcal{C}$ is Koszul if and only if $k\mathcal{D}$ is Koszul. In particular, $k\mathcal{C}$ is Koszul if and only if $k(\sk(\mathcal{C}))$ is Koszul. 
\end{prop}
\begin{proof}
By Proposition~\ref{equivalences preserve ext groups of simples} equivalences preserve vanishing of the Ext groups between simples. Now use Corollary \ref{Koszulity based on simples}. 
\end{proof}
\subsubsection{Quadratic algebras}
As Koszul algebras are quadratic, later we also provide a topological interpretation for being quadratic (see Proposition~\ref{prop: quad}). Here we state some basic background.
\begin{df}[\protect{\cite[Definition 1.2.2, p. 481]{bgs}}]
    A quadratic ring is a positively graded ring $A=\bigoplus_{j\geq 0}A_j$ such that $A_0$ is semisimple and $A$ is generated over $A_0$ by $A_1$ with relations of degree two. 
\end{df}
\begin{prop}[\protect{\cite[Proposition 2.3.1, p. 481]{bgs}}]
\label{generation over A_1}
 Let $A=\bigoplus_{i\geq 0}A_i$ be a positively graded ring where $A_0$ is semisimple. The following conditions are equivalent:
 \begin{enumerate}
     \item $\Ext^1(A_0,A_0)_{-n}=0$ for $n\ne 1$. 
     \item $A$ is generated by $A_1$ over $A_0$. 
 \end{enumerate}
\end{prop}
\begin{prop}[\protect{\cite[Theorem 2.3.2, p. 481]{bgs}}]
\label{quadratic}
 Let $A=\bigoplus_{i\geq 0}A_i$ be a positively graded ring where $A_0$ is semisimple. Suppose $A$ is generated by $A_1$ over $A_0$. If $\Ext^2(A_0,A_0)_{-n}\ne 0$ implies $n=2$, then $A$ is quadratic.    
\end{prop}

\subsubsection{Locally bouquet versus locally Cohen-Macaulay}
Following Baclawski \cite{kb} we adapt the following definition for being a bouquet (cohomologically):
\begin{df}
\label{bouquet semisimplicial set}
    Let $X$ be a semi-simplicial set. We say that $X$ is \newterm{bouquet (cohomologically)} if $\Tilde{H}^j(|X|, k)=0$ for $j<\dim |X|$ where $|X|$ is the geometric realization of $X$. 
\end{df}
\begin{ex}
    Any semi-simplicial set $X$ with $|X|$ homeomorphic to a wedge of spheres of the same dimension is bouquet. In particular any set of points is a bouquet. 
\end{ex}
\begin{rmk}
    We can assign a semi-simplicial set to any (combinatorial) ordered abstract simplicial complex (see Example \ref{ex: abstract simplicial complexes}). Definition \ref{bouquet semisimplicial set} generalizes the existing definition in \cite{kb} to semi-simplicial sets. 
\end{rmk}
\begin{df}
\label{locally bouquet}
    Let $\mathcal{C}$ be an indiscretely based category. We say that $\mathcal{C}$ is bouquet if the reduced nerve $\Bar{N}(\mathcal C)$ of $\mathcal{C}$ (see Definition \ref{reduced nerve}) is bouquet. We say that $\mathcal{C}$ is locally bouquet (cohomologically) if $\mathcal{C}(p)$ (see Definition \ref{Factorization space}) is bouquet (cohomologically) for all $p\in \Mor(\mathcal{C})$. 
\end{df}
\begin{ex}
    The poset 
        
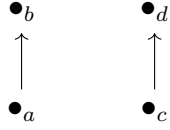
\begin{center}
	       \begin{tikzcd}
                    \bullet_b
		          &\bullet_d\\
                    \bullet_a\arrow{u}{}
		          &\bullet_c\arrow{u}{}
	       \end{tikzcd}     
        \end{center} 
    is not bouquet. However it is locally bouquet. 
\end{ex}
In case of posets, locally bouquet and locally Cohen-Macaulay are the same. To be more precise, we state the definition of Cohen-Macaulayness. Some good references are \cite{ab, bww, mw}.
\begin{df}
    A simplicial complex $\Delta$ is said to be \newterm{Cohen-Macaulay} over $k$ if 
    $$\widetilde{H}_i(\link_{\Delta}(F);k)=0$$
    for all $F\in \Delta$ and $i<\dim \link_{\Delta}(F)$ where 
    $$\link_{\Delta}(F)=\{G\in\Delta: F\cup G\in\Delta,F\cap G=\emptyset\}$$
\end{df}
\begin{df}
    Let $P$ be a poset. We say that $P$ is Cohen-Macaulay if its order complex $\Delta(P)$ is Cohen-Macaulay. We say that $P$ is \newterm{locally Cohen-Macaulay} if all the open intervals $(x,y)$ of $P$ are Cohen-Macaulay.
\end{df}
\begin{prop}
\label{prop: loc bouquet vs loc CM}
    A (graded) poset $P$ is locally Cohen-Macaulay if and only if it is locally bouquet.
\end{prop}
\begin{proof}
     Any link of a face of the order complex of an open interval $(x,y)$ can be written as a joint of the order complexes of open subintervals of $(x,y)$. Then, this follows immediately from the universal coefficient theorem and the K\"{u}nneth formula.
\end{proof}
\begin{ex}
\label{ex: CW posets are CM}
    Any (graded) \newterm{CW poset}, i.e. a poset which is isomorphic to the face poset of a regular CW complex (see Definition 2.1 and Proposition 3.1 in \cite{bjorner1984posets}), is locally Cohen-Macaulay.
\end{ex}

\subsubsection{Koszulity versus locally bouquet}
We now relate the Koszulity of category algebras to the locally bouquet property. Recall that $A$ is Koszul if and only if $\Ext^i_A(A_0,A_0)_{-n}=0$ for $i \neq n$.  In particular, we must have $\Ext^1_A(A_0,A_0)_{-n}=0$ for $n \neq 1$.  Now, to compare with the locally bouquet property for $i=1$, we use the standard convention $\widetilde H^{-1}(\emptyset) = \mathbb Z$ and is zero otherwise. This brings us to the following result.
\begin{prop} \label{prop: ext1 vanishes}
    Let $\mathcal{C}$ be an indiscretely based category. Then the  following are equivalent:
    \begin{enumerate}
   \item $\Ext^1((k\mathcal{C})_0, (k\mathcal{C})_0)_{-n}=0$ for  $n \neq 1$.
   \item $\widetilde{H}^{-1}(B\mathcal{C}(p))=0$ for all $p\in \Mor(\mathcal{C})$ with $l(p)\ne 1$.
   \item In $\mathcal C$, a morphism is \newterm{indecomposable}, i.e. has no nontrivial factorizations if and only if it has length 1.
        
   \item Any morphism of positive length in $\mathcal C$ is a composition of morphisms of length 1.
   \item The category algebra $k\mathcal C$ is generated by $k\mathcal C_1$ over $k\mathcal C _0$. 
\end{enumerate}
        In particular, if $k \mathcal C$ is Koszul then all of the above hold.
\end{prop}
\begin{proof}
We have 
\begin{align*}
\Ext^1((k\mathcal{C})_0,(k\mathcal{C})_0)_{-n} & =\bigoplus_{v,w\in Q_0}\Ext^1(S_w,S_v)_{-n} & \text{Since}(k\mathcal{C})_0= \bigoplus_{v\in Q_0} S_v \\ 
& = \bigoplus_{\substack{v,w\in Q_0,\\ p|l(p)=n,\\ t(p)=v,h(p)=w}}\Tilde{H}^{-1}(B\mathcal{C}(p)) & \text{By Proposition \ref{Ext groups}}
\end{align*}  
This immediately implies that (1) and (2) are equivalent. Since isomorphisms are the only morphisms of length zero in $\mathcal{C}$, in $\mathcal{C}$ any length one morphism is indecomposable. Now (2) implies that if a morphism has length greater than one then it is decomposable. So, (2) implies (3). Assuming (3) if a morphism has a length greater than $1$, then it is decomposable into morphisms with smaller length hence into length one morphisms by induction. So, (3) implies (4). Now (4) implies that for any morphism $p\in \mathcal{C}$ with $l(p)>1$, $B\mathcal{C}(p)\ne \emptyset$ and therefore $\widetilde{H}^{-1}(B\mathcal{C}(p))=0$. Hence (4) implies (2). Finally (5) is a categorical restatement of (4).
\end{proof}

% But still we need to have one more extra property on our categories to be able to relate Koszulity and Cohen-Macaulayness to each other. That is when we are working with an indiscretely based category $\mathcal{C}$ with grading $l:\mathcal{C}\to \mathbb{N}$, we have $\dim(B\mathcal{C}(p))=l(p)-2$ for all $p\in \Mor(\sk(\mathcal{C}))$. This turns out to be connected to the vanishing of some of the ext groups that we need to get Koszulity.

% First, we make the following definition:
% \begin{df}
% Let $\mathcal{C}$ be an indiscretely based category with grading $l:\mathcal{C}\to \mathbb{N}$. We say that $\mathcal{C}$ is well-graded if any morphism of $\mathcal{C}$ has a factorization as a composition of morphisms of length $1$.  
% \end{df}
\begin{rmk}
    Let $\mathcal{C}$ be an indiscretely based category satisfying the equivalent conditions of \Cref{prop: ext1 vanishes}. Then, $\dim(B\mathcal{C}(p))=l(p)-2$ for all $p\in \Mor(\mathcal{C})$. 
\end{rmk}
The following provides a topological interpretation for being quadratic: 
\begin{prop} \label{prop: quad}
  Let $\mathcal{C}$ be an indiscretely based category. Suppose $B\mathcal{C}(p)$ is non-empty for all $p\in \Mor(\mathcal{C})$ with $l(p)\ne 1$ and connected for all $p\in \Mor(\mathcal{C})$ with $l(p)\ne 2$. Then $k\mathcal{C}$ is quadratic. 
\end{prop}
\begin{proof}
Since $B\mathcal{C}(p)$ is non-empty for all $p\in \Mor(\mathcal{C})$ with $l(p)\ne 1$, $\widetilde{H}^{-1}(B\mathcal{C}(p))=0$ for all $p\in \Mor(\mathcal{C})$ with $l(p)\ne 1$. Hence by Proposition \ref{prop: ext1 vanishes}, $k\mathcal{C}$ is generated over $k\mathcal{C}_0$ by $k\mathcal{C}_1$. Furthermore, we have 
\begin{align*}
\Ext^i((k\mathcal{C})_0,(k\mathcal{C})_0)_{-n} & =\bigoplus_{v,w\in Q_0}\Ext^i(S_w,S_v)_{-n} & \text{Since}(k\mathcal{C})_0= \bigoplus_{v\in Q_0} S_v \\ 
& = \bigoplus_{\substack{v,w\in Q_0,\\ p|l(p)=n,\\ t(p)=v,h(p)=w}}\Tilde{H}^{i-2}(B\mathcal{C}(p)) & \text{By Proposition \ref{Ext groups}}
\end{align*}  
So, if $B\mathcal{C}(p)$ is connected for all $p\in \Mor(\mathcal{C})$ with $l(p)\ne 2$, then $$\Ext^2((k\mathcal{C})_0,(k\mathcal{C})_0)_{-n}=0$$ for $n\ne 2$. Now by Proposition \ref{quadratic} $k\mathcal{C}$ is quadratic. 
\end{proof}
\begin{ex}
Consider the poset $P=\{a\leq b\leq d\leq f,a\leq c\leq e\leq f\}$ with the following Hasse diagram 
        
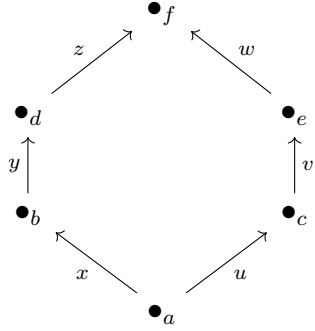
\begin{center}
	       \begin{tikzcd}
		          & \bullet_{f}\\
		          \bullet_{d}\arrow{ur}{z}&& \bullet_{e}\arrow[swap]{ul}{w}\\
                    \bullet_{b}\arrow{u}{y}&&\bullet_{c}\arrow[swap]{u}{v}\\
                    &\bullet_{a}\arrow[swap]{ur}{u}\arrow{ul}{x}
	       \end{tikzcd}     
        \end{center} 
        Let $p=xyz$. Then, $BP(p)$ is a union of two disjoint $1$-cells and so it is disconnected. So the incidence algebra of $P$ is not quadratic. This is of course obvious as the only relation is $xyz-uvw$. Also, note that $BP(p)$ is not locally bouquet.
\end{ex}
% The following proposition relates this definition to the vanishing of some of the ext groups:
% \begin{prop}
%     Let $\mathcal{C}$ be an indiscretely based category. Then, it is well-graded if and only if 
%     $\ext^1((k\mathcal{C})_0, (k\mathcal{C})_0(n))=0$ for $n\ne 1$. 
% \end{prop}
The following is our main result:
\begin{thm}
\label{Theorem}
  Let $\mathcal{C}$ be an indiscretely based category.  Then, $k\mathcal{C}$ is Koszul if and only if $\mathcal{C}$ is locally bouquet. 
\end{thm}
\begin{proof}
We have 
\begin{align*}
\Ext^i((k\mathcal{C})_0,(k\mathcal{C})_0)_{-n} & =\bigoplus_{v,w\in \Ob(\mathcal{C})}\Ext^i(S_w,S_v)_{-n} & \text{Since}(k\mathcal{C})_0= \bigoplus_{v\in \Ob(\mathcal{C})} S_v \\ 
& = \bigoplus_{\substack{v,w\in \Ob(\mathcal{C}),\\ p|l(p)=n,\\ t(p)=v,h(p)=w}}\Tilde{H}^{i-2}(B\mathcal{C}(p)) & \text{By Proposition \ref{Ext groups}}\\
& = 0 &\text{iff $B\mathcal{C}(p)$ is bouquet for all $p|l(p)=n$} 
\end{align*}
On the other hand, by Proposition \ref{Koszul}, $k\mathcal{C}$ 
 is Koszul iff $\Ext^i((k\mathcal{C})_0,(k\mathcal{C})_0)_{-n}=0$ for $i\ne n$.     
\end{proof}

The standard example of the symmetric algebra can now be seen as follows.
\begin{ex}
\label{ex: polynomial ring is Koszul}
    Let $k$ be a field and $\mathcal{C}$ be a graded category with one object $v$ whose endomorphisms are generated by $x$ i.e. $k\mathcal{C}\cong k[x]$. For $p=x^n\in \Mor(\mathcal{C})$, the space $B\mathcal{C}(p)$ is homeomorphic to the $(n-2)$-simplex which is bouquet (see also Example \ref{polynomial ring in one variable}). So we recover the fact that $k[x]$ is Koszul. Moreover, 
    $$k[x_1,...x_n]\cong \underbrace{k\mathcal{C}\otimes_k ...\otimes_k k\mathcal{C}}_{\text{n-times}}$$ 
    (See Example \ref{ex: polynomial ring in several variables}). Since $k[x]$ is Koszul, by Corollary \ref{tensor prodcut of koszul algebras} we get that $k[x_1,...,x_n]$ is Koszul.
\end{ex}
\begin{cor}[\protect{\cite{pp}\cite[Theorem 3.7, p. 406]{dw}}]
\label{cor incidence}
    The incidence algebra $k[P]$ of a graded poset $P$ is Koszul if and only if every $P$ is locally Cohen-Macaulay over $k$. 
\end{cor}
\begin{proof}
Let $P$ be a graded poset. Then the category algebra $kP$ of $P$ (regarded as a category) is isomorphic to the incidence algebra $k[P]$. For any morphism $\lambda\leq_p \mu$ in $P$, $BP(p)$ is isomorphic to the order complex of the open interval $(\lambda,\mu)$ (See Example \ref{ex: open intervals as BC(p)}). So by Theorem \ref{Theorem}, $kP$ is Koszul if and only if $P$ is locally bouquet. By Proposition \ref{prop: loc bouquet vs loc CM}, $P$ is locally bouquet if and only if it is locally Cohen-Macaulay over $k$.
\end{proof}
\begin{ex}[\protect{\cite[Proposition 7.1, p. 16]{yanagawa2005dualizing}}]
    By Example \ref{ex: CW posets are CM} and Corollary \ref{cor incidence}, the incidence algebra of any graded CW poset is Koszul.
\end{ex}
The following is another quick corollary of this Theorem \ref{Theorem}: 
\begin{cor}
\label{cor: Koszulity of subcategories}
    Let $\mathcal{D}$ be a subcategory of an indiscretely based category $\mathcal{C}$ such that any factorization of a morphism in $\mathcal{D}$ remains in $\mathcal{D}$. If $k\mathcal{C}$ is Koszul, then $\mathcal{D}$ is locally bouquet and $k\mathcal{D}$ is Koszul. 
\end{cor}
\begin{proof}
     If $k\mathcal{C}$ is Koszul, by Theorem \ref{Theorem}, $\mathcal{C}$ is locally bouquet. Now by assumption, $\mathcal{D}(p)=\mathcal{C}(p)$ for all $p\in \mathcal{D}$. So, $\mathcal{D}$ is locally bouquet and by Theorem \ref{Theorem} $k\mathcal{D}$ is Koszul. 
\end{proof}
\subsubsection{Almost discrete fibrations}
\begin{df}
    Let $F:\mathcal{C}\to \mathcal{D}$ be a functor between two (indiscretely based) categories. We say that $F$ is an \newterm{almost discrete fibration} if for each $q\in \Mor(\mathcal{D})$ and each $p\in F^{-1}(q)$ any nontrivial factorization $q=g_0\circ ...\circ g_{n+1}$ ($n\geq 0$) has a unique lift to a nontrivial factorization $p=f_0\circ ...\circ f_{n+1}$ such that $F(f_i)=g_i$ up to the following equivalence relation 
    $$(f_0, ... , f_{n+1})\sim (f'_0, ..., f'_{n+1})$$ if for all $i$ there exists an isomorphism $h_i$ such that $f'_i=f_i\circ h_i$ and $f'_{i+1}=h_i^{-1}\circ f_{i+1}$. 
 We provide a pictorial description below.
 \begin{center}
	       \tiny
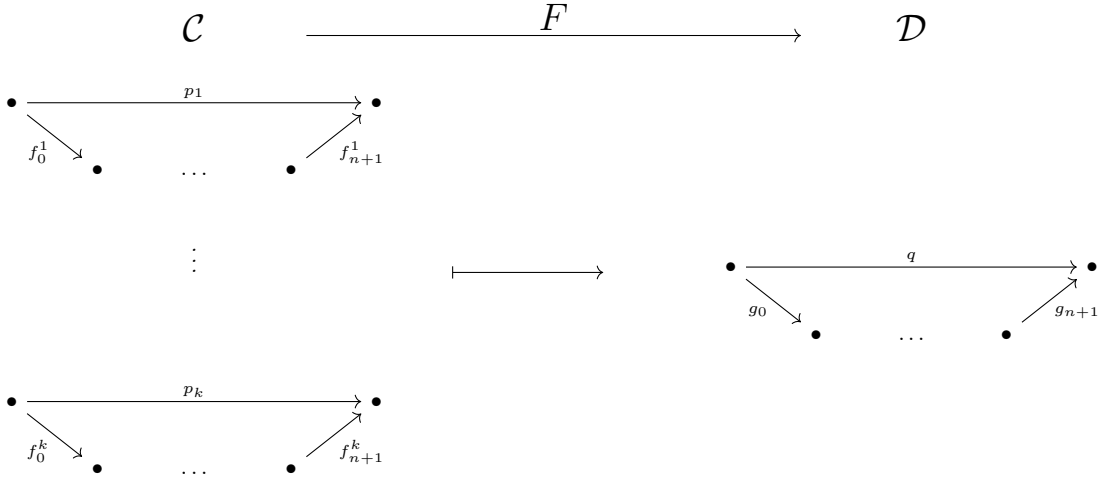
\begin{tikzcd}
                    &&\text{\large $\mathcal{C}$}&\phantom{\bullet}\arrow{rrrr}{\text{\large $F$}}
		          &&&&\phantom{\bullet}& \text{\large $\mathcal{D}$}\\
                    \bullet\ar[rrrr, "p_1"]\arrow[dr, swap,  "f^1_0"]&&&&\bullet\\
                    &\bullet&\ldots&\bullet\arrow[ur, swap, "f^1_{n+1}"]\\
                    && \vdotswithin{\ldots}&&\draw[|->] (1,0) --(3,0); &&\bullet\arrow[rrrr, "q"]\arrow[dr, swap,  "g_0"]&&&&\bullet\\
                    &&&&&&&\bullet&\ldots&\bullet\arrow[ur, swap, "g_{n+1}"]\\
                    \bullet\arrow[rrrr, "p_k"]\arrow[dr, swap, "f^k_0"]&&&&\bullet\\
                    &\bullet&\ldots&\bullet\arrow[ur, swap, "f^k_{n+1}"]\\
                    
	       \end{tikzcd}     
        \end{center}  
\end{df}
%    \begin{figure}

        % \label{fig: almost discrete}
        % \caption{Image of an almost discrete fibration with $F(p)=F(p')=q$}
        % \end{figure}
\begin{ex}
    Any equivalence between indiscretely based categories is an almost discrete fibration. 
\end{ex}
\begin{ex}
\label{ex: nonexample of alm. disc. fib.}
    Let $P=\{a\leq b,a\leq c,b\leq d,c\leq d\}$ and $P'=\{a'\leq b'\leq c'\}$ be posets. The projection 
    
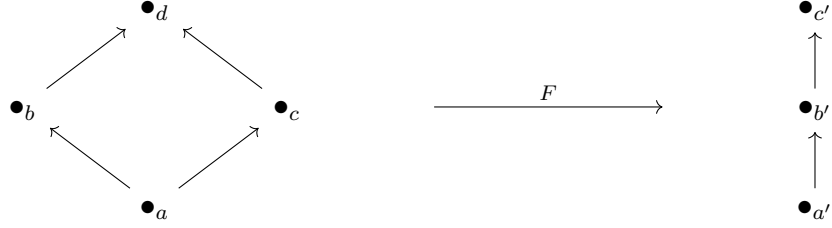
\begin{center}
	       \begin{tikzcd}
                    &\bullet_d &&&&&& \bullet_{c'}\\
                    \bullet_b\arrow{ur}&&\bullet_c\arrow{ul}&\phantom{\bullet}\arrow{rrr}{F}&&&\phantom{\bullet}&\bullet_{b'}\arrow{u}\\
                    &\bullet_a\arrow{ul}\arrow{ur}&&&&&&\bullet_{a'}\arrow{u}
	       \end{tikzcd}     
        \end{center}  
    is not an almost discrete fibration. Indeed $F^{-1}(a'\leq c')=\{a\leq d\}$ however the factorization $a'\leq b'\leq c'$ of $a'\leq c'$ has two lifts $a\leq b\leq d$ and $a\leq c\leq d$ of the factorizations of $a\leq d$. 
\end{ex}

Our main examples of almost discrete fibrations are discrete fibrations (See Proposition \ref{prop: discrete fibrations are almost discrete fibrations}). A good reference for discrete fibrations is \cite{LOREGIAN2020496}. 
 \begin{df}
    A functor $p:\mathcal{E}\to \mathcal{B}$ is a discrete fibration if for each object $E\in \Ob(\mathcal{E})$ and any morphism $f:B\to p(E)$, there exists a unique $g:E'\to E$ lifting $f$ i.e. $p(g) = f$. 
\end{df}
A nice example of discrete fibrations are quotients via free group actions on categories. These will be useful later in our applications.  Some good references for group actions on categories are \cite{Deligne1997ActionDG, BABSON2005439}.

Let $\mathcal{C}$ be a category and $G$ be a group. By a $G$-action on $\mathcal{C}$ we mean a group homomorphism 
    $$\rho: G\to \Iso(\mathcal{C})$$
where $\Iso(\mathcal{C})$ is the group of isomorphisms of $\mathcal{C}$. For any element $g\in G$, we write $g \cdot x$ to denote  $\rho(g)(x)$ where $x$ is either an object or morphism of $\mathcal{C}$. Any group action on $\mathcal{C}$ induces a group action on $\Ob(\mathcal{C})$ and $\Mor(\mathcal{C})$. We say that our group action is free if the induced group actions on $\Ob(\mathcal{C})$ and $\Mor(\mathcal{C})$ are free. By the quotient category $\mathcal{C}/G$ we mean a category whose objects are the orbits of the induced group action on $\Ob(\mathcal{C})$) and its morphisms are the orbits of the induced group action on $\Mor(\mathcal{C})$. For $[\alpha]:[A]\to [B]$ with $\alpha:A_1\to B_1$ and $[\beta]:[B]\to [C]$ with $\beta:B_2\to C_2$ define: 
$$[\beta]\circ [\alpha]:=[\beta\circ g \cdot \alpha]$$
where $g\in G$ such that $g \cdot B_1=B_2$.
One can easily show that $\mathcal{C}/G$ is a category.
\begin{ex}
\label{ex: group quotients are discrete fibrations}
Let $\mathcal{C}$ be a category with a free $G$-group action.  The projection functor 
    $$\pi:\mathcal{C}\to \mathcal{C}/G$$
    $$A\mapsto [A]$$
    $$f\mapsto [f]$$   
is a discrete fibration. Indeed if $[\alpha]:[A]\to [B]$ with $\alpha:A_1\to B_1$ and $B_2\in [B]=[B_1]$, then there exists a unique $g\in G$ such that $g.B_1=B_2$. Hence $g.\alpha:g.A_1\to g.B_1$ is a morphism with $[g.\alpha]=[\alpha]$.
\end{ex}

\begin{prop}
\label{prop: discrete fibrations are almost discrete fibrations}
    Discrete fibrations are almost discrete fibrations. 
\end{prop}
\begin{proof}
Let $p:\mathcal{E}\to \mathcal{B}$ be a discrete fibration. Let $f\in \Mor(\mathcal{E})$. We will show that for each $n$ the induced function $$p_n:\mathcal{E}(f)_n\to \mathcal{B}(p(f))_n$$
  is an isomorphism.
  \begin{enumerate}
      \item $p_n$ is injective:
      \\ Let $(f_0, ..., f_{n+1}), (f'_0, ..., f'_{n+1})\in \mathcal{E}(f)_n$ with $p_n(f_0, ..., f_{n+1})=p_n(f'_0, ..., f'_{n+1})$. So, $p(f_i)=p(f'_i)$. Since $p$ is a discrete fibration and $p(f_0)=p(f'_0)$ and $h(f_0)=h(f'_0)$, we get $f_0=f'_0$ and so $h(f'_1)=h(f'_1)$. Now by induction we get that $f_i=f'_i$ for $0\leq i\leq n+1$ and so $p_n$ is injective.  
      \item $p_n$ is surjective:
      \\ Let $(g_0, ..., g_{n+1})\in \mathcal{B}(p(f))_n$. So, $p(f)=g_0\circ ...\circ g_{n+1}$ and so $h(f)=h(g_0)$. Since $p$ is a discrete fibration, there exists $f_0\in \Mor(\mathcal{E})$ with $p(f_0)=g_0$ and $h(f_0)=h(f)$. Now note that $h(g_1)=t(g_0)=t(p(f_0))=p(t(f_0))$ and since $p$ is a discrete fibration there exists $f_1\in \Mor(\mathcal{E})$ with $h(f_1)=t(f_0)$. Repeating the same process and using induction, we get that for each $0\leq i \leq n+1$ there exists $f_i\in \Mor(\mathcal{E})$ such that $p(f_i)=g_i$, $h(f_0)=h(f)$, $t(f_i)=h(f_{i+1})$. So $f_0\circ ...\circ f_{n+1}$ is defined in $\mathcal{E}$ and we have $p(f_0\circ ...\circ f_{n+1})=g_0\circ ...\circ g_{n+1}=p(f)$ and $h(f_0\circ ...\circ f_{n+1})=h(f_0)=h(f)$. So since $p$ is a discrete fibration we get $f_0\circ ...\circ f_{n+1}=f$. So $p_n$ is surjective. 
  \end{enumerate}
\end{proof}
\begin{ex}
    Let $P=\{a\leq b,a\leq c,b\leq d,c\leq d\}$ be a poset
    
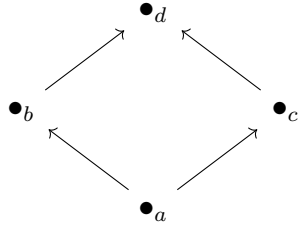
\begin{center}
	       \begin{tikzcd}
                    &\bullet_d\\
                    \bullet_b\arrow{ur}&&\bullet_c\arrow{ul}\\
                    &\bullet_a\arrow{ul}\arrow{ur}
	       \end{tikzcd}     
        \end{center}  
Let $\mathcal{D}$ be the following category

    
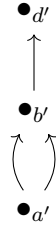
\begin{center}
	       \begin{tikzcd}
                    \bullet_{d'}\\
                    \bullet_{b'}\arrow{u}\\
                    \bullet_{a'}\arrow[bend left]{u}\arrow[bend right]{u}
	       \end{tikzcd}     
        \end{center}  
Then the projection 
        
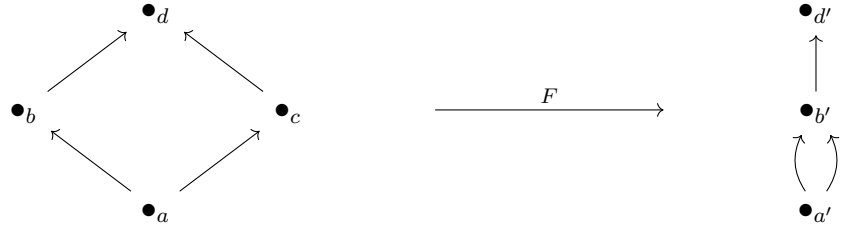
\begin{center}
	       \begin{tikzcd}
                    &\bullet_d&&&&&&\bullet_{d'}\\
                    \bullet_b\arrow{ur}&&\bullet_c\arrow{ul}&\phantom{\bullet}\arrow{rrr}{F}&&&\phantom{\bullet}&\bullet_{b'}\arrow{u}\\
                    &\bullet_a\arrow{ul}\arrow{ur}&&&&&&\bullet_{a'}\arrow[bend left]{u}\arrow[bend right]{u}
	       \end{tikzcd}     
        \end{center}  
        is an almost discrete fibration but not a discrete fibration (compare with Example \ref{ex: nonexample of alm. disc. fib.}).
\end{ex}
The following is a quick corollary of Theorem \ref{Theorem}:
\begin{cor}
\label{cor: almost discrete fibration}
        Let $F:\mathcal{C}\to \mathcal{D}$ be an almost discrete fibration between indiscretely based categories such that any morphism of $\mathcal{D}$ has a lift in $\mathcal C$. Then, the followings are equivalent: 
  \begin{enumerate}
      \item $k\mathcal{C}$ is Koszul.
      \item $k\mathcal{D}$ is Koszul.
      \item $\mathcal{C}$ is locally bouquet. 
      \item $\mathcal{D}$ is locally bouquet. 
  \end{enumerate}  
\end{cor}
\begin{proof}
Since $F:\mathcal{C}\to \mathcal{D}$ is an almost discrete fibration, the induced morphism of semi-simplicial sets
  $$F:\mathcal{C}(p)\to \mathcal{D}(F(p))$$
  defined by 
  $$F_n:\mathcal{C}(p)_n\to \mathcal{D}(F(p))_n$$
  $$(f_0, ..., f_{n+1})\mapsto (F(f_0), ..., F(f_{n+1}))$$
  is an isomorphism. So $\mathcal C$ is locally bouquet if and only if $\mathcal D$ is locally bouquet which are respectively equivalent to being Koszul by Theorem \ref{Theorem}. 
\end{proof}

\section{Reiner-Stamate Equivalence Relations}\label{sec: RS equivalences}
In this section we study the equivalence relations axiomatized by Riener and Stamate in \cite{rs}. We show that Reiner-Stamate equivalence relations can be realized as almost discrete fibrations on posets and any almost discrete fibration on a poset gives a Reiner-Stamate equivalence relation. Then we show that the category algebra of the quotients can be realized as the opposite ring of the reduced incidence algebra. Using the previous section, this recovers the Koszulity criterion for reduced incidence algebras in \cite{rs}. %Finally, we show that any homotopy path algebra can be realized as a reduced incidence algebra, i.e. the category algebra of the quotient of a poset by a Reiner-Stamate equivalence relation. 
%We will start this section by talking about when an HPA is Koszul. Since incidence algebras of posets are HPAs, this generalizes the Koszulity of incidence algebras. Koszulity of incidence algebras was previously studied in various references including \cite{pp, dw}.
%\david{This opening paragraph does not do enough to explain the section.   After this, the section has no words connecting any of the results.  Please write a narrative that explains to the reader what you are doing and why.  This does not mean just saying literally what you are doing before each thing.  Write a story that explains why you are talking about certain things.}
\subsection{Reiner-Stamate equivalence relations as almost discrete fibrations}
We begin by Reiner-Stamate equivalence relations from \cite{rs}:
\begin{df}
\label{Def: RS equivalences}
Let $P$ be a poset. Let $\sim$ be an equivalence relation on $\Int(P)$, i.e. the set of the closed intervals of $P$. We say that $\sim$ is a \newterm{Reiner-Stamate equivalence relation} if it satisfies the following axioms: 
\begin{comment}
\begin{enumerate}
    \item[A1]       If $[a,b]\sim [a',b']$ and $[b,c]\sim [b',c']$, then $[a,c]\sim [a',c']$.    \label{A1}   \item[A2] 
     The (lower, upper) interval mappings 
    $$\intt_{[a,b]}:[a,b]\to \Int(P)/\sim\times \Int(P)/\sim$$
    $$c\mapsto (\widetilde{[a,c]},\widetilde{[c,b]})$$
    have the property that whenever $[a,b]\sim [a',b']$,
    \begin{enumerate}
        \item there exists a map $\tau:[a,b]\to [a',b']$ that commutes with $\intt_{[a,b]}$, $\intt_{[a',b']}$:
        $$\intt_{[a',b']}\circ \tau=\intt_{[a,b]}$$
        \item and such a functor $\tau$ is unique. 
    \end{enumerate}   
    \label[A2]{A2}
    \item[A4] \label[A4]{A4}
     If $a\leq b_1$ and $b_2\leq c$ in $P$ with $b_1\sim b_2$, then there exist $a'\leq b'\leq c'$ in $P$ with 
    $$[a,b_1]\sim [a',b']$$
    $$[b_2,c]\sim [b',c']$$ 
\end{enumerate}    
\end{comment}

    \begin{axiom}
     \label{A1} If $[a,b]\sim [a',b']$ and $[b,c]\sim [b',c']$, then $[a,c]\sim [a',c']$.   
     \end{axiom}
     \begin{axiom}
     \label{A2}
      The (lower, upper) interval mappings 
     $$\intt_{[a,b]}:[a,b]\to \Int(P)/\sim\times \Int(P)/\sim$$
     $$c\mapsto (\widetilde{[a,c]},\widetilde{[c,b]})$$
     have the property that whenever $[a,b]\sim [a',b']$,
     \begin{enumerate}
         \item there exists a map $\tau:[a,b]\to [a',b']$ that commutes with $\intt_{[a,b]}$, $\intt_{[a',b']}$:
         $$\intt_{[a',b']}\circ \tau=\intt_{[a,b]}$$
         \item and such a map $\tau$ is unique. 
     \end{enumerate}   
     \end{axiom} 
     \setcounter{axiom}{3}
     \begin{axiom}
     \label{A4}
      If $a\leq b_1$ and $b_2\leq c$ in $P$ with $b_1\sim b_2$, then there exist $a'\leq b'\leq c'$ in $P$ with 
     $$[a,b_1]\sim [a',b']$$
     $$[b_2,c]\sim [b',c']$$   
     \end{axiom} 
\end{df}
\begin{rmk}
    In \cite{rs} there is another axiom which is as follows: 
    \setcounter{axiom}{2}
    \begin{axiom}
    \label{A3}
      The equivalence relation $\sim$ on $P$ defined by $x\sim y$ if $[a,a]\sim[b,b]$ has only finitely many equivalence classes $P/\sim$.    
    \end{axiom}
    This axiom is not necessary for our purposes.  It would imply that $P/\sim$, which we will view as a quotient category, has finitely many objects and hence that $kP/\sim$ is unital.
\end{rmk}

\begin{rmk}
\hyperref[A2]{A2} implies that each $\tau$ is an isomorphism. 

   %Indeed, let $[a,b]\sim [a',b']$ with $\tau:[a,b]\to [a',b']$. Then, since $\sim$ is an equivalence relation, it is symmetric and so if $[a',b']\sim [a,b]$ with $\sigma$. Then, $\sigma\circ \tau:[a,b]\to [a,b]$ and 
   %  $$\intt_{[a,b]}\circ (\sigma\circ \tau)=(\intt_{[a,b]}\circ \sigma)\circ \tau=\intt_{[a',b']}\circ \tau=\intt_{[a,b]}$$
   %  On the other hand, $\intt_{[a,b]}\circ Id_{[a,b]}=\intt_{[a,b]}$ and so the uniqueness will imply that $\sigma\circ \tau=Id_{[a,b]}$. In the same way we can show that $\tau\circ \sigma=Id_{[a',b']}$ and so $\tau$ is an isomorphism. 
\end{rmk}
\begin{rmk}
    If $P$ is a poset (regarded as a category) and $\sim$ is a Reiner-Stamate relation on $\Int(P)$, then $\sim$ induces an equivalence relation on $\Mor(P)$ defined by $(a\leq b)\sim (a'\leq b')$ if and only if $[a,b]\sim [a',b']$. 
\end{rmk}
The following proposition allows us to look at Reiner-Stamate equivalences as functors. 
\begin{prop}
Let $P$ be a graded poset and $\sim$ be a Reiner-Stamate equivalence relation on $\Int(P)$. Then, we have an indiscretely based category $P/\sim$ with 
\begin{itemize}
    \item $\Ob(P/\sim)=\Ob(P)/\sim$
    \item $\Mor(P/\sim)=\Mor(P)/\sim$ 
\end{itemize}
and composition 
$$[a\leq b_1]\circ [b_2\leq c]:=[(a'\leq b')\circ (b'\leq c')]=[a'\leq c']$$
defined using \hyperref[A4]{A4}. 
    \begin{center}
	    \tiny\begin{tikzcd}
		          &\bullet_{a'}\arrow{r}& \bullet_{b'}\arrow{r}&\bullet_{c'}\\
                  &\rotatebox{45}{$\sim$}&&\rotatebox{135}{$\sim$}\\
                  \bullet_a\arrow{r}&\bullet_{b_1}&&\bullet_{b_2}\arrow{r}&\bullet_c
	    \end{tikzcd}     
    \end{center} 
\end{prop}

\begin{proof}
$ $\newline
    \begin{itemize}
 \item Composition is well-defined: \\
 Follows from \hyperref[A1]{A1}.          
        \item Identity: \\
        For each object $[a]\in \Ob(P/\sim)$, $[Id_a]\in \Mor(P/\sim)$ is the identity morphism of $[a]$. 
        \item Associativity:
        \\ Let $[f]:[a]\to [b]$ with $f:a_1\to b_1$ (i.e. $a_1\leq_f b_1$), $a\sim a_1$, $b\sim b_1$, $[g]:[b]\to [c]$  with $g:b_2\to c_2$ (i.e. $b_2\leq_g c_2$), $b\sim b_2$, $c\sim c_2$, and $[h]:[c]\to [d]$  with $h:c_3\to d_3$ (i.e. $c_3\leq_h d_3$), $c\sim c_3$, $d\sim d_3$. We show that 
        $$[h]\circ ([g]\circ [f])=([h]\circ [g])\circ [f]$$
        Let $a'_1\leq_{f'_1}b'_1\leq_{g'_1}c'_1$ be such that $[a_1,b_1]\sim [a'_1,b'_1]$, $[b_2,c_2]\sim [b'_1,c'_1]$, and $[g]\circ [f]=[g'_1\circ f'_1]$.
            \begin{center}
	    \tiny\begin{tikzcd}
		          &\bullet_{a'_1}\arrow{r}{f'_1}& \bullet_{b'_1}\arrow{r}{g'_1}&\bullet_{c'_1}\\
                  &\rotatebox{45}{$\sim$}&&\rotatebox{135}{$\sim$}\\
                  \bullet_{a_1}\arrow{r}{f}&\bullet_{b_1}&&\bullet_{b_2}\arrow{r}{g}&\bullet_{c_2}
	    \end{tikzcd}     
    \end{center}  
        Let $a'_2\leq_{k'_2}c'_2\leq_{h'_2}d'_2$ be such that 
        $$[a'_1,c'_1]\sim [a'_2,c'_2]$$
        $$[c_3,d_3]\sim [c'_2,d'_2]$$
        with $[h]\circ ([g]\circ [f])=[h'_2\circ k'_2]$.
            \begin{center}
	    \tiny\begin{tikzcd}
		          \bullet_{a'_2}\arrow{rr}{k'_2}&& \bullet_{c'_2}\arrow{r}{h'_2}&\bullet_{d'_2}\\
                  &\rotatebox{90}{$\sim$}&&\rotatebox{135}{$\sim$}\\
                  \bullet_{a'_1}\arrow{r}{f'_1}& \bullet_{b'_1}\arrow{r}{g'_1}&\bullet_{c'_1}&\bullet_{c_3}\arrow{r}{h}&\bullet_{d_3}
	    \end{tikzcd}     
    \end{center}  
        Then since $[a'_1,c'_1]\sim [a'_2,c'_2]$ by \hyperref[A2]{A2} there exists $\tau:[a'_1,c'_1]\xrightarrow{\sim}[a'_2,c'_2]$. So, $a'_1\leq_{f'_1}b'_1\leq_{g'_1}c'_1$ implies $\tau(a'_1)\leq \tau(b'_1)\leq \tau(c'_1)$.  Thus, $a'_2\leq_{f'_2}\tau(b'_2)\leq_{g'_2}b_2$. So, $k'_2=g'_2\circ f'_2$. 
            \begin{center}
	    \tiny\begin{tikzcd}
		          \bullet_{a'_2}\arrow{rr}{f'_2}&& \bullet_{\tau(b'_1)}\arrow{rr}{g'_2}&&\bullet_{c'_2}\\
                  &\rotatebox{90}{$\sim$}&&\rotatebox{90}{$\sim$}\\
                  \bullet_{a'_1}\arrow{rr}{f'_1}&& \bullet_{b'_1}\arrow{rr}{g'_1}&&\bullet_{c'_1}
	    \end{tikzcd}     
    \end{center}         
        Hence, 
        \begin{center}
            \begin{align*}
                [h]\circ ([g]\circ [f])&=[h'_2\circ k'_2]\\
                &=[h'_2\circ (g'_2\circ f'_2)]\\
                &=[(h'_2\circ g'_2)\circ f'_2]\\
                &=[h'_2\circ g'_2]\circ [f'_2]\\
                &=([h'_2]\circ [g'_2])\circ [f'_2]
            \end{align*}
        \end{center}     
        So, it remains to show that $[f]=[f'_2]$, $[g]=[g'_2]$ and $[h]=[h'_2]$. By assumption, $[c_3,d_3]\sim [c'_2,d'_2]$. So, $[h]=[h'_2]$. To check $[f]=[f'_2]$ and $[g]=[g'_2]$, we have $\intt_{[a'_2,c'_2]}\circ \tau=\intt_{[a'_1,c'_1]}$ by \hyperref[A2]{A2}. Therefore, $\intt_{[a'_2,c'_2]}( \tau(b'_1))=\intt_{[a'_1,c'_1]}(b'_1)$. So, 
        $$(\widetilde{[a'_2,\tau(b'_1)]}, \widetilde{[\tau(b'_1),c'_2]})=(\widetilde{[a'_1,b'_1]},\widetilde{[b'_1,c'_1]})$$
        Therefore, $[a'_2,\tau(b'_1)]\sim [a'_1,b'_1]$ and $[\tau(b'_1),c'_1]\sim [b'_1,c'_1]$. Since by assumption $[a_1,b_1]\sim [a'_1,b'_1]$ and $[b_2,c_2]\sim [b'_1,c'_1]$, it follows that $[a'_2,\tau(b'_1)]\sim [a_1,b_1]$ and $[\tau(b'_1),c'_2]\sim [b_2,c_2]$. Hence, $[f]=[f'_2]$ and $[g]=[g'_2]$.
        \end{itemize}
           \begin{center}
	    \tiny\begin{tikzcd}
		          &&\bullet_{a'_2}\arrow{rr}{f'_2}&& \bullet_{\tau(b'_1)}\arrow{rr}{g'_2}&&\bullet_{c'_2}\arrow{rr}{h'_2}&&\bullet_{d'_2}\\
                  &&&\rotatebox{90}{$\sim$}&&\rotatebox{90}{$\sim$}&&&&\rotatebox{160}{$\sim$}\\
                  &&\bullet_{a'_1}\arrow{rr}{f'_1}&& \bullet_{b'_1}\arrow{rr}{g'_1}&&\bullet_{c'_1}&&&\bullet_{c_3}\arrow{rr}{h}&&\bullet_{d_3}\\
                  &&\rotatebox{45}{$\sim$}&&&&\rotatebox{135}{$\sim$}&&&&\rotatebox{90}{$\sim$}\\
                  \bullet_{a_1}\arrow{rr}{f}&& \bullet_{b_1}&&&&\bullet_{b_2}\arrow{rr}{g}&&\bullet_{c_2}&\bullet_{c_3}\arrow{rr}{h}&&\bullet_{d_3}
	    \end{tikzcd}     
    \end{center}   
        Thus, $P/\sim$ is a category. Since the $\tau$s are isomorphisms, it follows that $P/\sim$ is an indiscretely based category. 
\end{proof} 
\begin{ex}
Consider the poset $P=\{a\leq_f b, a\leq_g c\}$ with the following Hasse diagram 
        
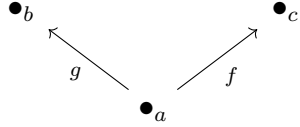
\begin{center}
	       \begin{tikzcd}
                    \bullet_{b}&&\bullet_{c}\\
                    &\bullet_{a}\arrow[swap]{ur}{f}\arrow{ul}{g}
	       \end{tikzcd}     
        \end{center}  
Then $[b,b]\sim[c,c]$ is a Reiner-Stamate equivalence relation on $\Int(P)$ and $p/\sim$ is as follows: 
    \begin{center}
	       \begin{tikzcd}
                    \bullet_{[a]}\arrow[bend left]{r}{[f]}\arrow[bend right]{r}{[g]}
		          &\bullet_{[b]}
	       \end{tikzcd}     
        \end{center} 
\end{ex}
\begin{prop}
\label{P and RS quotient of P}
    Let $P$ be a graded poset and $\sim$ be a Reiner-Stamate equivalence relation on $\Int(P)$. Then, the projection functor 
    $$\pi:P\to P/\sim$$
    $$a\mapsto [a]$$
    $$f\mapsto [f]$$
  is an  almost discrete fibration such that any morphism of $P/\sim$ has a lift in $P$.
\end{prop}
\begin{proof}
 % To check $\pi$ is a functor, we show that it preserves identities and compositions. 
 % \begin{itemize}
 %     \item $\pi$ preserves identities: 
 %     \\ In our proof in Proposition \ref{quotient of posets}, we showed that $[Id_a]=Id_{[a]}$. So, $\pi$ preserves the identities. 
 %     \item $\pi$ preserves compositions: 
 %     \\ Let $f:a\to b$ ($a\leq_f b$) and $g:b\to c$ ($b\leq_g c$). We should show that $[g\circ f]=[g]\circ [f]$ which is obvious based on our definition of $[g]\circ [f]$. 
 % \end{itemize}
 The fact that any morphism of $P/\sim$ has a lift in $P$ is immediate from our construction of $P/\sim$. Thus, it remains to show that it is an almost discrete fibration. Let $p\in \Mor(P)$.  We will show that for each $n$, 
 $$(\pi)_n:P(p)_n\to (P/\sim)(\pi(p))_n$$
 $$(f_0, ..., f_{n+1})\mapsto ([f_0], ..., [f_{n+1}])$$
 is a bijection. 
\begin{enumerate}
    \item $\pi_n$ is surjective:
    \\ Let $[f_0]\circ ...\circ [f_{n+1}]=[p]$ be a factorization of $[p]$. Then, by definition, 
    $$[f_0]\circ ...\circ [f_{n+1}]=[\theta_0(f_0)\circ ...\circ \theta_{n+1}(f_{n+1})]$$
   Therefore, $[\theta_0(f_0)\circ ...\circ \theta_{n+1}(f_{n+1})]=[p]$ . Thus if we assume $\theta_0(f_0)\circ ...\circ \theta_{n+1}(f_{n+1}):a\to b$ and $p:a'\to b'$, then there exists $\tau: [a,b]\to [a',b']$ with $\tau(\theta_0(f_0)\circ ...\circ \theta_{n+1}(f_{n+1}))=p$. Now, 
    $$\tau(\theta_0(f_0)\circ ...\circ \theta_{n+1}(f_{n+1}))=\tau(\theta_0(f_0))\circ ...\circ \tau(\theta_{n+1}(f_{n+1}))$$
    and by definition, $[\tau(\theta_i(f_i))]=[f_i]$. Therefore since 
    $$p=\tau(\theta_0(f_0)\circ ...\circ \theta_{n+1}(f_{n+1}))=\tau(\theta_0(f_0))\circ ...\circ \tau(\theta_{n+1}(f_{n+1}))$$
    we get
    \begin{center}
        \begin{align*}
            \pi_n(\tau(\theta_0(f_0)), ..., \tau(\theta_{n+1}(f_{n+1})))&=([\tau(\theta_0(f_0))], ..., [\tau(\theta_{n+1}(f_{n+1}))])\\
            &=([f_0], ..., [f_{n+1}])\\
        \end{align*}
    \end{center}
    Thus $\pi_n$ is surjective. 
    \item $\pi_n$ is injective: 
    \\ Let $(f_0,..., f_{n+1}),(f'_0, ..., f'_{n+1})\in \mathcal{C}(p)_n$ such that $\pi_n(f_0,..., f_{n+1})=\pi_n(f'_0, ..., f'_{n+1})$ as elements of $(P/\sim)(\pi(p))_n$. Therefore $[f_0]=[f'_0]$, ..., $[f_{n+1}]=[f'_{n+1}]$. So, if we assume $p:a\to b$ and $f_0:a_0=a\to a_1$, $f_1:a_1\to a_2$, ..., $f_{n+1}:a_{n+1}\to a_{n+2}=b$ and $f'_0:a'_0=a\to a'_1$, $f'_1:a'_1\to a'_2$, ..., $f'_{n+1}:a'_{n+1}\to a'_{n+2}=b$, then $[a_i,a_{i+1}]\sim [a'_i,a'_{i+1}]$ with 
    $$\tau_i:[a_i,a_{i+1}] \xrightarrow{\sim} [a'_i,a'_{i+1}]$$
    such that $\tau_i(f_i)=f'_i$. By \hyperref[A1]{A1}, we get $[a,b]\sim [a,b]$ with $\tau:[a,b]\to [a,b]$ such that for $f\in [a_i,a_{i+1}]$, $\tau(f)=\tau_i(f)$. By \hyperref[A2]{A2}, $\tau=\text{Id}$. Therefore, since $\tau(f_i)=\tau_i(f_i)$ and $\tau=Id$, we get $f_i=\tau_i(f_i)$, but  $\tau(f_i)=f'_i$ by assumption. Hence, $f_i=f'_i$. Thus $\pi_n$ is injective. 
    \end{enumerate}
\end{proof}

\begin{prop}
\label{prop: almost discrete fibrations define RS equivalence}
    Let $P$ be a graded poset regarded as a category and $F:P\to \mathcal D$ be a almost discrete fibration from $P$ to an indiscretely based category $\mathcal{D}$. Assume that any morphism of $\mathcal{D}$ has a lift in $\mathcal{C}$. Then,
    \begin{center}
     $[a,b]\sim_F [a',b']$ if and only if $F(a\leq b)=F(a'\leq b')$ 
    \end{center}
    defines a Reiner-Stamate equivalence relation on $P$ and $P/{\sim_F}$ is equivalent to  $\mathcal{D}$. 
\end{prop}
\begin{proof}
    Without loss of generality we can assume $\mathcal D$ is skeletal. It is easy to check that $\sim_F$ is an equivalence relation on $\Int(P)$. We now check that $\sim_F$ is a Reiner-Stamate equivalence relation on $\Int(P)$. 
    
    \hyperref[A1]{A1} follows from our assumption that $F$ is a functor.
    
    \hyperref[A4]{A4} follows from our assumption that any morphism of $\mathcal{D}$ has a lift in $\mathcal{C}$ and that $\mathcal{D}$ is a category.
    
    It remains to check \hyperref[A2]{A2}. Let $[a,b]\sim_F [a',b']$ and $c\in [a,b]$. So, $a\leq c\leq b$. Thus we get a factorization $F(a'\leq b')=F(a\leq b)=F(a\leq c)\circ F(c\leq b)$ of $F(a'\leq b')$. Now since $F$ is an almost discrete fibration, this factorization has a unique lift $a'\leq c'\leq b'$. Define $\tau_F(c)=c'$. Again, since $F$ is an almost discrete fibration,  $\tau_F$ is a map of posets. From our definition of $\tau$ and our definition of the equivalence relation on $\Int(P)$, it follows that 
    $$\intt_{[a',b']}\circ \tau_F=\intt_{[a,b]}$$
    So it remains to check $\tau_F$ is the unique map with this property. Consider $\tau:[a,b]\to [a',b']$ satisfying $$\intt_{[a',b']}\circ \tau_F=\intt_{[a,b]}.$$
Let $c\in [a,b]$. Then, $[a,c]\sim_F[a',\tau(c)]$ and $[c,b]\sim_F [\tau(c),b']$. Therefore, $F(a\leq c)=F(a'\leq \tau(c))$ and $F(c\leq b)=F(\tau(c)\leq b')$, which implies $a'\leq \tau(c)\leq b'$ is a lift of a factorization $F(a'\leq b')=F(a\leq c)\circ F(c\leq b)$. On the other hand, $a'\leq \tau_F(c)\leq b'$ is a lift of this factorization. Therefore $\tau(c)=\tau_F(c)$ since $F$ is an almost discrete. In summary, \hyperref[A2]{A2} is satisfied. 

Finally, $P/{\sim_F}$ is equivalent to $\mathcal{D}$ via 
     $$\Psi: P/{\sim_F}\to \mathcal{D}$$
     $$[a]\mapsto F(a)$$
     $$[a\leq_f b]\mapsto F(f)$$
\end{proof}
\begin{comment}
   \begin{thm}
    Let $P$ be a graded poset regarded as a category. There is a one-to-one correspondence
      \begin{center}
	       \begin{tikzcd}
                    \{\text{Reiner-Stamate equivalence relations on $P$}\}\arrow[bend right,swap]{d}{\Phi}
		          \\ \{\text{Almost discrete fibrations $F:P\to \mathcal{D}$ where $\mathcal{D}$ is an indiscretely based category}\}/\sim\arrow[bend right,swap]{u}{\Psi}
	       \end{tikzcd}     
        \end{center}   
    where $(F:P\to \mathcal{D})\sim (F':P\to \mathcal{D}')$ if there exists an equivalence of indiscretely based categories $\sigma:\mathcal{D}\to \mathcal{D}'$ with $F'=\sigma\circ F$. Furthermore, $\Phi(\sim)=\pi_{\sim}$ as constructed in Proposition \ref{P and RS quotient of P} and $\Psi(F)=\sim_F$ as constructed in Proposition \ref{prop: almost discrete fibrations define RS equivalence}.  
\end{thm} 
\end{comment}
\begin{thm}
\label{thm: RS equivalences vs. almost disc. fib.}
    Let $P$ be a graded poset regarded as a category. There is a one-to-one correspondence between Reiner-Stamate equivalence relations on $P$ and almost discrete fibrations $F:P\to \mathcal{D}$ where $\mathcal{D}$ is an indiscretely based category and any morphism of $\mathcal{D}$ has a lift in $P$ up to the equivalences of indiscretely based categories $\sigma:\mathcal{D}\to \mathcal{D}'$ with $F'=\sigma\circ F$.
\end{thm}
\begin{proof}
    This immediately follows from Propositions \ref{P and RS quotient of P} and \ref{prop: almost discrete fibrations define RS equivalence}. 
\end{proof}
\subsection{(Reduced) incidence algebras as category algebras and their Koszulity}
We show here that the reduced incidence algebras defined in \cite{rs} are examples of category algebras. Then using Corollary \ref{cor: almost discrete fibration}, we recover a result of \cite{rs} connecting the Koszulity of incidence and reduced incidence algebras. 
\begin{df}[\protect{\cite[Definition 1.15, p. 6]{rs}}]
Let $P$ be a graded poset in which every interval $[x,y]$ is finite, and assume one has an equivalence relation on $\Int(P)$ which is order compatible in the sense of \hyperref[A1]{A1}. Define $k[P]_{\red}$ to be the $k$-vector space having basis $\{\xi_{\widetilde{[x,y]}}\}$ indexed by the equivalence classes $\Int(P)/\sim$, with multiplication defined $k$-bilinearly via 
$$\xi_{\widetilde{[x,y]}}. \xi_{\widetilde{[z,w]}}=\xi_{\widetilde{[x',w']}}$$
if there exist $x'\leq y' \leq w'$ in $P$ with 
$$[x',y']\sim [x,y]$$
$$[y',w']\sim [z,w]$$
and zero otherwise. We call $k[P]_{\red}$ the reduced incidence algebra. 
\end{df}
\begin{prop}
\label{reduced incidence algebra is a category algebra}
    Let $P$ be a poset (not necessarily finite) and $\sim$ be a Reiner-Stamate equivalence relation on $\Int(P)$ such that $P/\sim$. Then, the category algebra $k(P/\sim)$ is isomorphic to $(k[P]_{\red})^{op}$.
\end{prop}
\begin{proof}
   Define 
    $$\phi:k(P/\sim)\to (k[P]_{\red})^{op}$$
    by sending $[f]:[a]\to [b]$ with $f:a'\to b'$ ($a'\leq_f b')$, $a\sim a'$, $b\sim b'$ to $\Bar{\xi}_{\widetilde{[a',b']}}$ and extend it linearly to all $k(P/\sim)$. Define 
    $$\psi:(k[P]_{\red})^{op}\to k(P/\sim)$$
    by sending $\Bar{\xi}_{\widetilde{[a,b]}}$ to $[f]$ where $a\leq_f b$ and extend it linearly to all $(k[P]_{\red})^{op}$. One can easily check that $\phi$ and $\psi$ are well-defined mutually inverses. We check that $\phi$ preserves multiplication.

    % \begin{itemize}
    %     \item $\phi$ preserves multiplication:
    %     \\ 
        
        Let $[f]:[a]\to [b]$ with $f:a_1\to b_1$ ($a_1\leq_f b_1)$, $a\sim a_1$, $b\sim b_1$ and $[g]:[b]\to [c]$ with $g:b_2\to c_2$ ($b_2\leq_g c_2)$, $b\sim b_2$, $c\sim c_2$ and $[g]\circ [f]=[g'\circ f']$ with $f':a_3\to b_3$, $g':b_3\to c_3$ (i.e. $a_3\leq_{f'}b_3\leq_{g'}c_3$), $[a_3,b_3]\sim [a_1,b_1]$ and $[b_3,c_3]\sim [b_2,c_2]$.
        Then,
           \begin{align*}
            \phi([g]\circ [f]) & =\phi([g'\circ f'])\\
    & =\Bar{\xi}_{\widetilde{[a_3,c_3]}} \\ 
    & = \Bar{\xi}_{\widetilde{[a_1,b_1]}}.\Bar{\xi}_{\widetilde{[b_2,c_2]}} \\
    & = \Bar{\xi}_{\widetilde{[b_2,c_2]}}.^{op}\Bar{\xi}_{\widetilde{[a_1,b_1]}}\\
    &= \phi([g]).^{op} \phi([f])
        \end{align*}
        % So, $\phi([g]\circ [f])=\phi([g'\circ f'])=\Bar{\xi}_{\widetilde{[a_3,c_3]}}$. Thus,
     
        % since $\phi([g])=\Bar{\xi}_{\widetilde{[b_2,c_2]}}$ and $\phi([f])=\Bar{\xi}_{\widetilde{[a_1,b_1]}}$, to check that $\phi$ preserves multiplication, it is enough to show that
        % $$\Bar{\xi}_{\widetilde{[b_2,c_2]}}.^{op}\Bar{\xi}_{\widetilde{[a_1,b_1]}}=\Bar{\xi}_{\widetilde{[a_3,c_3]}}$$
        % But we have 
        % $$\Bar{\xi}_{\widetilde{[b_2,c_2]}}.^{op}\Bar{\xi}_{\widetilde{[a_1,b_1]}}=\Bar{\xi}_{\widetilde{[a_1,b_1]}}.\Bar{\xi}_{\widetilde{[b_2,c_2]}}=\Bar{\xi}_{\widetilde{[a_3,c_3]}}$$ 
    %     \item $\psi$ preserves multiplication:
    %     \\ Consider $\Bar{\xi}_{\widetilde{[a_1,b_1]}}$ with $a_1\leq_f b_1$ and $\Bar{\xi}_{\widetilde{[b_2,c_2]}}$ with $b_2\leq_g c_2$. Let 
    %     $$\Bar{\xi}_{\widetilde{[b_2,c_2]}}.^{op}\Bar{\xi}_{\widetilde{[a_1,b_1]}}=\Bar{\xi}_{\widetilde{[a_3,c_3]}}$$
    %     with $a_3\leq_{f'}b_3\leq_{g'}c_3$ such that $[a_3,b_3]\sim [a_1,b_1]$ and $[b_3,c_3]\sim [b_2,c_2]$. Note that $\psi(\Bar{\xi}_{\widetilde{[a_1,b_1]}})=[f]$, $\psi(\Bar{\xi}_{\widetilde{[b_2,c_2]}})=[g]$ and 
    %     $$\psi(\Bar{\xi}_{\widetilde{[b_2,c_2]}}.^{op}\Bar{\xi}_{\widetilde{[a_1,b_1]}})=\psi(\Bar{\xi}_{\widetilde{[a_3,c_3]}})=[g'\circ f']$$
    %     Thus, to check $\psi$ preserves the multiplication, we should show that $[g]\circ [f]=[g'\circ f']$, but this is basically how we defined the composition in $P/\sim$. 
    % \end{itemize}
\end{proof}
Although the incidence algebra itself is an HPA, the reduced incidence algebras are not necessarily HPAs. The following is an easy example that shows that $k[P]_{\red}$ is not necessarily an HPA. Indeed the quotient of a cancellative category via a Reiner-Stamate equivalence relation may not be cancellative, as in the following example.
\begin{ex}
\label{example}
Consider the poset $P=\{a\leq b\leq d\leq f,a\leq c\leq e\leq f\}$ with the following Hasse diagram 
        
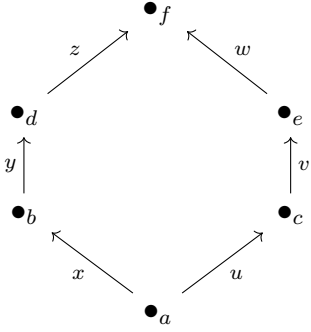
\begin{center}
	       \begin{tikzcd}
		          & \bullet_{f}\\
		          \bullet_{d}\arrow{ur}{z}&& \bullet_{e}\arrow[swap]{ul}{w}\\
                    \bullet_{b}\arrow{u}{y}&&\bullet_{c}\arrow[swap]{u}{v}\\
                    &\bullet_{a}\arrow[swap]{ur}{u}\arrow{ul}{x}
	       \end{tikzcd}     
        \end{center}  
By looking at this as a category with equality of morphisms $z\circ y\circ x=w\circ v\circ u$ (or a quiver with relation $xyz-uvw$), we have an HPA $A=kQ/I$ such that $A$ is isomorphic to the incidence algebra of this poset. Now, if we assume $[a,b]\sim [a,c]$, this equivalence relation will satisfy the axioms and we will get the following category after taking the quotient 
        
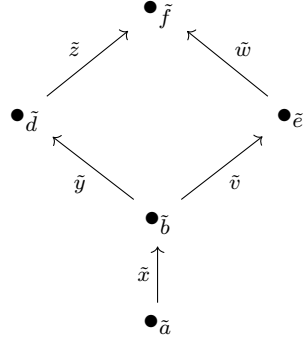
\begin{center}
	       \begin{tikzcd}
		          & \bullet_{\Tilde{f}}\\
		          \bullet_{\Tilde{d}}\arrow{ur}{\Tilde{z}}&& \bullet_{\Tilde{e}}\arrow[swap]{ul}{\Tilde{w}}\\
                    &\bullet_{\Tilde{b}}\arrow{ul}{\Tilde{y}}\arrow[swap]{ur}{\Tilde{v}}\\
                    &\bullet_{\Tilde{a}}\arrow{u}{\Tilde{x}}
	       \end{tikzcd}     
        \end{center}  
with equality of morphisms $\Tilde{z}\circ \Tilde{y}\circ \Tilde{x}=\Tilde{w}\circ \Tilde{v}\circ \Tilde{x}$. But this is not an HPA since we don't have $\Tilde{z}\circ \Tilde{y}=\Tilde{w}\circ \Tilde{v}$
\end{ex}
The following corollary is \cite[Corollary 1.19]{rs} in the graded case. 
\begin{cor}
\label{Koszul k[P]_{red}}
    Let $P$ be a graded poset (not necessarily finite) and $\sim$ be a Reiner-Stamate equivalence relation on $\Int(P)$. The followings are equivalent:
    \begin{enumerate}
        \item $k[P]$ is Koszul.
        \item $k[P]_{\red}$ is Koszul.
        \item The open intervals $(\lambda,\mu)$ of $P$ are Cohen-Macaulay. 
    \end{enumerate}
\end{cor}
\begin{proof}
            \begin{align*}
                \text{$k[P]$ is Koszul}&\iff \text{Open intervals of $P$ are Cohen-Macaulay}&\text{By Corollary \ref{cor incidence}}\\
                &\iff \text{$k (P/\sim)$ is Koszul} & \text{By Proposition \ref{P and RS quotient of P}}\\
                &&\text{and Corollary \ref{cor: almost discrete fibration}}\\
                &\iff \text{$k((P/\sim)^{op})$ is Koszul} &\text{By Corollary \ref{cor: Koszulity of kC and k(C^{op})}}\\
                &\iff\text{$k[P]_{\red}$ is Koszul}&\text{By Proposition \ref{reduced incidence algebra is a category algebra}}
            \end{align*}
\end{proof}

\section{Applications to Homotopy Path Algebras}\label{sec: App. in HPAs}
Homotopy path algebras appear naturally in toric algebraic geometry. We dedicate this section to study homotopy path algebras and their Koszulity from different perspectives. 

\subsection{Homotopy path algebras and the path poset}
Here we show that any homotopy path algebra can be obtained via a Reiner-Stamate equivalence relation on their path poset. Indeed homotopy path algebras are examples of reduced incidence algebras. Then we use this to study the Koszulity of homotopy path algebras through their path poset. 
\subsubsection{Homotopy path algebras as the reduced incidence algebra of the path poset}
 Let $A=kQ/I=k\mathcal{C}_A$ be a homotopy path algebra. Let $\Path_Q$ be the path poset of $Q$ ordered by $p<q$ if $q=p \cdot r$ for some path $r$. Define the \newterm{path poset of $A$} as follows: 
$$\Path_A=\Path_Q/\sim$$
where $p\sim q$ if and only if $p-q\in I$. One can easily check the HPA relations imply $\Path_A$ is a poset. 
\begin{prop}
\label{prop: hpas as RS quotient of path poset}
    We have an almost discrete fibration 
    $$\Phi:\Path_A\to \mathcal{C}_A$$
    $$p\mapsto h(p)$$
    $$(p<q)\mapsto q/p$$
\end{prop}
\begin{proof}
    The fact that $\Phi$ is a functor follows from the definitions. Since any morphism $r:h(p)\to h(q)$ has a unique lift $p<p \cdot r$, as in the proof of Proposition \ref{prop: discrete fibrations are almost discrete fibrations}, we get that $\Phi$ is an almost discrete fibration.   
\end{proof}  
\begin{cor}
\label{cor: hpas as red. inc. alg. of path poset}
    Let $A=kQ/I=k\mathcal{C}_A$ be an HPA. Then $\mathcal{C}_A$ can be obtained from the Reiner-Stamate equivalence relation $\sim_{\Phi}$ ($\Phi$ as in Proposition \ref{prop: hpas as RS quotient of path poset}) on $\Path_A$ given by 
    $$[p,q]\sim_{\Phi}[p',q']$$ 
    if and only if 
    $$q/p-q'/p'\in I$$
    In particular $A$ is the reduced incidence algebra $k[\Path_A]_{red}$ obtained from this equivalence relation. 
\end{cor}
\begin{proof}
    By Proposition \ref{prop: hpas as RS quotient of path poset}, $\Phi: \Path_A\to \mathcal{C}_A$ is an almost discrete fibration. By Theorem \ref{thm: RS equivalences vs. almost disc. fib.}, $[p,q]\sim_{\Phi}[p',q']$ if and only if $q/p-q'/p'\in I$ is a Reiner-Stamate equivalence relation on $\Path_A$. 
\end{proof}
\subsubsection{Koszulity of homotopy path algebras from the path poset}

\begin{thm}
\label{thm: Koszul HPA}
A graded HPA is Koszul if and only if the path poset is locally Cohen-Macaulay. 
\end{thm}
\begin{proof}
    Let $A=kQ/I$ be a graded HPA. By Corollary \ref{cor: hpas as red. inc. alg. of path poset} $A=k[\Path_A]_{red}$ via the Reiner-Stamate equivalence relation 
    \begin{center}
     $[p,q]\sim_{\Phi}[p',q']$ if and only if $q/p-q'/p'\in I$   
    \end{center}
    By Corollary \ref{Koszul k[P]_{red}} $k[\Path_A]_{red}$ is Koszul if and only if $\Path_A$ is locally Cohen-Macaulay. 
\end{proof} 
Let $A=kQ/I$ be a graded HPA. We have 
$$\Path_A=\bigsqcup_{v\in Q_0}\Path_{A,v}$$
where $Q_0$ is the set of vertices/ objects of $Q$ and $\Path_{A,v}$ is a subposet of $\Path_A$ consisting of all the path in $Q$ starting from $v$ up to our relations.
\begin{lemma}
\label{lemma: Path_A is CM iff Path_{A,v} are CM}
 Let $A=kQ/I$ be a graded HPA. Then $\Path_A$ is locally Cohen-Macaulay if and only if $\Path_{A,v}$ is locally Cohen-Macaulay for all $v\in Q_0$.    
\end{lemma}
\begin{proof}
    It follows from the fact that $\Path_A$ is the disjoint union of all $\Path_{A,v}$.  
\end{proof}
\begin{cor}
\label{cor: Koszul hpa if and only if CM Path_{A,v}}
 Let $A=kQ/I$ be a graded HPA. Then $A$ is Koszul if and only if $\Path_{A,v}$ is locally Cohen-Macaulay for all $v\in Q_0$.       
\end{cor}
\begin{proof}
    By Theorem \ref{thm: Koszul HPA} is Koszul if and only if the path poset is locally Cohen-Macaulay. By Lemma \ref{lemma: Path_A is CM iff Path_{A,v} are CM} the path poset is locally Cohen-Macaulay if and only if $\Path_{A,v}$ is locally Cohen-Macaulay for all $v\in Q_0$.     
\end{proof}

\subsection{Homotopy path algebras and stratified spaces}
 Our goal here is to show that any homotopy path algebra can be obtained as a (Reiner-Stamate) quotient via a fundamental group action on the poset of stratas of the universal cover. We first provide a general framework. Then we apply it to stratified spaces. 
\subsubsection{Skew (group) category algebras}
In various places in toric algebraic geometry, we work with multi-graded rings, i.e.\ rings with group gradings. We generalize this notion by talking about skew category algebras. For more information about skew group categories see \cite{chen2024skew}. 
\begin{df}
\label{def: Grothendieck construction}
Let $\widehat G$ be a group, $S\subseteq \widehat G$ be any subset, and $\mathcal{C}$ be an indiscretely based category with $\widehat G$-grading $\deg:\mathcal{C}\to \widehat G$.  We define the \newterm{skew category} $\mathcal{C}_{S}$ to have objects $(C,\chi)$ where $C\in \Ob(\mathcal{C})$ and $\chi\in S$ and morphisms $f_{\chi}:(C',\chi')\to (C,\chi)$ where $f:C'\to C\in \Mor(\mathcal{C})$ with $\deg(f)=\chi-\chi'$. If $S=\widehat G$ we call $\mathcal{C}_{\widehat G}$ the \newterm{skew group category}.
\end{df}

\begin{ex}
\label{McKay quiver}
Let $G$ be a finite diagonal subgroup of $GL(n,k)$, $\widehat{G}=\Hom(G,k^*)$ and $\mathcal{C}$ be a category with one object and morphisms generated by loops $x_1,x_2,...,x_n$, i.e. $k\mathcal{C}\cong k[x_1,...,x_n]$. Let $\rho_i(g)$ be the $i$-th diagonal element of the matrix $g$.  Define a $G$-action on $k\mathcal{C}$ by $g \cdot x_i=\rho_i(g^{-1})x_i$. There is a $\widehat{G}$-grading $\deg:\mathcal{C}\to \widehat{G}$ given by $\deg(x_i)=\rho_i$. Then, $\mathcal{C}_{\widehat{G}}$ is the \newterm{McKay quiver} of $G$. For more on McKay quivers see e.g\ \cite{craw1, craw2, wemyss}. 
\end{ex}
\begin{ex}
\label{ex: Beilinson quiver}
Let $k$ be a field and $\mathcal{C}$ be a graded category with one object $v$ and morphisms generated by $x_1, ..., x_n$. Then $k\mathcal{C}\cong k[x_1,...,x_n]$ and we have a $\mathbb{Z}$ grading $\deg:\mathcal{C}\to \mathbb{Z}$ defined by $\deg(x_1^{\alpha_1}x_2^{\alpha_2}...x_n^{\alpha_n})=\Sigma_{i=1}^n\alpha_i$. Now if we let $S=\{1,...,n\}$, then $\mathcal{C}_S$ is the Beilinson quiver of $\mathbb{P}^n$. For example if $n=3$, then $\mathcal{C}_S$ is the following quiver with relations $\Bar{x}_i\circ x_j=\Bar{x}_j\circ x_i$:
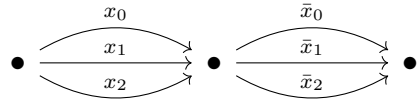
    \begin{center}
	       \begin{tikzcd}
                    \bullet\arrow[bend left]{rr}{x_0}\arrow{rr}{x_1}\arrow[bend right]{rr}{x_2}
		          && \bullet\arrow[bend left]{rr}{\Bar{x}_0}\arrow{rr}{\Bar{x}_1}\arrow[bend right]{rr}{\Bar{x}_2}
                  &&\bullet
	       \end{tikzcd}     
        \end{center} 
\end{ex}
\begin{ex}
\label{ex: hpas are C_S}
Any polynomial ring can be represented as a category algebra of a category with one object (see Example \ref{ex: polynomial ring in several variables}), in particular, the total coordinate ring of a toric variety $X_{\Sigma}$, i.e. $\mathbb{C}[x_{\rho}:\rho\in \Sigma(1)]$ is a category algebra.  This category is graded by the class group $\Cl(X_{\Sigma})$ (see \cite[Chapter 5]{cls}).  Let $S$ be any finite subset of $\Cl(X_{\Sigma})$. Then, $k\mathcal{C}_S$ is the endomorphism algebra of $\bigoplus_{D\in S}\mathcal{O}(D)$. This generalizes Example \ref{ex: Beilinson quiver} and shows that the HPAs obtained from a collection of line bundles are category algebras $k\mathcal{C}_S$. 
\end{ex}
\begin{ex}
\label{ex: hpa of F_1}
    % As a special case of Example \ref{ex: hpas are C_S}, let $X_{\Sigma}$ be the Hirzebruch surface $\mathbb{F}_1$ %(for more information about Hirzebruch surfaces as toric varieties see \cite{cls}). Then in this case, $Cl(X_{\Sigma})=\mathbb{Z}^2$ and if we consider $S=\{(0,0), (1,0),(1,1),(2,1)\}$, 
    % \david{Again, the presentation of the toric variety is not fixed you have to say what the Cox ring is and what the grading on the variables is.  Also, I don't think these weights match your choice in a different example.} 
    
The Hirzebruch surface $\mathbb F_1$ is a toric variety for the complete fan with rays $(1,0),(0,1),(-1,1),(0,-1)$ pictured below.
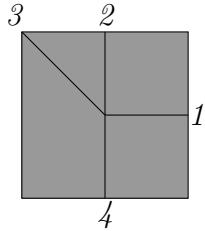
\begin{center}    
\begin{tikzpicture}[scale = .55]
%\draw[step=1cm,black,very thin] (0,0) grid (7,7);
\draw[fill=gray!80] (-2,-2)--(2,-2)--(2,2)--(-2,2)--(-2,-2);
\draw (0,0) -- (2,0);
\draw (0,0) -- (0,2);
\draw (0,0) -- (-2,2);
\draw (0,0) -- (0,-2);
\node (1) at (2.2,0) {1};
\node (2) at (0,2.4) {2};
\node (3) at (-2.2,2.4) {3};
\node (4) at (0,-2.4) {4};
\end{tikzpicture}
\end{center}
The Cox ring $R = k[x_1,x_2,x_3,x_4]$ is $\mathbb Z^2$-graded with degrees $(1,0),(-1,1),(1,0),(0,1)$.  We obtain an HPA from the degree zero piece of the endomorphism algebra of $R \oplus R(1,0) \oplus R(0,1) \oplus R(1,1)$.
Equivalently we consider the exceptional collection of the corresponding line bundles $v_0:=\mathcal{O}$, $v_1:=\mathcal{O}(1,0)$, $v_2:=\mathcal{O}(0,1)$ and $v_3:=\mathcal{O}(1,1)$. The category $\mathcal{C}_S$ is pictured as follows (see also \cite[Example 6.32]{dj}):
    \begin{center}
	       \begin{tikzcd}
                    \bullet\arrow[bend left]{rr}{x_1}\arrow[bend right]{rrrr}{x_4}\arrow[bend right]{rr}{x_3}
		          && \bullet\arrow{rr}{x_2}\arrow[bend left]{rrrr}{x_4}
                  &&\bullet\arrow[bend left,swap]{rr}{x_1}\arrow[bend right]{rr}{x_3}
                  &&\bullet
	       \end{tikzcd}     
        \end{center}     
    
\end{ex}
\begin{prop}
Let $\widehat G$ and $\widehat G'$ be two groups, and $S\subseteq \widehat G$ and $S'\subseteq \widehat G'$ be any subsets. Let $\mathcal{C}$ be an indiscretely based category with $\widehat G$-grading $\deg:\mathcal{C}\to \widehat G$  and $\mathcal{C}'$ be an indiscretely based category with $\widehat G'$-grading $\deg':\mathcal{C}'\to \widehat G'$. Then
$$\mathcal{C}_S\times {\mathcal{C}'}_{S'}\cong (\mathcal{C}\times \mathcal{C}')_{S\times S'}$$
\end{prop}
\begin{proof}
    One can easily check that 
    $$\Phi:\mathcal{C}_S\times {\mathcal{C}'}_{S'}\to  (\mathcal{C}\times \mathcal{C}')_{S\times S'}$$
    $$((C,\chi),(C',\chi'))\mapsto((C,C'),(\chi,\chi'))$$
    $$(f,f')\mapsto (f,f')$$
    is the desired isomorphism with the following inverse
    $$\Psi: (\mathcal{C}\times \mathcal{C}')_{S\times S'}\to \mathcal{C}_S\times {\mathcal{C}'}_{S'}$$
    $$((C,C'),(\chi,\chi'))\mapsto ((C,\chi),(C',\chi'))$$
    $$(f,f')\mapsto (f,f')$$
\end{proof}
\begin{ex}
\label{ex: P^1xP^1}
    Consider the Beilinson quiver of $\mathbb{P}^1$ (See Example \ref{ex: Beilinson quiver}):
    \begin{center}
	       \begin{tikzcd}
                    \bullet_0\arrow[bend left]{r}{x}\arrow[bend right]{r}{y}
		          &\bullet_1
	       \end{tikzcd}     
        \end{center} 
    obtained from $S=\{\mathcal{O},\mathcal{O}(1)\}\subset \Cl(\mathbb P^1)\cong \mathbb Z$ (See Example \ref{ex: hpas are C_S}) where $\mathcal{C}$ is a category with one object $v$ and morphisms generated by $x$ and $y$, i.e. $k\mathcal{C}\cong k[x,y]$. Then, $\mathcal{C}_S\times \mathcal{C}_S$ is as follows
    \begin{center}
	       \begin{tikzcd}
                    \bullet_{(0,1)}\arrow[bend left]{rrr}{(x,Id_1)}\arrow[bend right]{rrr}{(y,Id_1)}
		          &&&\bullet_{(1,1)}\\\\\\
                    \bullet_{(0,0)}\arrow[swap, bend left]{rrr}{(x,Id_0)}\arrow[swap, bend right]{rrr}{(y,Id_0)}\arrow[bend left]{uuu}{(Id_0,x)}\arrow[swap, bend right]{uuu}{(Id_0,y)}
		          &&&\bullet_{(1,0)}\arrow[bend left]{uuu}{(Id_1,x)}\arrow[swap, bend right]{uuu}{(Id_1,y)}
	       \end{tikzcd}     
        \end{center}   
    This can be identified with the skew category $(\mathcal{C}\times \mathcal{C})_{S\times S}$. Here  $\mathcal{C}\times \mathcal{C}$ is a category with one object $v$ and morphisms generated by $x$, $y$, $z$ and $w$, i.e. $k(\mathcal{C}\times \mathcal{C})=k[x,y,z,w]$. Furthermore 
    $$S\times S=\{(0,0),(1,0),(0,1), (1,1)\}\subseteq \Cl(\mathbb P^1)\times \Cl(\mathbb P^1)\cong \Cl(\mathbb P^1\times\mathbb{P}^1)$$
\end{ex}
\begin{prop}
\label{prop: skew category}
 Let $\widehat G$ be a group and $\mathcal{C}$ be an indiscretely based category with $\widehat G$-grading $\deg:\mathcal{C}\to \widehat G$. Then the projection functor 
     $$\pi:\mathcal{C}_{\widehat G}\to \mathcal{C}$$
    $$(C,\chi)\mapsto C$$
    $$(f_{\chi}:(C',\chi')\to (C,\chi))\mapsto f$$
is a discrete fibration. 
\end{prop}
\begin{proof}
 Let $f:C'\to C$ be a morphism in $\mathcal{C}$ and let $(C,\chi)$ be an object in the fiber $\pi^{-1}(C)$. Then, $f_{\chi}:(C',\chi-\deg(f))\to (C,\chi)$ is morphism in $\mathcal{C}_{\widehat G}$ with $\pi(f_{\chi})=f$. From the definition of morphisms in $\mathcal{C}_{\widehat G}$ it follows that $f_{\chi}:(C',\chi-\deg(f))\to (C,\chi)$ is the unique morphism with $\pi(f_{\chi})=f$. 
\end{proof}
\begin{cor}
\label{cor: Koszulity of kC and kC_H}
 Let $\widehat G$ be a group and $\mathcal{C}$ be an indiscretely based category with $\widehat G$-grading $\deg:\mathcal{C}\to \widehat G$. Then the followings are equivalent:
 \begin{enumerate}
        \item $\mathcal{C}$ is locally bouquet. 
        \item $\mathcal{C}_{\widehat G}$ is locally bouquet. 
        \item $k\mathcal{C}$ is Koszul. 
        \item $k\mathcal{C}_{\widehat G}$ is Koszul. 
 \end{enumerate}
\end{cor}
\begin{proof}
    Since $\mathcal{C}$ is an indiscretely based category, one can easily check that $\mathcal{C}_{\widehat{G}}$ is also an indiscretely based category. By Proposition \ref{prop: skew category} $\pi:\mathcal{C}_{\widehat G}\to \mathcal{C}$ is a discrete fibration. By Proposition \ref{prop: discrete fibrations are almost discrete fibrations}, $\pi:\mathcal{C}_{\widehat G}\to \mathcal{C}$ is an almost discrete fibration. Now by Corollary \ref{cor: almost discrete fibration} we get that all the above are equivalent. 
\end{proof}
We are particularly interested in the case where $\widehat G$ is the dual group of a finite diagonal subgroup of $GL(n,k)$. So, let $G$ be a finite diagonal subgroup of $GL(n,k)$ and let $\widehat{G}=\Hom(G,k^*)$ (Note that $G$ is abelian since it is a diagonal subgroup of $GL(n,k)$ and so since it is finite, $G\cong \widehat{G}$). In this case $\mathcal{C}_{\widehat{G}}$ generalizes the McKay quiver and its category algebra is isomorphic to the skew group algebra $k\mathcal{C}\# G$. For more information about skew group algebras, see \cite{wemyss}. We begin by recalling some definitions. 
\begin{df}
Let $A$ be a $k$-algebra and $G$ be a finite group together with a group homomorphism $G\to \Aut_{k-alg}(A)$. We define the skew group ring $A\#G$ as follows: as a vector space it is $A\otimes_k k[G]$, with multiplication defined as 
$$(f_1\otimes g_1).(f_2\otimes g_2):=(f_1.g_1(f_2))\otimes g_1g_2$$
for any $f_1,f_2\in A$ and $g_1,g_2\in G$ extended by linearity. 
\end{df}
\begin{ex}[\protect{\cite[Remark 2.7, Proposition 2.8 (3), p.\ 6]{craw2}}]
\label{ex: path algebra of the McKay quiver is the skew group algebra}
   The path algebra of the McKay quiver of the finite abelian group $G\subset GL(n,k)$ (See Example \ref{McKay quiver}) is isomorphic to the skew group algebra $k[x_1,...,x_n]\# G$. 
\end{ex}
Now let $G$ be a finite diagonal subgroup of $GL(n,k)$, $\mathcal{C}$ be an indiscretely based category with a free $G$-action on $k\mathcal{C}$ and $\widehat{G}$-grading $\deg:\mathcal{C}\to \widehat{G}$ where $g \cdot f=\deg(f)(g^{-1}) \cdot f$ in $k\mathcal{C}$ for all $g\in G$ and $f\in \Mor(\mathcal{C})$. Under these assumptions the following generalizes Example \ref{ex: path algebra of the McKay quiver is the skew group algebra}:
\begin{prop} Assume $\chr k \nmid |G|$.
\label{prop: k((C_{hat{G}})^{op}) is isomorphic to kC*G}
    $$k((\mathcal{C}_{\widehat{G}})^{op})\cong k\mathcal{C}\#G$$
\end{prop}
\begin{proof}
    Define 
    $$\Phi:k((\mathcal{C}_{\widehat{G}})^{op})\to k\mathcal{C}\#G$$
    $$f_{\chi}\mapsto \frac{1}{|G|}\Sigma_{g\in G}\chi(g).f\otimes g$$
    where $f_{\chi}:(C,\chi)\to (C',\chi')$ in $(\mathcal{C}_{\widehat{G}})^{op}$ and 
    $$\Psi: k\mathcal{C}\# G\to k\mathcal{C}_{\widehat{G}}$$
    $$f\otimes g\mapsto \Sigma_{\chi\in \widehat{G}}\chi^{-1}(g)f_{\chi}$$
    One easily can check that these give the desired isomorphism.
\end{proof}
\begin{cor}
\label{cor: Koszulity of kC and kC*G}
Let $G$ be a finite diagonal subgroup of $GL(n,k)$ with $\chr k \nmid |G|$. Let $\mathcal{C}$ be an indiscretely based category with a free $G$-action on $k\mathcal{C}$ and $\widehat{G}$-grading $\deg:\mathcal{C}\to \widehat{G}$ where $g.f=\deg(f)(g^{-1}).f$ in $k\mathcal{C}$ for all $g\in G$ and $f\in \Mor(\mathcal{C})$. Then the followings are equivalent:
\begin{enumerate}
    \item $k\mathcal{C}$ is Koszul.
    \item $k\mathcal{C}\# G$ is Koszul. 
\end{enumerate}
\end{cor}
\begin{proof}
By Corollary \ref{cor: Koszulity of kC and kC_H}, $kC$ is Koszul if and only if $kC_{\widehat{G}}$ is Koszul. By Corollary \ref{cor: Koszulity of kC and k(C^{op})} $kC_{\widehat{G}}$ is Koszul if and only if $k((\mathcal{C}_{\widehat{G}})^{op})$ is Koszul. By Proposition \ref{prop: k((C_{hat{G}})^{op}) is isomorphic to kC*G} $k((\mathcal{C}_{\widehat{G}})^{op})$ is isomorphic to $k\mathcal{C}\# G$. So (1) and (2) are equivalent. 
\end{proof}
\begin{ex}
Let $G$ be a finite diagonal subgroup of $GL(n,k)$ with $\chr k \nmid |G|$, $\widehat{G}=\Hom(G,k^*)$ and $\mathcal{C}$ be a category with one object and morphisms generated by loops $x_1,x_2,...,x_n$, i.e. $k\mathcal{C}\cong k[x_1,...,x_n]$. 
 Let $\rho_i(g)$ be the $i$-th diagonal element of the matrix $g$.  Define a $G$-action on $k\mathcal{C}$ by $g \cdot x_i=\rho_i(g^{-1})x_i$.  Define a $\widehat{G}$-grading $\deg:\mathcal{C}\to \widehat{G}$ by $\deg(x_i)=\rho_i$. %Then by Corollary \ref{cor: Koszulity of kC and kC*G}, the Koszulity of $k[x_1,...,x_n]$ and $k[x_1,...,x_n]\# G$ are equivalent.
Since $k[x_1,...,x_n]$ is Koszul (see Example~\ref{ex: polynomial ring is Koszul}),  $k[x_1,...,x_n]\# G$ is Koszul as well. 
\end{ex} 

Let $\widehat G$ be a partially ordered group. Assume we have a $\widehat G$-grading $\deg:\mathcal{C}\to \widehat G$. Then, if $\deg(f)\geq 0_{\widehat G}$ for all $f\in \Mor(\mathcal{C})$, we have a functor 
$$p:\mathcal{C}_S\to S$$
$$(C,\chi)\mapsto \chi$$
$$(f_{\chi}:(C',\chi')\to (C,\chi))\mapsto (\chi'<\chi)$$
where $S$ is viewed as a subposet of $G$. $p$ may not be an almost discrete fibration. In the following specific example $p$ is actually an isomorphism:
\begin{ex}
    Let $\lambda$ be a \newterm{pointed affine semigroup}, i.e. a finitely-generated sub-semigroup of the additive group $(\mathbb{Z}^n,+)$ such that an element $\lambda\ne 0$ in $\Lambda$ never has its additive inverse $-\lambda$ in $\Lambda$ (For more information see \cite{rs}). Let $\widehat G$ be its Grothendieck group. Then $\widehat G$ is a poset via $\lambda\leq \mu$ if and only if $\mu-\lambda\in \Lambda$. Since $\widehat G$ is a partially ordered group and $\Lambda=\widehat G_{\geq 0}$, we get that $\Lambda$ is locally Cohen-Macaulay if and only if $\widehat G$ is locally Cohen-Macaulay. If $\mathcal{C}$ is a category with one object and morphisms $\Lambda$, then $\mathcal{C}_{\widehat G}$ is isomorphic to $\widehat G$. So, by Proposition \ref{prop: loc bouquet vs loc CM}, $\Lambda$ is locally Cohen-Macaulay if and only if $\mathcal{C}_{\widehat G}$ is locally bouquet. By Corollary \ref{cor: Koszulity of kC and kC_H}, $k\mathcal{C}=k\Lambda$ is Koszul if and only if $\Lambda$ is locally Cohen-Macaulay. This recovers \cite[Corollary 2.2, p.\ 381]{prs} and \cite[Theorem 1.7, p. 4]{rs}. 
\end{ex}

\begin{lemma}
\label{lemma: morphism factorizations of skew categories}
Let $\widehat G$ be a partially ordered abelian group and $S\subseteq \widehat  G$ be a \newterm{saturated} subset, i.e. if $a,b\in S$ and $a<c<b$ then $c\in S$. Let $\mathcal{C}$ be an indiscretely based category with $\widehat G$-grading $\deg: \mathcal{C}\to \widehat G$ such that $\deg(f) \geq 0_{\widehat G}$. Then any morphism factorization of a morphism of $\mathcal{C}_S$ stays in $\mathcal{C}_S$.
\end{lemma}
\begin{proof}
   Since $\mathcal{C}$ is positively graded via $\deg:\mathcal{C}\to G$, any morphism factorization in $\mathcal{C}_S$ will give us a chain in $S$. Now since $S$ is saturated, any morphism factorization of a morphism of $\mathcal{C}_S$ stays in $\mathcal{C}_S$.  
\end{proof}
\begin{cor}
\label{cor: Koszul skew category algebras}
    Let $\widehat G$ be a partially ordered abelian group and $S\subseteq \widehat  G$ be a saturated subset. Let $\mathcal{C}$ be an indiscretely based category with $\widehat G$-grading $\deg: \mathcal{C}\to \widehat G$ such that $\deg(f) \geq 0_{\widehat G}$. If $k\mathcal{C}$ is Koszul, then $k\mathcal{C}_S$ is Koszul. 
\end{cor}
\begin{proof}
Since $k\mathcal{C}$ is Koszul, by Corollary \ref{cor: Koszulity of kC and kC_H}, $k\mathcal{C}_G$ is Koszul and $\mathcal{C}_G$ is locally bouquet. Now, since $S$ is saturated by Lemma \ref{lemma: morphism factorizations of skew categories}, any morphism factorization of a morphism of $\mathcal{C}_S$ stays in $\mathcal{C}_S$. By Corollary \ref{cor: Koszulity of subcategories}, $\mathcal{C}_S$ is Koszul. 
\end{proof}
\begin{ex}
As in Example \ref{ex: Beilinson quiver}, the Beilinson quiver can be written as $k\mathcal{C}_S$ where $\mathcal{C}$ is a graded category with one object $v$ and morphisms generated by $x_1,..., x_n$ and $\mathbb Z$-grading given by  $\deg(x_1^{\alpha_1}x_2^{\alpha_2}...x_n^{\alpha_n})=\Sigma_{i=1}^n\alpha_i$. Consider the saturated chain $S=\{1,...,n\} \subseteq \mathbb Z$. As $k\mathcal{C}=k[x_1,...,x_n]$ is Koszul, we see that the path algebra of the Beilinson quiver $k\mathcal{C}_S$ is Koszul. 
\end{ex}
\begin{ex}
\label{ex: F_1 is not Koszul}
The HPA of the Hirzebruch surface $\mathbb F_1$ from Example~\ref{ex: hpa of F_1} is not Koszul. Indeed it is not quadratic since we have the relation $x_1x_2x_3=x_3x_2x_1$ (see also \cite[Example 6.32]{dj}). %In this case, we consider $\{(1,0),(-1,1)\}$ as a basis of $\Cl(F_1)$.
From Example~\ref{ex: hpa of F_1}, $S$ is ordered as follows
$$(0,0)<(1,0)<(0,1)<(1,1)$$
However it is not saturated. For example 
$$(0,0)<(-1,1)<(0,1)$$
and $(0,0),(0,1)\in S$ but $(-1,1)\notin S$. 
\end{ex}

\subsubsection{Stratified spaces}
Stratified spaces are topological spaces stratified by a poset. First we give a brief overview.  An excellent detailed reference is \cite{l}. For more  examples and information about stratifications relevant to our context see \cite{dj, dm, hhl}. 
\begin{df}
    Let $S$ be a partially ordered set. We will regard $S$ as a topological space, where a subset $U\subset S$ is open if it is closed upward: that is, if $x\leq y$, and $x\in U$ implies $y\in U$. We call this topology \newterm{Alexandroff topology}. 
\end{df}
\begin{df}
    Let $X$ be a topological space and $S$ be a poset. An $S$-stratification of $X$ is a continuous map $f:X\to S$ where $S$ is equipped with the Alexandroff topology. Given an $S$-stratification of $X$ and an element $a\in S$, we define $X_a:=f^{-1}(a)$ and we call them stratas. Furthermore, for an open interval $(a,b)\subset S$, we define 
    $$X_{(a,b)}=f^{-1}(a,b)=\cup_{c\in (a,b)}X_c$$
\end{df}
\begin{comment}

\begin{rmk}
    For a $S$-stratification $f:X\to S$ and an open interval $(a,b)\subset S$, 
    $$f|_{X_{(a,b)}}:X_{(a,b)}\to (a,b)$$
    $$x\mapsto f(x)$$
    defines a $(a,b)$-stratification on $X_{(a,b)}$.
\end{rmk}
\end{comment}
%\begin{rmk}
    %In \cite{l}, stratified topological spaces were defined using the opposite convention of upward closed opens. 
%\end{rmk}
\begin{ex}
\label{ex: stratification of P^1}
    We have the following stratification of circle $S^1$: 
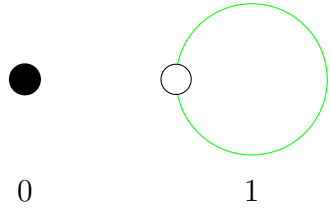
\begin{center}
    \begin{tikzpicture}
    % Dot labeled with 1
    \node[circle, fill=black, inner sep=0.15cm] at (0, 0) {};
    \node[below] at (0, -1.2) {$0$};

    % Circle with a smaller open circle on its left, labeled with 0
    \draw[green] (3, 0) circle (1cm);
    \filldraw[white] (2, 0) circle (0.2cm);
    \draw (2, 0) circle (0.2cm);
    \node[below] at (3, -1.2) {$1$};
\end{tikzpicture}
\end{center}
    
\end{ex}
\begin{ex}
\label{ex: stratification of P^2}
    We have the following stratification of the $2$-dimensional torus:
    
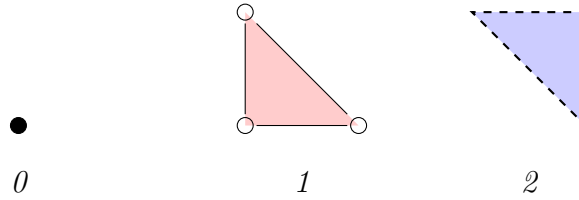
\begin{center}
    % I drew this with Chat GPT (Pouya)
    \begin{tikzpicture}[scale=1.5]

% Triangle 1 (Dashed Dark Blue Boundary, Light Blue Fill, Rotated 90° Counterclockwise)
\filldraw[fill=blue!20, draw=black, dashed, thick] (5,0) -- (5,1) -- (4,1) -- cycle;

% Triangle 2 (Erase red lines, Draw black lines between circles and form boundary)
\fill[red!20] (2,0) -- (3,0) -- (2,1) -- cycle; % Light red fill for inside

% Black lines covering all the edges of the triangle, stopping at the circle edges
\draw[black] (2,0.1) -- (2,0.9); % Bottom Left Node to Bottom Left Circle (Stops before entering)
\draw[black] (2.1,0) -- (2.9,0); % Bottom Right Node to Bottom Right Circle (Stops before entering)
\draw[black] (2.1,0.9) -- (2.9,0.1); % Top Node to Top Circle (Stops before entering)

% Empty Circles (ensure black lines don't overlap)

\draw (2,0) circle (2pt); % Bottom Left Node (Empty Circle)
\draw (3,0) circle (2pt); % Bottom Right Node (Empty Circle)
\draw (2,1) circle (2pt); % Top Node (Empty Circle)

% Point 3 (Single Black Filled Point)
\filldraw (0,0) circle (2pt); % Single Point

% Labels
\node at (0,-0.5) {0}; % Label under Point
\node at (2.5,-0.5) {1}; % Label under Triangle 2
\node at (4.5,-0.5) {2}; % Label under Triangle 1

\end{tikzpicture}       
    \end{center}

\end{ex}
\begin{ex}
\label{ex: stratification of F_1}
    The following is another stratification of $2$-dimensional torus:
    
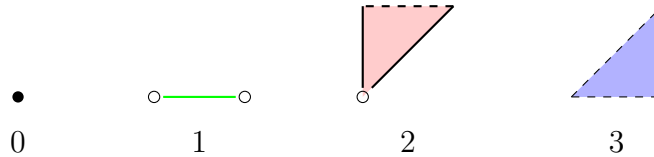
\begin{center}
    % I drew this with Chat GPT (Pouya)
\begin{tikzpicture}[scale=1.2]

% Diagram 1: Single point
\node[fill=black, circle, inner sep=1.5pt] at (0,0) {};
\node at (0,-0.5) {$0$};

% Diagram 2: Edge between two points with green line and both empty circles
\node[draw, circle, inner sep=1.5pt] at (1.5,0) {}; % Empty circle
\node[draw, circle, inner sep=1.5pt] at (2.5,0) {}; % Empty circle
\draw[green, thick] (1.6,0) -- (2.4,0); % Green line
\node at (2,-0.5) {$1$};

% Diagram 3: Flipped triangle (mirrored around diagonal line)
\begin{scope}[scale=1,xscale=-1,shift={(-8.6,0)}] % Flip around the diagonal
    \fill[red!20] (3.8,1) -- (4.8,0) -- (4.8,1) -- cycle; % Light red fill
    \draw[thick] (3.8,1) -- (4.7,0.1); % Bottom solid line
    \draw[dashed, thick] (3.8,1) -- (4.8,1); % Right solid line
    \draw[thick] (4.8,0.1) -- (4.8,1); % Top dashed line
    \node[draw, circle, inner sep=1.5pt] at (4.8,0) {}; % Empty circle at bottom vertex
\end{scope}

\node at (4.3,-0.5) {$2$};

% Diagram 4: Triangle with dashed lines and light blue fill
\fill[blue!30] (6.1,0) -- (7.1,0) -- (7.1,1) -- cycle; % Light blue fill
\draw[dashed] (6.1,0) -- (7.1,0);
\draw[dashed] (7.1,0) -- (7.1,1);
\draw[dashed] (6.1,0) -- (7.1,1);
\node at (6.6,-0.5) {$3$};

\end{tikzpicture}    
        
    \end{center}
\end{ex}
\begin{ex}
\label{cube stratification}
    The negative of the floor map 
    $$f:\mathbb R^n\to \mathbb Z^n$$
    $$(a_1,...,a_n)\mapsto-(\lfloor a_1\rfloor,...,\lfloor a_n,\rfloor)$$
    defines a $\mathbb Z^n$-stratification of $\mathbb R^n$ called the \newterm{cube stratification}. 
\end{ex}
\begin{ex}
    Any regular CW complex $X$ is a $S_{CW}$-stratified space where $S_{CW}$ is its face poset.
\end{ex}
\begin{comment}
\begin{df}[\protect{\cite[Definition 4.13, p. 13]{dj}}]
    Let $A$ be an HPA with vertices the poset $\mathcal{I}$ and $\mathcal{C}=\mathcal{C}_A$. The \newterm{tree stratification} denoted by $S^{tr}$ of $X_A=B\mathcal{C}$ is the stratification given by 
    $$X_A=\coprod_{i\in \mathcal{I}}S_i^{tr}$$
    where 
    $$S_i^{tr}=B(\mathcal{C}_{\geq i})\setminus\cup_{j>i} B(\mathcal{C}_{\geq j})$$
    where $\mathcal{C}_{\geq k}$ is the full subcategory of $\mathcal{C}$ consisting of all the objects greater than or equal to $k$. 
\end{df}
\end{comment}
\begin{df}
   Let $f:X\to S$ be an $S$-stratified space. An \newterm{entrance path} is a continuous map 
$$\gamma:[0,1]\to X$$
such that $f\circ \gamma:[0,1]\to S$ is a map of posets. 
\end{df}
\begin{df}
Let $X=\coprod_{i\in I}S_i$ be a stratification. Choose a collection of base points $x_i\in S_i$ (indexed by $i\in I$). The \newterm{entrance path category} with respect to $S$ is the category whose objects are the base points $x_i$ and whose morphisms from $x_i$ to $x_j$ are homotopy classes of entrance paths $\gamma$ with $\gamma(0)=x_i$ and $\gamma(1)=x_j$. We denote this category by $\Ent_S(X)$.
\end{df}
\begin{ex}
    Consider the stratification of the circle $S^1$ as in Example \ref{ex: stratification of P^1} for $X=S^1$ and $S=\{0,1\}$. Then $\Ent_S(X)$ is as follows
    \begin{center}
	       \begin{tikzcd}
                    \bullet_{0}\arrow[bend left]{r}{x}\arrow[bend right]{r}{y}
		          &\bullet_{1}
	       \end{tikzcd}     
        \end{center} 
    This is indeed the Beilinson quiver of $\mathbb P^1$. For the stratification of the $2$-dimensional torus $\mathbb T^2$ as in Example \ref{ex: stratification of P^2} for $X= \mathbb{T}^2$ and $S=\{0,1,2\}$, $\Ent_S(X)$ is the Beilinson quiver of $\mathbb{P}^2$ (See Example \ref{ex: Beilinson quiver}). 
\end{ex}
\begin{ex}
    Consider the stratification of $2$-dimensional torus $\mathbb{T}^2$ as in Example \ref{ex: stratification of F_1} for $X=\mathbb{T}^2$ and $S=\{0,1,2,3\}$. Then $\Ent_S(X)$ is the same category as in  Example \ref{ex: hpa of F_1}. 
\end{ex}
Homotopy path algebras are precisely category algebras of entrance categories.
\begin{df}
    Let $f:X\to S$ be an $S$-stratified space. We define the HPA of $f:X\to S$ to be $k \Ent_S(X)$ and we denote it by $A_S(X)$. 
\end{df}
This is equivalent to the definition in Example \ref{hpa quiver category} (see \cite[Proposition 3.12]{dj}).
Next we discuss the stratification of the universal cover obtained from the stratification of the base. Let $f:X\to S$ be a path-connected $S$-stratified space admitting a universal covering map $\widetilde{X}\xrightarrow{p} X$. Fix $x_0\in X$ and for each $v\in X$ let $\alpha_v$ be a path from $x_0$ to $v$. Consider the set \[
\widetilde{S}=S\times \pi_1(X, x_0)\] ($x_0\in X$) equipped with poset structure
$(s,g)\leq (s',g')$ if and only if
\[
\text{ $\exists$ an entrance path in $X$  from $s$ to $s'$ with a unique lift from $(s,g)$ to $(s',g')$.}
\]
This makes $\widetilde{S}$ into a poset. Then the map 
$$\Tilde{f}:\widetilde{X}\to \widetilde{S}$$
$$[\gamma]\mapsto (f(\gamma(1)),\alpha^{-1}_{\gamma(1)}\star \gamma)$$
defines a $\widetilde{S}$-stratification on $\widetilde{X}$. This makes $\Ent_{\widetilde{S}}(\widetilde{X})$ to be the skew group category associated to $\Ent_S(X)$ graded by the fundamental group $\pi_1(X, x_0)$. That is, $\Ent_S(X)$ is naturally a subcategory of the fundamental groupoid $\pi_1(X)$.  Denote the inclusion functor by 
$i:\Ent_S(X)\to \pi_1(X)$. Let $F:\pi_1(X)\to \pi_1(X,x_0)$ be an equivalence. Then, $\pi_S:=p\circ i: \Ent_S(X)\to  \pi_1(X,x_0)$ is a $\pi_1(X,x_0)$-grading on $\Ent_S(X)$. 
\begin{prop}
    Let $X\to S$ be a path-connected $S$-stratified space admitting a universal cover $\widetilde{X}$. Then the skew group category $(\Ent_S(X))_{\pi_1(X,x_0)}$ is isomorphic to $\Ent_{\widetilde{S}}({\widetilde{X}})$.
\end{prop}
\begin{proof} 
Let $p: \widetilde X \to X$ denote the universal covering map.  Fix a lift $\widetilde x_0 \in \widetilde X$ of $x_0$.  Then given a path $\alpha$ in $X$ starting at $x_0$, denote by $\widetilde \alpha$ the unique lift of this path starting at $\widetilde x_0$.  Let $\widetilde{\gamma}_{\tilde v}$ be a path from $\widetilde x_0$ to $\widetilde v$.
The isomorphism is given by
\begin{align*}
\Psi: \Ent_{\widetilde{S}}({\widetilde{X}}) & \to (\Ent_S(X))_{\pi_1(X,x_0)} \\
\tilde v & \mapsto (p(\widetilde v), F(p \circ \widetilde{\gamma}_{\tilde v})) \\
\widetilde \alpha & \mapsto p \circ \widetilde \alpha
\end{align*}
The inverse is described as follows.  For each $v \in \Ob(\Ent_{S}(X))$, the equivalence $F$ induces an isomorphism  $\pi_1(X, x_0) \cong \Hom_{\pi_1(X)}(x_0, v)$.  Let $\gamma_v \in \Hom_{\pi_1(X)}(x_0, v)$ be the image of the identity under this isomorphism.  Then we have
\begin{align*}
\Psi^{-1}: (\Ent_S(X))_{\pi_1(X,x_0)} & \to  \Ent_{\widetilde{S}}({\widetilde{X}}) \\
(v, \gamma) & \mapsto \widetilde{\gamma \circ \gamma_v}(1) \\
\alpha & \mapsto \widetilde \alpha
\end{align*}
\end{proof}

Moreover we have the following: 
\begin{prop}
\label{prop: koszulity and universal cover}
    Let $f:X\to S$ be a path-connected $S$-stratified space admitting a universal cover $\widetilde{X}$. Then
    \begin{enumerate} 
        \item The functor 
        $$\Phi:\Ent_{\widetilde{S}}(\widetilde{X})\to \Ent_S(X)$$
        $$(s,g)\mapsto s$$
        $$\widetilde{\gamma}\mapsto \pi\circ \widetilde{\gamma}$$
        is a discrete fibration.  
        \item $\Ent_{\widetilde{S}}(\widetilde{X})$ is isomorphic to $\widetilde{S}$.
        \item We have a free $\pi_1(X,x_0)$-action on $\widetilde{S}$ and $\widetilde{S}/\pi_1(X,x_0)$ is isomorphic to $\Ent_S(X)$. 
        \item The HPA $A_S(X)$ is the reduced incidence algebra $k[\widetilde{S}]_{red}$ obtained from the Reiner-Stamate equivalence relation $\sim_{\Phi}$. 
        \item For each $v\in \Ob(\Ent_S(X))$ and $\widetilde{v}\in \Phi^{-1}(v)$, $\Phi$ induces an isomorphism
        
        $$\Path_{A_{\widetilde{S}}(\widetilde{X}),\widetilde{v}}\xrightarrow{\sim}\Path_{A_S(X),v}$$
        $$\widetilde{\gamma}\mapsto \Phi(\widetilde{\gamma})$$
        Furthermore $\Path_{A_{\widetilde{S}}(\widetilde{X}),\widetilde{v}}$ is isomorphic to the open set $\widetilde{S}_{\geq\widetilde{v}}$ in $\widetilde{S}$ with the Alexandroff topology.  
        \item The following are equivalent:
        \begin{enumerate}
           % \item %$\Exit_{\widetilde{S}}(\widetilde X)$ is locally Cohen-Macaulay. 
            \item  $\widetilde{S}$ is locally Cohen-Macaulay. 
            \item Open subsets $\widetilde{S}_{\geq\widetilde{v}}$ of $\widetilde{S}$ with the Alexandroff topology are locally Cohen-Macaulay. 
            \item $A_{S}(X)$ is Koszul. 
            \item $k[\widetilde S]$ is Koszul. 
        \end{enumerate}
    \end{enumerate}
\end{prop}
\begin{proof}
    (1) follows immediately from the definition of a discrete fibration. (2) follows from our definitions since on a simply connected space, any entrance path is uniquely determined by its starting and ending point. Observe that
    $\pi_1(X,x_0)$ acts on $\Ent_{\widetilde{S}}({\widetilde{X}})$ by 
    $$g.((s_1,g_1)\to (s_2,g_2)):=(s_1,gg_1)\to (s_2,gg_2)$$
    so that $\Ob(\Ent_{\widetilde{S}}({\widetilde{X}}))/\pi_1(X,x_0)=\Ob(\Ent_S(X))$ and $\Mor(\Ent_{\widetilde{S}}({\widetilde{X}}))/\pi_1(X,x_0)=\Mor(\Ent_S(X))$. Therefore $\Ent_{\widetilde{S}}({\widetilde{X}})/\pi_1(X,x_0)$ is  isomorphic to $\Ent_S(X)$. Now (3) follows from (2). By Proposition \ref{prop: discrete fibrations are almost discrete fibrations}, any discrete fibration is an almost discrete fibration. Hence (4) follows from Theorem \ref{thm: RS equivalences vs. almost disc. fib.}. The isomorphism $\Path_{A_{\widetilde{S}}(\widetilde{X}),\widetilde{v}}\xrightarrow{\sim}\Path_{A_S(X),v}$ in (5) follows immediately from the path lifting property of the universal cover. Since $\widetilde{S}$ is a poset, we get that $\Path_{A_{\widetilde{S}}(\widetilde{X}),\widetilde{v}}$ is isomorphic to the open set $\widetilde{S}_{\geq\widetilde{v}}$. For (6) the definition of locally Cohen-Macaulay immediately implies that (a) and (b) are equivalent. By (2) and Corollary \ref{cor: almost discrete fibration} we get that (a) is also equivalent to (c) and (d). 
\end{proof}
\begin{ex}
    For the stratification of the circle $S^1$ in Example \ref{ex: stratification of P^1}, the stratification of the universal cover $\mathbb R$ together with its entrance paths can be pictured as follows:
    
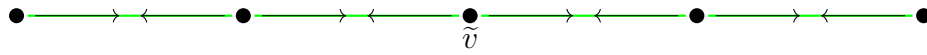
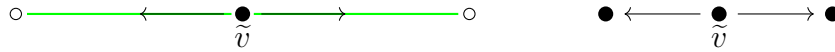
\begin{center}
\begin{tikzpicture}[scale=1.5, dot/.style={circle, fill, inner sep=2pt}]
    % Nodes and line
    \node (A) at (0,0) {};
    \node (B) at (2,0) {};
    \node (C) at (4,0) {};
    \node (D) at (6,0) {};
    \node (E) at (8,0) {};
    \draw[green,thick] (A) -- (B) -- (C) -- (D) -- (E);

    % Dots
    \node[dot] (A) at (0,0) {};
    \node[dot] (B) at (2,0) {};
    \node[dot] (C) at (4,0) {};
    \node[anchor=north] at (C) {\text{$\widetilde{v}$}};
    \node[dot] (D) at (6,0) {};
    \node[dot] (E) at (8,0) {};

    % Midpoints
    \coordinate (AB) at ($(A)!0.5!(B)$);
    \coordinate (BC) at ($(B)!0.5!(C)$);
    \coordinate (CD) at ($(C)!0.5!(D)$);
    \coordinate (DE) at ($(D)!0.5!(E)$);

    % Arrows from vertices to midpoints (flipped to top)

    \draw[->, shorten >=4pt, shorten <=4pt] (A) to (AB);

    \draw[->, shorten >=4pt, shorten <=4pt] (B) to (AB);
    \draw[->, shorten >=4pt, shorten <=4pt] (B) to (BC);

    \draw[->, shorten >=4pt, shorten <=4pt] (C) to (BC);
    \draw[->, shorten >=4pt, shorten <=4pt] (C) to (CD);

    \draw[->, shorten >=4pt, shorten <=4pt] (D) to (CD);
    \draw[->, shorten >=4pt, shorten <=4pt] (D) to (DE);

    \draw[->, shorten >=4pt, shorten <=4pt] (E) to (DE);

\end{tikzpicture}

    \end{center}
$\widetilde{S}_{\geq \widetilde{v}}$ is as follows
\begin{center}
\begin{minipage}{.4\textwidth}
    \begin{center}
\begin{tikzpicture}[scale=1.5, dot/.style={circle, fill, inner sep=2pt}]
    % Nodes and line
    \node (A) at (0,0) {};
    \node (B) at (2,0) {};
    \node (C) at (4,0) {};
    \draw[green,thick] (A) -- (B) -- (C);

    % Dots
    \node[dot] (B) at (2,0) {};
    \node[anchor=north] at (B) {\text{$\widetilde{v}$}};
    
    %Circles
    \node[draw, circle, inner sep=1.5pt] at (0,0) {};
    \node[draw, circle, inner sep=1.5pt] at (4,0) {};
    
    % Midpoints
    \coordinate (AB) at ($(A)!0.5!(B)$);
    \coordinate (BC) at ($(B)!0.5!(C)$);

    % Arrows from vertices to midpoints (flipped to top)

    \draw[->, shorten >=4pt, shorten <=4pt] (B) to (AB);
    \draw[->, shorten >=4pt, shorten <=4pt] (B) to (BC);

\end{tikzpicture}
    \end{center}
\end{minipage}
\begin{minipage}{.4\textwidth}
    \begin{center}
\begin{tikzpicture}[scale=1.5, dot/.style={circle, fill, inner sep=2pt}]
    % Nodes and line
    \node (A) at (0,0) {};
    \node (B) at (2,0) {};
    \node (C) at (4,0) {};

    % Dots
    \node[dot] (B) at (2,0) {};
    \node[anchor=north] at (B) {\text{$\widetilde{v}$}};
    
    % Midpoints
    \coordinate[dot] (AB) at ($(A)!0.5!(B)$);
    \coordinate[dot] (BC) at ($(B)!0.5!(C)$);

    % Arrows from vertices to midpoints (flipped to top)

    \draw[->, shorten >=4pt, shorten <=4pt] (B) to (AB);
    \draw[->, shorten >=4pt, shorten <=4pt] (B) to (BC);

\end{tikzpicture}
    \end{center}
\end{minipage}
\end{center}
which is of course locally Cohen-Macaulay. 
\end{ex}
\begin{ex}
\label{ex: universal cover stratification of F_1}
For the stratification of the $2$-dimensional torus in Example \ref{ex: stratification of F_1}, the stratification of the universal cover $\mathbb R^2$ together with its entrance paths are displayed in the following figures:
    
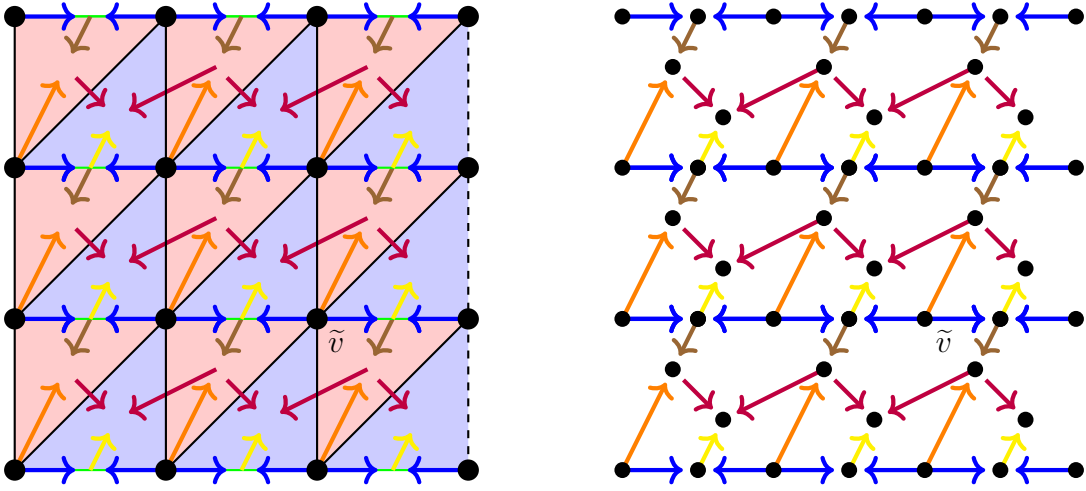
\begin{center}
    %\begin{minipage}{\textwidth}  
\begin{tikzpicture}[scale=2]

% Grid size and spacing
\pgfmathsetmacro{\spacing}{1} % Define spacing as a numerical value

% Loop through all squares in the grid
\foreach \i in {0,...,2} {
    \foreach \j in {0,...,2} {
        % Compute centroids of left and right triangles
        \pgfmathsetmacro{\cxleft}{(\j*\spacing + \j*\spacing + (\j+1)*\spacing)/3}
        \pgfmathsetmacro{\cyleft}{(\i*\spacing + (\i+1)*\spacing + (\i+1)*\spacing)/3}
        \pgfmathsetmacro{\cxright}{(\j*\spacing + (\j+1)*\spacing + (\j+1)*\spacing)/3}
        \pgfmathsetmacro{\cyright}{(\i*\spacing + \i*\spacing + (\i+1)*\spacing)/3}

        % Compute midpoints of horizontal sides
        \pgfmathsetmacro{\midxbottom}{\j*\spacing + 0.5*\spacing}
        \pgfmathsetmacro{\midybottom}{\i*\spacing}
        \pgfmathsetmacro{\midxtop}{\j*\spacing + 0.5*\spacing}
        \pgfmathsetmacro{\midytop}{\i*\spacing + \spacing}

        % Draw left triangle (light red)
        \fill[red!20] 
            (\j*\spacing, \i*\spacing) -- 
            (\j*\spacing, \i*\spacing+\spacing) -- 
            (\j*\spacing+\spacing, \i*\spacing+\spacing) -- cycle;

        % Draw right triangle (light blue)
        \fill[blue!20] 
            (\j*\spacing, \i*\spacing) -- 
            (\j*\spacing+\spacing, \i*\spacing) -- 
            (\j*\spacing+\spacing, \i*\spacing+\spacing) -- cycle;

        % Draw black diagonal line
        \draw[black, thick] 
            (\j*\spacing, \i*\spacing) -- (\j*\spacing+\spacing, \i*\spacing+\spacing);
        % Connect square vertices horizontally with green lines
          \draw[green, thick] ({\j*\spacing}, {\i*\spacing}) -- ({(\j+1)*\spacing}, {\i*\spacing});
        % Connect square vertices horizontally with green lines
          \draw[green, thick] ({\j*\spacing}, {3*\spacing}) -- ({(\j+1)*\spacing}, {3*\spacing});

        % Draw ultra-thick  blue arrows from vertices to midpoints
        % Bottom-left vertex
        \draw[blue, ultra thick, ->, shorten >=0.2cm] (\j*\spacing, \i*\spacing) -- (\midxbottom, \midybottom); % Bottom midpoint
        % Top-left vertex
        \draw[blue, ultra thick, ->, shorten >=0.2cm] (\j*\spacing, \i*\spacing+\spacing) -- (\midxtop, \midytop); % Top midpoint
        % Bottom-right vertex
        \draw[blue, ultra thick, ->, shorten >=0.2cm] (\j*\spacing+\spacing, \i*\spacing) -- (\midxbottom, \midybottom); % Bottom midpoint
        % Top-right vertex
        \draw[blue, ultra thick, ->, shorten >=0.2cm] (\j*\spacing+\spacing, \i*\spacing+\spacing) -- (\midxtop, \midytop); % Top midpoint

        % Draw ultra-thick arrows from square vertices to triangle centers
        % Bottom-left vertex to both centers
        \draw[orange, ultra thick, ->, shorten >=0.2cm] (\j*\spacing, \i*\spacing) -- (\cxleft, \cyleft); % To left triangle
        % \draw[red, ultra thick, ->, shorten >=0.2cm] (\j*\spacing, \i*\spacing) -- (\cxright, \cyright); % To right triangle

        % Top-left vertex to left center
        % \draw[orange, ultra thick, ->, shorten >=0.2cm] (\j*\spacing, \i*\spacing+\spacing) -- (\cxleft, \cyleft);

        % Bottom-right vertex to right center
        % \draw[red, ultra thick, ->, shorten >=0.2cm] (\j*\spacing+\spacing, \i*\spacing) -- (\cxright, \cyright);

        % Top-right vertex to both centers
        % \draw[orange, ultra thick, ->, shorten >=0.2cm] (\j*\spacing+\spacing, \i*\spacing+\spacing) -- (\cxleft, \cyleft); % To left triangle
        % \draw[red, ultra thick, ->, shorten >=0.2cm] (\j*\spacing+\spacing, \i*\spacing+\spacing) -- (\cxright, \cyright); % To right triangle

        % Draw ultra-thick purple arrow connecting the red triangle center to the blue triangle center
        \draw[purple, ultra thick, ->, shorten >=0.2cm, shorten <=0.2cm] (\cxleft, \cyleft) -- (\cxright, \cyright);

        % Draw ultra-thick purple arrow from red triangle center to the blue triangle on the left
        \ifnum\j>0 % Ensure the arrow stays within bounds
            \pgfmathsetmacro{\cxblueleft}{((\j-1)*\spacing + (\j)*\spacing + (\j)*\spacing)/3}
            \pgfmathsetmacro{\cyblueleft}{((\i)*\spacing + (\i)*\spacing + (\i+1)*\spacing)/3}
            \draw[purple, ultra thick, ->, shorten >=0.2cm] (\cxleft, \cyleft) -- (\cxblueleft, \cyright);
        \fi

        % Draw ultra-thick dark brown arrows from midpoints of horizontal sides to triangle centers
        \draw[yellow, ultra thick, ->, shorten >=0.2cm] (\midxbottom, \midybottom) -- (\cxright, \cyright); % Bottom midpoint to right triangle center
        \draw[brown!80!black, ultra thick, ->, shorten >=0.2cm] (\midxtop, \midytop) -- (\cxleft, \cyleft); % Top midpoint to left triangle center
    }
}

% Draw solid vertical lines
\foreach \j in {1,2} {
    \draw[thick] 
        (\j*\spacing, 0) -- (\j*\spacing, 3*\spacing);
}

% Draw left boundary as solid
\draw[thick] 
    (0, 0) -- (0, 3*\spacing);

% Draw right boundary as dashed
\draw[thick,dashed]
    (3*\spacing, 0) -- (3*\spacing, 3*\spacing);

% Add bullets (filled circles) only at vertices
\foreach \i in {0,...,3} {
    \foreach \j in {0,...,3} {
        \fill[black] (\j*\spacing, \i*\spacing) circle (2pt); % Filled black circle only at vertices
    }
}
    \node[anchor=north west] at (2*\spacing, \spacing) {\text{$\widetilde{v}$}};
\end{tikzpicture}
%\end{minipage}
\hspace{1.5 cm}
%\begin{minipage}{.1\textwidth}
\begin{tikzpicture}[scale=2]

% Grid size and spacing
\pgfmathsetmacro{\spacing}{1} % Define spacing as a numerical value

% Loop through all squares in the grid
\foreach \i in {0,...,2} {
    \foreach \j in {0,...,2} {
        % Compute centroids of left and right triangles
        \pgfmathsetmacro{\cxleft}{(\j*\spacing + \j*\spacing + (\j+1)*\spacing)/3}
        \pgfmathsetmacro{\cyleft}{(\i*\spacing + (\i+1)*\spacing + (\i+1)*\spacing)/3}
        \pgfmathsetmacro{\cxright}{(\j*\spacing + (\j+1)*\spacing + (\j+1)*\spacing)/3}
        \pgfmathsetmacro{\cyright}{(\i*\spacing + \i*\spacing + (\i+1)*\spacing)/3}

        % Compute midpoints of horizontal sides
        \pgfmathsetmacro{\midxbottom}{\j*\spacing + 0.5*\spacing}
        \pgfmathsetmacro{\midybottom}{\i*\spacing}
        \pgfmathsetmacro{\midxtop}{\j*\spacing + 0.5*\spacing}
        \pgfmathsetmacro{\midytop}{\i*\spacing + \spacing}

        % Draw left triangle (light red)
        %\fill[red!20] 
            %(\j*\spacing, \i*\spacing) -- 
            %(\j*\spacing, \i*\spacing+\spacing) -- 
            %(\j*\spacing+\spacing, \i*\spacing+\spacing) -- cycle;

        % Draw right triangle (light blue)
        %\fill[blue!20] 
            %(\j*\spacing, \i*\spacing) -- 
            %(\j*\spacing+\spacing, \i*\spacing) -- 
            %(\j*\spacing+\spacing, \i*\spacing+\spacing) -- cycle;

        % Draw black diagonal line
        %\draw[black, thick] 
            %(\j*\spacing, \i*\spacing) -- (\j*\spacing+\spacing, \i*\spacing+\spacing);

        % Draw ultra-thick blue arrows from vertices to midpoints
        % Bottom-left vertex
        \draw[blue, ultra thick, ->, shorten >=0.2cm] (\j*\spacing, \i*\spacing) -- (\midxbottom, \midybottom); % Bottom midpoint
        % Top-left vertex
        \draw[blue, ultra thick, ->, shorten >=0.2cm] (\j*\spacing, \i*\spacing+\spacing) -- (\midxtop, \midytop); % Top midpoint
        % Bottom-right vertex
        \draw[blue, ultra thick, ->, shorten >=0.2cm] (\j*\spacing+\spacing, \i*\spacing) -- (\midxbottom, \midybottom); % Bottom midpoint
        % Top-right vertex
        \draw[blue, ultra thick, ->, shorten >=0.2cm] (\j*\spacing+\spacing, \i*\spacing+\spacing) -- (\midxtop, \midytop); % Top midpoint

        % Draw ultra-thick arrows from square vertices to triangle centers
        % Bottom-left vertex to both centers
        \draw[orange, ultra thick, ->, shorten >=0.2cm] (\j*\spacing, \i*\spacing) -- (\cxleft, \cyleft); % To left triangle

        % Draw ultra-thick purple arrow connecting the red triangle center to the blue triangle center
        \draw[purple, ultra thick, ->, shorten >=0.2cm, shorten <=0.2cm] (\cxleft, \cyleft) -- (\cxright, \cyright);

        % Draw ultra-thick purple arrow from red triangle center to the blue triangle on the left
        \ifnum\j>0 % Ensure the arrow stays within bounds
            \pgfmathsetmacro{\cxblueleft}{((\j-1)*\spacing + (\j)*\spacing + (\j)*\spacing)/3}
            \pgfmathsetmacro{\cyblueleft}{((\i)*\spacing + (\i)*\spacing + (\i+1)*\spacing)/3}
            \draw[purple, ultra thick, ->, shorten >=0.2cm] (\cxleft, \cyleft) -- (\cxblueleft, \cyright);
        \fi

        % Draw ultra-thick dark brown arrows from midpoints of horizontal sides to triangle centers
        \draw[yellow, ultra thick, ->, shorten >=0.2cm] (\midxbottom, \midybottom) -- (\cxright, \cyright); % Bottom midpoint to right triangle center
        \draw[brown!80!black, ultra thick, ->, shorten >=0.2cm] (\midxtop, \midytop) -- (\cxleft, \cyleft); % Top midpoint to left triangle center
         % Add bullets for horizontal line midpoints
        \fill[black] (\midxbottom, \midybottom) circle (1.5pt); % Bottom horizontal line midpoint
        \fill[black] (\midxtop, \midytop) circle (1.5pt);       % Top horizontal line midpoint

        % Add bullets for triangle centers
        \fill[black] (\cxleft, \cyleft) circle (1.5pt);  % Center of the left triangle
        \fill[black] (\cxright, \cyright) circle (1.5pt); % Center of the right triangle
    }
}

% Add bullets (filled circles) only at vertices
\foreach \i in {0,...,3} {
    \foreach \j in {0,...,3} {
        \fill[black] (\j*\spacing, \i*\spacing) circle (1.5 pt); % Filled black circle only at vertices
    }
}
    \node[anchor=north west] at (2*\spacing, \spacing) {\text{$\widetilde{v}$}};
\end{tikzpicture}
%\end{minipage}
    \end{center}
The space $\widetilde{S}_{\geq \widetilde{v}}$ is pictured below.
\\
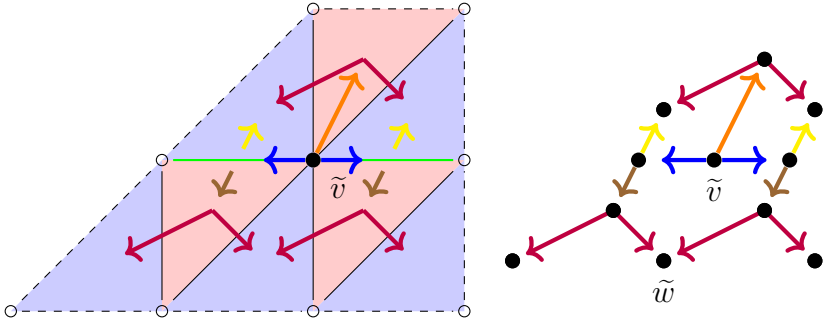
\begin{minipage}{.4\textwidth}
\begin{center}
\begin{tikzpicture}[scale=1.0, dot/.style={circle, fill, inner sep=2pt}]

% Draw grid points
\node (A) at (0,0) {};
\node (B) at (2,0) {};
\node (C) at (4,0) {};
\node (D) at (6,0) {};
\node (E) at (2,2) {};
\node[dot] (F) at (4,2) {};
\node (G) at (6,2) {};
\node (H) at (4,4) {};
\node (I) at (6,4) {};

% Draw triangles with alternating colors
\fill[blue!20] (A.center) -- (B.center) -- (E.center) -- cycle; % Bottom-left triangle
\fill[red!20]  (B.center) -- (E.center) -- (F.center) -- cycle;
\fill[blue!20]  (B.center) -- (C.center) -- (F.center) -- cycle;
\fill[red!20]  (C.center) -- (F.center) -- (G.center) -- cycle;
\fill[blue!20] (C.center) -- (D.center) -- (G.center) -- cycle;

\fill[blue!20]  (E.center) -- (F.center) -- (H.center) -- cycle;
\fill[red!20]  (F.center) -- (H.center) -- (I.center) -- cycle;
\fill[blue!20] (F.center) -- (G.center) -- (I.center) -- cycle;

% Draw grid lines
\draw[dashed] (H) -- (I);
\draw[dashed] (A) -- (E) -- (H);
\draw (B) -- (F) -- (I);
\draw (C) -- (G);
\draw[green, thick] (E) -- (F) -- (G);
\draw (B) -- (E);
\draw (C) -- (F) -- (H);
\draw[dashed] (D) -- (G) -- (I);
\draw[dashed] (A) -- (B) -- (C) -- (D);

% Draw midpoints for each triangle
\node (MiddleAB) at ($(A)!0.5!(B)$) {};
\node (MiddleEF) at ($(E)!0.5!(F)$) {};
\node (MiddleCB) at ($(C)!0.5!(B)$) {};
\node (MiddleFG) at ($(F)!0.5!(G)$) {};
\node (MiddleCD) at ($(C)!0.5!(D)$) {};
\node (MiddleHI) at ($(H)!0.5!(I)$) {};

% Draw centroids for each triangle
\node (CentroidABE) at (4/3, 2/3) {}; % Centroid of triangle ABE
\node (CentroidBEF) at (8/3, 4/3) {}; % Centroid of triangle BEF
\node (CentroidCFG) at (14/3, 4/3) {}; % Centroid of triangle CFG
\node (CentroidCDG) at (16/3, 2/3) {}; % Centroid of triangle CDG
\node (CentroidFHI) at (14/3, 10/3) {}; % Centroid of triangle FHI
\node (CentroidFGI) at (16/3, 8/3) {}; % Centroid of triangle FGI
\node (CentroidBCF) at (10/3, 2/3) {}; % Centroid of triangle BCF
\node (CentroidEFH) at (10/3, 8/3) {}; % Centroid of triangle EFH

% \draw[red,->, ultra thick, ->, shorten >=0.2cm] (F) -- (CentroidBCF.center);
% \draw[red,->, ultra thick, ->, shorten >=0.2cm] (F) -- (CentroidFGI.center);
% \draw[red,->, ultra thick, ->, shorten >=0.2cm] (F) -- (CentroidEFH.center);
% \draw[orange,->, ultra thick, ->, shorten >=0.2cm] (F) -- (CentroidBEF.center);
% \draw[orange,->, ultra thick, ->, shorten >=0.2cm] (F) -- (CentroidCFG.center);
\draw[orange,->, ultra thick, ->, shorten >=0.2cm] (F) -- (CentroidFHI.center);
\draw[blue,->, ultra thick, ->, shorten >=0.2cm] (F) -- (MiddleEF);
\draw[blue,->, ultra thick, ->, shorten >=0.2cm] (F) -- (MiddleFG);
\draw[purple,->, ultra thick, ->, shorten >=0.2cm] (CentroidFHI.center) -- (CentroidEFH.center);
\draw[purple,->, ultra thick, ->, shorten >=0.2cm] (CentroidFHI.center) -- (CentroidFGI.center);
\draw[purple,->, ultra thick, ->, shorten >=0.2cm] (CentroidBEF.center) -- (CentroidBCF.center);
\draw[purple,->, ultra thick, ->, shorten >=0.2cm] (CentroidBEF.center) -- (CentroidABE.center);
\draw[purple,->, ultra thick, ->, shorten >=0.2cm] (CentroidCFG.center) -- (CentroidBCF.center);
\draw[purple,->, ultra thick, ->, shorten >=0.2cm] (CentroidCFG.center) -- (CentroidCDG.center);

\draw[yellow,->, ultra thick, ->, shorten >=0.2cm] (MiddleEF) -- (CentroidEFH.center);
\draw[brown!80!black,->, ultra thick, ->, shorten >=0.2cm] (MiddleEF) -- (CentroidBEF.center);
\draw[yellow,->, ultra thick, ->, shorten >=0.2cm] (MiddleFG) -- (CentroidFGI.center);
\draw[brown!80!black,->, ultra thick, ->, shorten >=0.2cm] (MiddleFG) -- (CentroidCFG.center);

% Draw grid circles
\node[draw, circle, inner sep=1.5pt] at (0,0) {};
\node[draw, circle, inner sep=1.5pt] at (2,0) {};
\node[draw, circle, inner sep=1.5pt] at (4,0) {};
\node[draw, circle, inner sep=1.5pt] at (6,0) {};
\node[draw, circle, inner sep=1.5pt] at (2,2) {};
\node[dot, label=below right:\text{$\widetilde{v}$}] (F) at (4,2) {};
\node[draw, circle, inner sep=1.5pt] at (6,2) {};
\node[draw, circle, inner sep=1.5pt] at (4,4) {};
\node[draw, circle, inner sep=1.5pt] at (6,4) {};

\end{tikzpicture}
\end{center}
\end{minipage}
\begin{minipage}{.2\textwidth}
\begin{center}
\begin{tikzpicture}[scale=1.0, dot/.style={circle, fill, inner sep=2pt}]

% Draw grid points

\node[dot, label=below:\text{$\widetilde{v}$}] (F) at (4,2) {};

% Draw midpoints for each triangle
\node (MiddleAB) at ($(A)!0.5!(B)$) {};
\node[dot, label] (MiddleEF) at ($(E)!0.5!(F)$) {};
\node (MiddleCB) at ($(C)!0.5!(B)$) {};
\node[dot] (MiddleFG) at ($(F)!0.5!(G)$) {};
\node (MiddleCD) at ($(C)!0.5!(D)$) {};
\node (MiddleHI) at ($(H)!0.5!(I)$) {};

% Draw centroids for each triangle
\node[dot] (CentroidABE) at (4/3, 2/3) {}; % Centroid of triangle ABE
\node[dot] (CentroidBEF) at (8/3, 4/3) {}; % Centroid of triangle BEF
\node[dot] (CentroidCFG) at (14/3, 4/3) {}; % Centroid of triangle CFG
\node[dot] (CentroidCDG) at (16/3, 2/3) {}; % Centroid of triangle CDG
\node[dot] (CentroidFHI) at (14/3, 10/3) {}; % Centroid of triangle FHI
\node[dot] (CentroidFGI) at (16/3, 8/3) {}; % Centroid of triangle FGI
\node[dot, label=below:\text{$\widetilde{w}$}] (CentroidBCF) at (10/3, 2/3) {}; % Centroid of triangle BCF
\node[dot] (CentroidEFH) at (10/3, 8/3) {}; % Centroid of triangle EFH
%\node[dot] (CentroidHIJ) at (16/3, 14/3) {}; % Centroid of triangle HIJ

%\draw[red,->, ultra thick, ->, shorten >=0.2cm] (F) -- (CentroidBCF.center);
%\draw[red,->, ultra thick, ->, shorten >=0.2cm] (F) -- (CentroidFGI.center);
%\draw[red,->, ultra thick, ->, shorten >=0.2cm] (F) -- (CentroidEFH.center);
%\draw[orange,->, ultra thick, ->, shorten >=0.2cm] (F) -- (CentroidBEF.center);
%\draw[orange,->, ultra thick, ->, shorten >=0.2cm] (F) -- (CentroidCFG.center);
\draw[orange,->, ultra thick, ->, shorten >=0.2cm] (F) -- (CentroidFHI.center);
\draw[blue,->, ultra thick, ->, shorten >=0.2cm] (F) -- (MiddleEF);
\draw[blue,->, ultra thick, ->, shorten >=0.2cm] (F) -- (MiddleFG);
\draw[purple,->, ultra thick, ->, shorten >=0.2cm] (CentroidFHI.center) -- (CentroidEFH.center);
\draw[purple,->, ultra thick, ->, shorten >=0.2cm] (CentroidFHI.center) -- (CentroidFGI.center);
\draw[purple,->, ultra thick, ->, shorten >=0.2cm] (CentroidBEF.center) -- (CentroidBCF.center);
\draw[purple,->, ultra thick, ->, shorten >=0.2cm] (CentroidBEF.center) -- (CentroidABE.center);
\draw[purple,->, ultra thick, ->, shorten >=0.2cm] (CentroidCFG.center) -- (CentroidBCF.center);
\draw[purple,->, ultra thick, ->, shorten >=0.2cm] (CentroidCFG.center) -- (CentroidCDG.center);

\draw[yellow,->, ultra thick, ->, shorten >=0.2cm] (MiddleEF) -- (CentroidEFH.center);
\draw[brown!80!black,->, ultra thick, ->, shorten >=0.2cm] (MiddleEF) -- (CentroidBEF.center);
\draw[yellow,->, ultra thick, ->, shorten >=0.2cm] (MiddleFG) -- (CentroidFGI.center);
\draw[brown!80!black,->, ultra thick, ->, shorten >=0.2cm] (MiddleFG) --  (CentroidCFG.center);

% Draw centroids for each triangle
\node[dot] (CentroidABE) at (4/3, 2/3) {}; % Centroid of triangle ABE
\node[dot] (CentroidBEF) at (8/3, 4/3) {}; % Centroid of triangle BEF
\node[dot] (CentroidCFG) at (14/3, 4/3) {}; % Centroid of triangle CFG
\node[dot] (CentroidCDG) at (16/3, 2/3) {}; % Centroid of triangle CDG
\node[dot] (CentroidFHI) at (14/3, 10/3) {}; % Centroid of triangle FHI
\node[dot] (CentroidFGI) at (16/3, 8/3) {}; % Centroid of triangle FGI
\node[dot] (CentroidBCF) at (10/3, 2/3) {}; % Centroid of triangle BCF
\node[dot] (CentroidEFH) at (10/3, 8/3) {}; % Centroid of triangle EFH
%\node[dot] (CentroidHIJ) at (16/3, 14/3) {}; % Centroid of triangle HIJ
\end{tikzpicture}
\end{center}
\end{minipage}
\vspace{5mm}

\noindent This is not locally Cohen-Macaulay since $(\widetilde{v},\widetilde{w})$ consists of two disconnected lines (pictured in brown).  
%In fact, in this case $k\widetilde{S}$ is not quadratic (for the same reason). 
\end{ex}
\begin{df}[\protect{\cite[Definition 4.8, p. 11]{dj}}]
    A stratification $f:X\to S$ where $S$ is a finite poset is called \newterm{exceptional} if $X_a$ is contractible for all $a\in S$ and 
    $$a\leq b \iff \overline{X_a}\cap X_b\ne \emptyset$$
\end{df}
\begin{df}[\protect{\cite[Definition 4.12, p. 13]{dj}}]
Let $f:X\to S$ be an exceptional stratification space admitting a universal cover $\widetilde{X}$. Let $\widetilde{f}:\widetilde{Y}\to \widetilde{S}$ be the induced stratification on the universal cover. For a point $\widetilde{x}\in \widetilde{X}$, we define the \newterm{entrance space} at $\widetilde{x}$ to be the subspace 
$$\widetilde{X}_{\Ent}(\widetilde{x}):=\{y\in \widetilde{Y}: \exists \widetilde{\gamma}\in \Ent_{\widetilde{S}}(\widetilde{Y}) \text{ with $\widetilde{\gamma}(0)=\widetilde{x}$, $\widetilde{\gamma}(1)=y$}\}$$
\end{df}
\begin{df}[\protect{\cite[Definition 4.13, p. 13]{dj}}]
    An exceptional stratification of $X$ is called simple if for all $\widetilde{x}\in \widetilde{X}$ the entrance space at $\widetilde{x}$ is contractible and for all $x$, $x'$ the difference $\widetilde{X}_{\Ent}(\widetilde{x})\setminus\widetilde{X}_{\Ent}(\widetilde{x'})$ is contractible whenever it is nonempty. 
\end{df}
\begin{thm}[\protect{\cite[Definition 4.23, p. 18]{dj}}]
\label{thm: exodromy}
    Let $f:X\to S$ be a simple stratification. Then 
    $$Sh_S(X)\cong A_S(X)^{op}-mod$$
\end{thm}
\begin{ex}
\label{ex: tree stratification of a poset}
    Let $P$ be a poset with order complex $K(P)$. Define $f:K(P)\to P$ by sending $x\in \intt([a_1<...<a_k])$ to $a_k$. This is a simple stratification. Hence, 
    $$Sh_P(K(P))\cong A_P(K(P))^{op}-mod$$
    Although $K(P)$ is not necessarily simply-connected, the canonical functor 
    $$F:\Ent_P(K(P))\to P$$
    $$a\mapsto a$$
    $$[\gamma]\mapsto \gamma(0)<\gamma(1)$$
    is an isomorphism since for any $a<b$ in $P$, any chain that contains $a<b$ in $P$, has a deformation retraction to the 1-simplex $[a<b]$ in $K(P)$. Therefore, $A_P(K(P))\cong kP$ and 
    $$Sh_P(K(P))\cong (kP)^{op}-mod$$
\end{ex}
\subsubsection{Koszulity of Homotopy Path Algebras of Bondal-Thomsen Type}
We finish this section by applying our results to homotopy path algebras of Bondal-Thomsen type (for details see \cite{dj} and \cite{bondal2006derived}).

Let $M$ be a free subgroup of $\mathbb Z^{n+k}$ of rank $n$ and $\widehat G=\mathbb Z^{n+k}/M$. Consider the short exact sequence 
$$0\to M \xrightarrow{i} \mathbb Z^{n+k}\xrightarrow{\mu} \widehat G\to 0$$
Since $\mathbb R$ is a flat $\mathbb Z$-module, tensoring with $\mathbb R$ gives the following short exact sequence 
$$0\to M_{\mathbb R} \xrightarrow{i_{\mathbb R}} \mathbb R^{n+k}\xrightarrow{\mu_{\mathbb R}} \widehat G_{\mathbb R}\to 0$$
We denote the standard basis of $\mathbb R^{n+k}$ by $\{D_1,...,D_{n+k}\}$, the elements of $\mathbb Z^{n+k}$ by $D$ and the elements of $\widehat G$ by $[D]$. 

The intersection of the cube stratification 
$$f:\mathbb R^{n+k}\to \mathbb Z^{n+k}$$
$$\sum_ia_iD_i\mapsto -\sum_i\lfloor a_i\rfloor D_i$$
with
$\mu^{-1}(0)= i_{\mathbb R}(M_{\mathbb R})\cong M_{\mathbb R}\cong \mathbb R^n$ provides a $\mathbb Z^{n+k}$-stratification of $\mu^{-1}(0)$
$$\widetilde{\Phi}:\mu^{-1}(0)\to \mathbb Z^{n+k}$$
$$x\mapsto f(x)$$
So, 
$$\mathbb R^n_{D}=(-D+[0,1)^{n+k})\cap \mu_{\mathbb R}^{-1}(0)$$
Assuming $\mu_{\mathbb R}(\mathbb R^{n+k}_{\geq 0})$ is a strongly convex cone in $\widehat G_{\mathbb R}$ gives $\widehat G$ a poset structure via 
$$[E]\leq [D] \iff [D-E]\in \mu_{\mathbb R}(\mathbb R^{n+k}_{\geq 0})$$
Since $\widetilde{\Phi}(m+x)=-m+\widetilde{\Phi}(x)$ for all $m\in M$, the \newterm{Bondal-Thomsen map} 
$$\Phi:\mathbb T^n\to \widehat G$$
$$\sum_i a_iD_i+ M\mapsto \mu(-\sum_i \lfloor a_i\rfloor D_i)$$
defines a $\Image \Phi$-stratification on $\mathbb T^n=\mu_{\mathbb R}^{-1}(0)/M=M_{\mathbb R}/M$. 
\begin{comment}

\begin{center}
	       \begin{tikzcd}
                    0\arrow{r}
                    &M_{\mathbb R}\arrow{r}{i_{\mathbb R}}
                    &\mathbb R^{n+k}\arrow{d}{-\lfloor\rfloor}\arrow{r}{\mu_{\mathbb R}}
                    &\widehat G_{\mathbb R}\arrow{r}
                    & 0\\
                    0\arrow{r}
                    &M\arrow{r}{i}
                    &\mathbb Z^{n+k}\arrow{r}{\mu}
                    &\widehat G\arrow{r}
                    & 0
	       \end{tikzcd}     
        \end{center} 
\end{comment}
Furthermore, we have an isomorphism
$$\widetilde{\Image \Phi}\xrightarrow{\sim}\Image\widetilde{\Phi}$$
$$(m,\Phi(x))\mapsto \widetilde{\Phi}(m+x)$$
As $\Image\widetilde{\Phi}\subset \mathbb Z^{n+k}$, we denote the elements of $\widetilde{\Image \Phi}$ by $D$. 
\begin{df}[\protect{\cite[Definition 5.10, p. 25]{dj}}]
    The \newterm{Bondal-Thomsen HPA} is defined as the category algebra $k\Ent_{\Image \Phi}(\mathbb T^n)$. 
\end{df}
As proved in \cite{dj}, Bondal-Thomsen HPAs can be interpreted as skew categories (see Definition \ref{def: Grothendieck construction} and Example \ref{ex: hpas are C_S}) in the following way:
\begin{prop}[\protect{\cite[Corollary 5.11, p. 25]{dj}}]
\label{prop: BR hpas are C_S}
Let $\mathcal C$ be a $\widehat G$-graded category with one object, $\Mor(\mathcal C)=\{x^{\underline{m}}:\underline{m}\in \mathbb N^{n+k}\}$ and $\deg(x^{\underline{m}})=\mu(\underline{m})$. Let $S=\Image \Phi$. Then, we have an equivalence of categories
    $$\mathcal C_S\simeq \Ent_{\Image \Phi}(\mathbb T^n)$$
\end{prop}

\begin{comment}
\begin{prop}
  $$\mathbb R^n_{D}=(-D+[0,1)^{n+k})\cap \mu_{\mathbb R}^{-1}(0)$$
for any $D\in \widetilde{\Image \Phi}$.   
\end{prop}
\begin{proof}
\begin{align*}
                x\in \mathbb R^n_{D}&\iff x\in \mu_{\mathbb R}^{-1}(0) \text{ and } \widetilde{\Phi}(x)=D
                &\text{By definition and $\widetilde{\Image \Phi}=\Image\widetilde{\Phi}$}\\
                & \iff x\in \mu_{\mathbb R}^{-1}(0) \text{ and } x\in (-D+[0,1)^{n+k})
                & \text{By definition of $\widetilde \Phi$}
            \end{align*}
\end{proof}
\end{comment}
\begin{comment}
    \begin{prop}
\label{prop: BR almost discrete fibration}
    We have an almost discrete fibration 
    $$\Psi:\Ent_{\widetilde{\Image\Phi}}(\mathbb R^n)\to \Ent_{\Image\Phi}(\mathbb T^n)$$
\end{prop}
\begin{proof}
    Follows immediately from part (1) of Proposition \ref{prop: koszulity and universal cover}.
\end{proof}
\end{comment}

\begin{lemma}
\label{lemma: straight lines are entrance path}
    Let $D\in \widetilde{\Image \Phi}$ and $y\in \mathbb R^n_{\Ent}(x_D)$. Then any straight line from $x_D$ to $y$ in $\mathbb{R}^n$ is an entrance path.
\end{lemma}
\begin{proof}
    %This is part of the proof of \cite[Proposition 5.5]{dj}. 
    The straight line from $x_D$ to $y$ in $\mathbb{R}^{n+k}$ is an entrance path. Since $\mu_{\mathbb R}$ is linear and $x_D,y\in \mathbb R^n_{\Ent}(x_D)\subset\mu_{\mathbb R}^{-1}(0)\cong \mathbb R^n$, $\mu_{\mathbb R}^{-1}(0)$ contains this line.
\end{proof}
\begin{prop}
\label{prop: local simple stratification}
    For any open interval $(D,E)$ in $\widetilde{\Image\phi}$, 
    $$\widetilde{\Phi}|_{\mathbb R^n_{(D,E)}}:\mathbb R^n_{(D,E)}\to (D,E)$$
    is a simple stratification.
\end{prop}
\begin{proof}
     Using Lemma \ref{lemma: straight lines are entrance path} and following the definitions, it follows that 
     $$\mathbb R^n_{\Ent}(x_D)_{<E}:=\mathbb R^n_{\Ent}(x_D)\cap \mathbb R^n_{<E}$$
     and 
     $$\mathbb R^n_{\Ent}(x_D)_{<E}\setminus \mathbb R^n_{\Ent}(x_{D'})_{<E}$$
     are star-shaped and hence contractible for all $D,E\in \widetilde{\Image\phi}$. Now, since $\mathbb R^n_{\Ent}(x_D)_{<E}$ is contractible, any point $y$ in $\mathbb R^n_{\Ent}(x_D)_{<E}$ can be characterized by a unique homotopy class of entrance path from $x_D$ to $y$. Therefore, for all $y\in Y:=\mathbb R^n_{(D,E)}$, the restriction of the universal covering map $p:\widetilde{Y}\to Y$ to $\widetilde{Y}_{\Ent}(\widetilde{y})$ induces a homeomorphism from $\widetilde{Y}_{\Ent}(\widetilde{y})$ to $\mathbb R^n_{\Ent}(y)_{<E}$. Thus, for all $y,y'\in \mathbb R^n_{(D,E)}$, $\widetilde{Y}_{\Ent}(\widetilde{y})$ and $\widetilde{Y}_{\Ent}(\widetilde{y})\setminus\widetilde{Y}_{\Ent}(\widetilde{y'})$ are contractible. Hence $$\widetilde{\Phi}|_{\mathbb R^n_{(D,E)}}:\mathbb R^n_{(D,E)}\to (D,E)$$
    is a simple stratification.
\end{proof}
\begin{lemma}
\label{lemma: ent cat over interval}
    For any open interval $(D,E)$ in $\widetilde{\Image\phi}$, the natural functor 
    $$\Phi:\Ent_{(D,E)}(\mathbb R^n_{(D,E)})\to (D,E)$$
    $$x_F\mapsto F$$
    is an isomorphism of categories.
\end{lemma}
\begin{proof}
    Let $\gamma$ and $\gamma'$ be two entrance path from $x_F$ to $x_G$ in $\mathbb R^n_{(D,E)}$. Then, they are two paths from $x_F$ to $x_G$ in $\mathbb R^n_{[F,G]}\subset \mathbb R^n_{(D,E)}$. Since $\mathbb R^n_{[F,G]}=\mathbb{R}^{n+k}_{[F,G]}\cap \mu^{-1}_{\mathbb R}(0)$ and $\mu_{\mathbb R}$ is linear, for any $y\in \mathbb R^n_{[F,G]}$,  $\mathbb R^n_{[F,G]}$ contains the straight line from $x_F$ to $y$. Therefore, $\mathbb R^n_{[F,G]}$ is star-shaped and so contractible. Thus, $\gamma$ and $\gamma'$ are homotopic in $\mathbb R^n_{(F,G)}$. Hence $\Phi$ is an isomorphism.
\end{proof}
\begin{cor}
\label{cor: topological interpretation of cohomologies}
    For any open interval $(D,E)$ in $\widetilde{\Image\phi}$,
    $$\widetilde{H}^i(\mathbb R^n_{(D,E)})\cong \widetilde{H}^i(K(D,E))$$
    where $K(D,E)$ is the order complex of $(D,E)$.
\end{cor}
\begin{proof}
     Let $X:=\mathbb R^n_{(D,E)}$ and $S=(D,E)$. By Proposition \ref{prop: local simple stratification}, $X\to S$ is a simple stratification. Hence by Theorem \ref{thm: exodromy}, 
    $$Sh_S(X)\cong A_S(X)^{op}-mod$$
    By Proposition \ref{lemma: ent cat over interval}, $\Ent_S(X)\cong S$. So, $A_S(X)\cong kS$.
        On the other hand Example \ref{ex: tree stratification of a poset},  demonstrates that
    $$\Sh_S(K(S))\cong (kS)^{op}-mod$$.
    Putting these together we get: 
    $$Sh_S(X)\cong (kS)^{op}-mod\cong \Sh_S(K(S)).$$

    Via these isomorphisms the constant sheaf $k_X$ on $X$ in $\Sh_S(X)$ maps to the constant sheaf $k_{K(S)}$ on $K(S)$ in $\Sh_S(K(S))$. Hence,
                \begin{align*}
                H^i(X)&= H^i(\RHom(k_X,k_X))\\
                &= H^i(\RHom(k_{K(S)},k_{K(S)}))\\
                &= H^i(K(S))
            \end{align*}
    Thus, $H^i(X)=H^i(K(S))$ for all $i$. The result follows.
    %Hence, $\widetilde{H}^i(\mathbb R^n_{(D,E)})\cong \widetilde{H}^i(K(D,E))$ for all $i$.
\end{proof}
Using Proposition \ref{prop: BR hpas are C_S} we can identify the entrance paths on $\mathbb T^n$ with monomials, i.e. the morphisms $p:[D]\to [E]$ of $\Ent_{\Image \Phi}(\mathbb T^n)$ correspond to monomials $x^{\underline{m}}$ where $\deg(\underline{m})=[E-D]$. 
\begin{prop}
\label{prop: top description of cohomologies in BR cases}
    For any $p=x^{\underline{m}}:[D]\to [E]\in \Mor(\Ent_{\Image \Phi}(\mathbb T^n))$,
    $$\widetilde{H}^i(\mathbb R^n_{(D,D+\underline{m})})\cong \widetilde{H}^i(B\mathcal{C}(p))$$
\end{prop}
\begin{proof}
    By part (1) of Proposition \ref{prop: koszulity and universal cover}, the functor 
    $$\Psi:\Ent_{\widetilde{\Image \Phi}}(\mathbb R^n)\to \Ent_{\Image \Phi}(\mathbb T^n)$$ 
        $$x_D\mapsto [D]$$
        $$\widetilde{\gamma}\mapsto \pi\circ \widetilde{\gamma}$$
    is an almost discrete fibration. Hence 
    $$\widetilde{\mathcal{C}}(l_{\underline{m}})\cong\mathcal{C}(p)$$
    where $\mathcal{C}=\Ent_{\Image \Phi}(\mathbb T^n)$, $\widetilde{\mathcal{C}}=\Ent_{\widetilde{\Image \Phi}}(\mathbb R^n)$ and $l_{\underline{m}}:[0,1]\to M_{\mathbb R}$ is the straight line connecting $x_D$ to $x_D+\underline{m}$. By part (2) of Proposition \ref{prop: koszulity and universal cover}, we have an isomorphism 
$$F:\Ent_{\widetilde{\Image \Phi}}(\mathbb R^n)\to \widetilde{\Image\Phi}$$
$$x_D\mapsto D$$
$$q\mapsto t(q)<h(q)$$
    Therefore $B\mathcal{C}(p)$ is homeomorphic to the order complex $K(D,D+\underline{m})$ of the open interval $(D,D+\underline{m})$ in $\widetilde{\Image \Phi}$. Thus, by Corollary \ref{cor: topological interpretation of cohomologies},
    $$\widetilde{H}^i(\mathbb R^n_{(D,D+\underline{m})})\cong \widetilde{H}^i(B\mathcal{C}(p))$$
    for all $i$. 
\end{proof}
\begin{thm}
\label{thm: top description of BR cases}
    Let $k\Ent_{\Image \Phi}(\mathbb T^n)$ be a graded Bondal-Thomsen HPA. Then, $k\Ent_{\Image \Phi}(\mathbb T^n)$ is Koszul if and only if $$\widetilde{H}^i(\mathbb R^n_{(D,E)})=0$$
    for all $(D,E)\subset\widetilde{\Image\Phi}$ and $i$ less than the maximal chain length in $(D,E)$.
    %The following are equivalent:
    %\begin{enumerate}
        %\item $k\Ent_{\Image \Phi}(\mathbb T^n)$ is Koszul.
        %\item $\widetilde{H}^i(\mathbb R^n_{(D,E)})=0$ for all $(D,E)\subset\widetilde{\Image\Phi}$ and $i$ less than the maximal chain length in $(D,E)$.
        %\item $\widetilde{H}^i(\cup_{F\in (D,E)}(-F+[0,1)^{n+k}))=0$ for all $(D,E)\subset\widetilde{\Image\Phi}$ and $i$ less than the maximal chain length in $(D,E)$.
    %\end{enumerate}
\end{thm}
\begin{proof}
    By Theorem \ref{Theorem}, $k\Ent_{\Image \Phi}(\mathbb T^n)$ is Koszul if and only if $\Ent_{\Image \Phi}(\mathbb T^n)$ is locally bouquet. By Proposition \ref{prop: top description of cohomologies in BR cases}, $\Ent_{\Image \Phi}(\mathbb T^n)$ is locally bouquet if and only if 
    $$\widetilde{H}^i(\mathbb R^n_{(D,E)})=0$$
    for all $(D,E)\subset\widetilde{\Image\Phi}$ and $i<\dim K(D,E)$. Now observe that $\dim K(D,E)$ is the maximal chain length in $(D,E)$.
    
    %Since 
    %$$\mathbb R^n_{(D,E)}=(\cup_{F\in (D,E)}(-F+[0,1)^{n+k}))\cap \mu^{-1}_{\mathbb R}(0)$$
    %we have a deformation retraction of $\cup_{F\in (D,E)}(-F+[0,1)^{n+k})$ to $\mathbb R^n_{(D,E)}$. Thus, (2) and (3) are equivalent. 
\end{proof}
\begin{ex}
    Let $n=1, M= \mathbb Z$, and  $\widehat G=\{0\}$. Then, $\Phi =0$ and $\mathbb R_{(m,n)} =(m,n)$ is just an open interval with the usual Euclidean topology.  In this case, $\Ent_{\Image \Phi}(S^1)$ is a category with one object $\{0\}$ whose morphisms are $\{x^n:n\in \mathbb N\}$. Hence, $k\Ent_{\Image \Phi}(S^1)=k[x]$. %By Theorem \ref{thm: top description of BR cases}, $k[x]$ is Koszul if and only if $\widetilde{H}^i((m,n))=0$ for all open intervals $(m,n)$ of $\mathbb R$ with the Euclidean topology and $i<m-n-2$.
    Since open intervals are contractible, $\widetilde{H}^i((m,n))=0$ for all $(m,n)$ and all $i$. Therefore $k[x]$ is Koszul by Theorem~\ref{thm: top description of BR cases} (see also Example~\ref{ex: polynomial ring is Koszul}). 
\end{ex}

\subsection{Collections of line bundles on toric varieties}
Full strong exceptional collections of line bundles play an important role in algebraic geometry.  In \cite{b}, Bondal proved that Koszulity of their endomorphism algebras is related to strongness of the dual exceptional collection.  This allows us to interpret strongness of the dual collection in terms of the Cohen-Macaulay property.  We start by recalling some basic definitions from \cite{b} and \cite{huybrechts}.
\begin{df}
Let $\mathcal{T}$ be a triangulated category. An object $E$ of $\mathcal{T}$ is called exceptional if it satisfies the following conditions:
$$\Hom(E,E[n])=\begin{cases}
    k, \text{ for $n\ne 0$}\\
    0, \text{ otherwise}
\end{cases}$$
\end{df}
\begin{df}
    An ordered collection $(E_0,...,E_n)$ of exceptional objects of a triangulated category $\mathcal{T}$ is called an exceptional collection if $\RHom(E_i,E_j)=0$ for $i>j$. 
\end{df}
\begin{df}
    An exceptional collection $(E_0,...,E_n)$ of a triangulated category $\mathcal{T}$ is called full if $\mathcal{T}$ is generated by $\{E_i\}$, i.e. any full triangulated subcategory of $\mathcal{T}$ containing all $E_i$'s is equivalent to $\mathcal{T}$ via inclusion. 
\end{df}
\begin{df}
The collection $(E_0,...,E_n)$ of objects of a triangulated category $\mathcal{T}$ which satisfy 
\begin{enumerate}
    \item $\Hom(E_i,E_j[k])=0$, for all $i,j$; $k\ne 0$, $k\in \mathbb{Z}$,
    \item  $\Hom(E_i,E_j)=0$, for all $i>j$,
\end{enumerate}
is called a strong exceptional collection. 
\end{df}
\subsubsection{Koszulity for collections of line bundles}
Let $X_{\Sigma}$ be a toric variety. Let $S\subset\Cl(X_{\Sigma})$. Then as we noted in Example \ref{ex: hpas are C_S}, 
$$k\mathcal{C}_S=\End(\bigoplus_{D\in S}\mathcal{O}(D))$$
where $\mathcal{C}$ is a category with one objects and loops such that $k\mathcal{C}=k[x_{\rho}:\rho\in \Sigma(1)]$. 
\begin{prop}
\label{prop: Koszulity and cm of monomial posets}
    Let $\mathcal O_X(D_1), ..., \mathcal O_X(D_n)$ be any collection of line bundles on a toric variety $X$.
 Let $A$ be the endomorphism algebra of $\bigoplus_{i=1}^n \mathcal O_X(D)$.  
 For $1 \leq i \leq n$, consider the poset $P_i$ of all monomials in the Cox ring of degree $D_j - D_i$. Then $A$ is Koszul iff.\ $P_i$ is locally Cohen-Macaulay for all $i$.
\end{prop}
\begin{proof} 
    Let $S=\{D_1, ...,D_n\}\subset \Cl(X)$. Then $A=k\mathcal{C}_S$. By Corollary \ref{cor: Koszul hpa if and only if CM Path_{A,v}}, $A$ is Koszul if and only if $P_i:=\Path_{A,D_i}$ are locally Cohen-Macaulay. But 
    $$P_i = \Path_{A,D_i} = \{\text{all monomials in the Cox ring of degree $D_j - D_i$}\}$$
\end{proof}
\begin{rmk}
    Since shellability implies Cohen-Macaulay, Proposition \ref{prop: Koszulity and cm of monomial posets} generalizes  \cite[Proposition 6.29]{dj}. 
\end{rmk}
\begin{ex}
    Consider the HPA of $\mathbb P^2$ obtained by considering the line bundles $v_0:=\mathcal{O}$, $v_1:=\mathcal{O}(1)$ and $v_2:=\mathcal{O}(2)$:
    
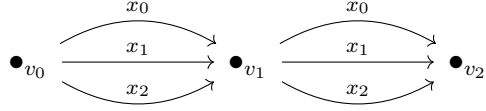
\begin{center}
	       \begin{tikzcd}
                    \bullet_{v_0}\arrow[bend left]{rr}{x_0}\arrow{rr}{x_1}\arrow[bend right]{rr}{x_2}
		          && \bullet_{v_1}\arrow[bend left]{rr}{x_0}\arrow{rr}{x_1}\arrow[bend right]{rr}{x_2}
                  &&\bullet_{v_2}
	       \end{tikzcd}     
        \end{center}   
    $P_0$ is 
    
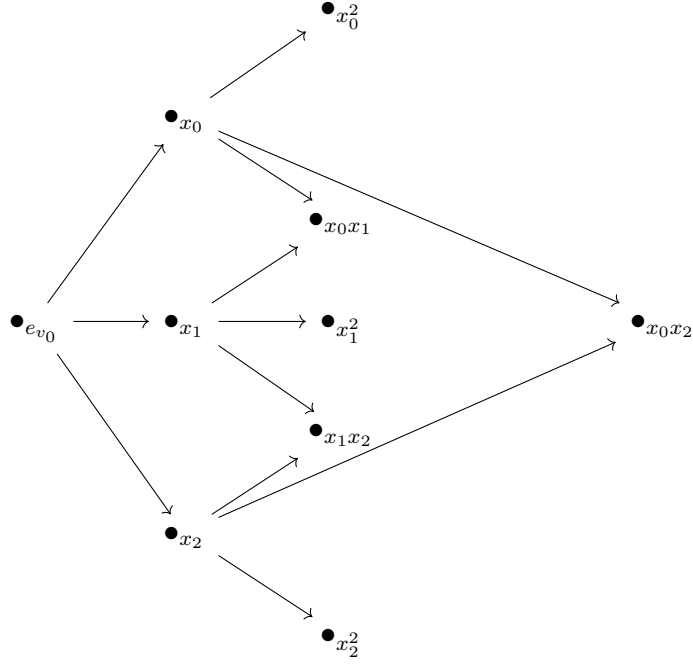
\begin{center}
	       \begin{tikzcd}
                    && \bullet_{x_0^2}\\
                    &\bullet_{x_0}\arrow{ddrrrr}\arrow{ur}\arrow{dr}\\
                    &&\bullet_{x_0x_1}\\
                    \bullet_{e_{v_0}}\arrow{uur}\arrow{r}\arrow{ddr}
		          & \bullet_{x_1}\arrow{ur}\arrow{r}\arrow{dr}&\bullet_{x_1^2}&&&\bullet_{x_0x_2}\\
                    &&\bullet_{x_1x_2}\\
                    &\bullet_{x_2}\arrow{uurrrr}\arrow{ur}\arrow{dr}\\
                    &&\bullet_{x_2^2}
	       \end{tikzcd}     
        \end{center}       
    $P_1$ is
    
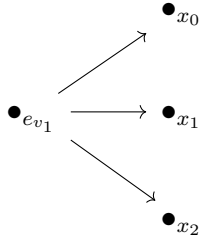
\begin{center}
	       \begin{tikzcd}
                    &\bullet_{x_0}\\
                    \bullet_{e_{v_1}}\arrow{ur}\arrow{r}\arrow{dr}
		          & \bullet_{x_1}\\
                    &\bullet_{x_2}
	       \end{tikzcd}     
        \end{center}       
    $P_2$ is just $\{e_{v_2}\}$. 
\\ They are all locally Cohen-Macaulay. So the HPA of $\mathbb{P}^2$ is Koszul. 
\end{ex}
\begin{ex}
The Hirzebruch surface $\mathbb F_n$ is a toric variety for the complete fan with rays $(1,0),(0,1),(-1,n),(0,-1)$ pictured below with $n=2$ (see Example~\ref{ex: hpa of F_1}
 for the $n=1$ case).
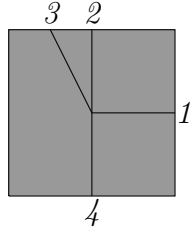
\begin{center}    
\begin{tikzpicture}[scale = .55]
%\draw[step=1cm,black,very thin] (0,0) grid (7,7);
\draw[fill=gray!80] (-2,-2)--(2,-2)--(2,2)--(-2,2)--(-2,-2);
\draw (0,0) -- (2,0);
\draw (0,0) -- (0,2);
\draw (0,0) -- (-1,2);
\draw (0,0) -- (0,-2);
\node (1) at (2.2,0) {1};
\node (2) at (0,2.4) {2};
\node (3) at (-1,2.4) {3};
\node (4) at (0,-2.4) {4};
\end{tikzpicture}
\end{center}
As a generalization of that example, the Cox ring $S = k[x_1,x_2,x_3,x_4]$ is $\mathbb Z^2$-graded with degrees $(1,0),(-n,1),(1,0),(0,1)$.  Again, we obtain an HPA from the degree zero piece of the endomorphism algebra of $S \oplus S(1,0) \oplus S(0,1) \oplus S(1,1)$ or equivalently we the corresponding line bundles $v_0:=\mathcal{O}$, $v_1:=\mathcal{O}(1,0)$, $v_2:=\mathcal{O}(0,1)$ and $v_3:=\mathcal{O}(1,1)$. 
The general quiver can be pictured as follows
    
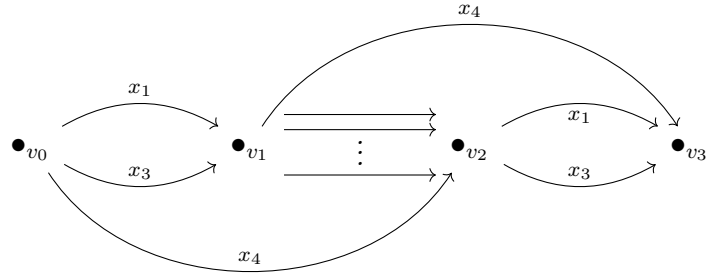
\begin{center}
	       \begin{tikzcd}
                    \bullet_{v_0}\arrow[bend left]{rr}{x_1}\arrow[bend right=20mm]{rrrr}{x_4}\arrow[bend right]{rr}{x_3}
		          && \bullet_{v_1}\arrow[shift left=4 mm]{rr}\arrow[shift left=2mm]{rr}\arrow[rr, phantom, "\vdots"]\arrow[shift right=4 mm]{rr}\arrow[bend left=20mm]{rrrr}{x_4}
                  &&\bullet_{v_2}\arrow[bend left,swap]{rr}{x_1}\arrow[bend right]{rr}{x_3}
                  &&\bullet_{v_3}
	       \end{tikzcd}     
        \end{center}
where the middle arrows from $v_1$ to $v_2$ are labeled by $x_1^{n-1}x_2$, $x_1^{n-2}x_2x_3$, ... and $x_2x_3^{n-1}$. The relations are given by commutativity of the variables. 
 For $n=1$ this is not quadratic and hence it is not Koszul (see also Example \ref{ex: universal cover stratification of F_1} and Example \ref{ex: F_1 is not Koszul}). Consider the case of $n=2$. Then this HPA is quadratic and $P_0$ is as follows:
    
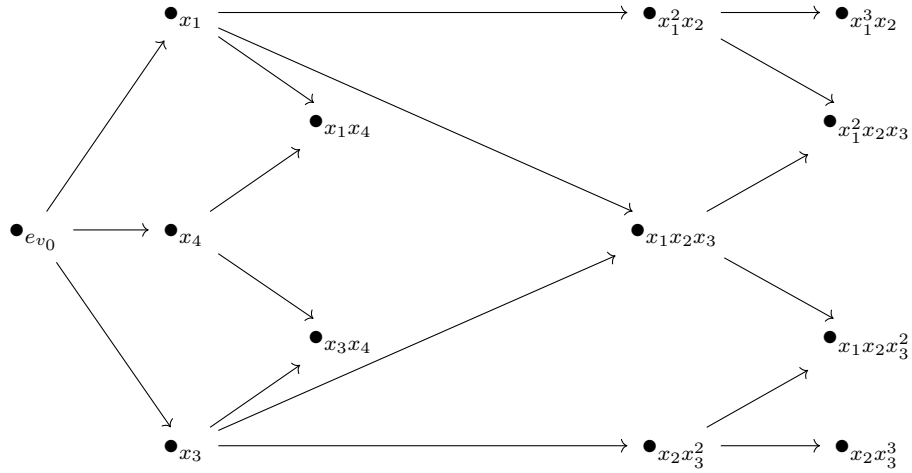
\begin{center}
\begin{tikzcd}
    &\bullet_{x_1}\arrow{ddrrrr}\arrow{rrrr}\arrow{dr}&&&&\bullet_{x_1^2x_2}\arrow{r}\arrow{dr}&\bullet_{x_1^3x_2}\\
    & &\bullet_{x_1x_4}&&&&\bullet_{x_1^2x_2x_3}  \\
   \bullet_{e_{v_0}} \arrow{uur}\arrow{r}\arrow{ddr}& \bullet_{x_4} \arrow{ur}\arrow{dr}&&&& \bullet_{x_1x_2x_3}\arrow{ur} \arrow{dr}\\
   &&\bullet_{x_3x_4}&&&&\bullet_{x_1x_2x_3^2}\\
    & \bullet_{x_3} \arrow{uurrrr}\arrow{ur}\arrow{rrrr}&&&&\bullet_{x_2x_3^2}\arrow{r}\arrow{ur}&\bullet_{x_2x_3^3}
\end{tikzcd} 
        \end{center}
In this poset, for example the interval $(e_{v_1},x_1^2x_2x_3)$ pictured below
    
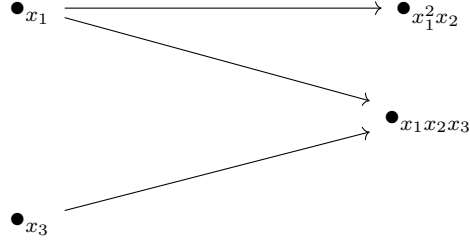
\begin{center}
\begin{tikzcd}
    \bullet_{x_1}\arrow{drrrr}\arrow{rrrr}&&&&\bullet_{x_1^2x_2}\\
    &&&& \bullet_{x_1x_2x_3}\\
    \bullet_{x_3} \arrow{urrrr}
\end{tikzcd} 
        \end{center}
is Cohen-Macaulay. Indeed all the intervals of $P_i$'s are homeomorphic to lines, or are collection of points or are empty. Hence this HPA is Koszul. For $n\geq 2$, among the $P_i$s, all the intervals of the form $(e_{v_0},x_1^rx_2x_3^s)$ with $r+s \neq n+1$ are either empty or a collection of points and hence Cohen-Macaulay. For the intervals of the form $(e_{v_0},x_1^rx_2x_3^s)$ with $r+s=n+1$, we have the following different cases: 
\begin{itemize}
    \item $r=n+1,s=0$: In this case 
    $$(e_{v_0},x_1^{n+1}x_2)=\{x_1\leq x_1^nx_2\}$$
    which is of course Cohen-Macaulay. 
    \item $r=n,s=1$:
    In this case $(e_{v_0},x_1^nx_2x_3)$ is as follows: 
    
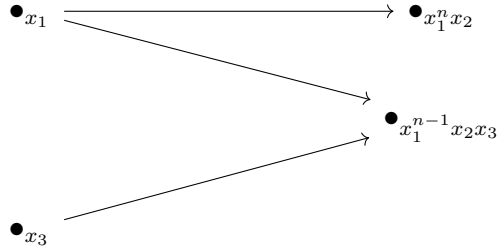
\begin{center}
\begin{tikzcd}
    \bullet_{x_1}\arrow{drrrr}\arrow{rrrr}&&&&\bullet_{x_1^nx_2}\\
    &&&& \bullet_{x_1^{n-1}x_2x_3}\\
    \bullet_{x_3} \arrow{urrrr}
\end{tikzcd} 
        \end{center}
        which is Cohen-Macaulay. 
    \item $r,s\geq 2$:
    In this case $(e_{v_0},x_1^rx_2x_3^s)$ is as follows: 
        
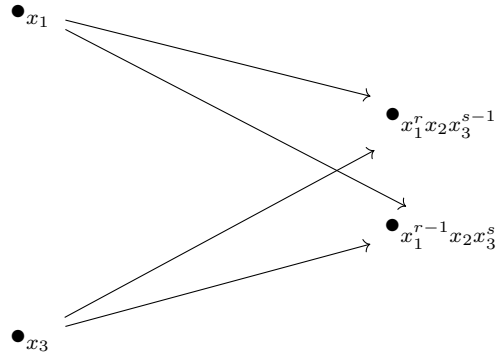
\begin{center}
\begin{tikzcd}
    \bullet_{x_1}\arrow{drrrr} \arrow{ddrrrr}\\
    &&&& \bullet_{x_1^rx_2x_3^{s-1}}\\
    &&&& \bullet_{x_1^{r-1}x_2x_3^s}\\
    \bullet_{x_3} \arrow{urrrr}\arrow{uurrrr}
\end{tikzcd} 
        \end{center}
        which is homeomorphic to the circle $S^1$ and is Cohen-Macaulay. 
    \item $r=1,s=n$:
    In this case $(e_{v_0},x_1x_2x_3^n)$ is as follows: 
        
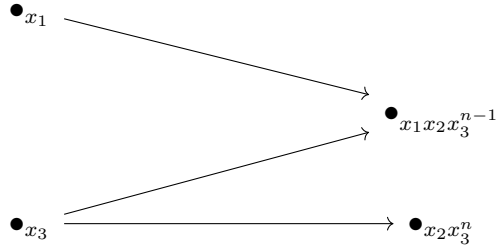
\begin{center}
\begin{tikzcd}
    \bullet_{x_1}\arrow{drrrr}\\
    &&&& \bullet_{x_1x_2x_3^{n-1}}\\
    \bullet_{x_3} \arrow{urrrr}\arrow{rrrr}&&&&\bullet_{x_2x_3^n}
\end{tikzcd} 
        \end{center}
    which is Cohen-Macaulay. 
    \item $r=0,s=n+1$: In this case 
    $$(e_{v_0},x_2x_3^{n+1})=\{x_3\leq x_2x_3^n\}$$
    which is of course Cohen-Macaulay. 
\end{itemize}
Hence for all $n\geq 2$ this HPA is Koszul. 

\end{ex}

\subsubsection{Strong exceptional collections and Koszulity}
Let $X_{\Sigma}$ be a toric variety. Let $\{\mathcal O_X(D_1), ..., \mathcal O_X(D_n)\}$ be a full strong exceptional collection of line bundles on $X_{\Sigma}$ and $A=\End(\bigoplus_{i=1}^n\mathcal{O}(D_i))$. Then $A=k\mathcal{C}_S$ (see Example \ref{ex: hpas are C_S}). We set $S_i$ to be the simple object corresponding to $\mathcal{O}(D_i)$ as an object of $\mathcal{C}_S$. 
In the bounded derived category $D^b(X_{\Sigma})$, 
$$\Hom(S_i,S_j[k])=\Ext^k(S_i,S_j)$$
and Koszulity is just the vanishing of certain graded pieces 
$$\Ext^k(S_i,S_j)=\bigoplus_{n\in \mathbb{Z}}\Ext^k(S_i,S_j)_n$$
We assume the partial ordering by morphisms is graded via 
$$f:\{\mathcal O_X(D_1), ..., \mathcal O_X(D_n)\}\to \mathbb{N}.$$ 
For clarity, we provide a proof of our interpretation of \cite[Corollary 7.2, p. 8]{b}.
\begin{prop}
\label{prop: Koszul iff strong}
Let $\mathcal O_X(D_1), ..., \mathcal O_X(D_n)$ be a full strong exceptional collection of line bundles on a toric variety $X_{\Sigma}$ such that the partial ordering by morphisms on objects is graded via a grading $f:\{\mathcal O_X(D_1), ..., \mathcal O_X(D_n)\}\to \mathbb{N}$ such that for any arrow $f(h(a))-f(t(a)) =1$. Then $A$ is Koszul if and only if $S_n[-f(\mathcal O_X(D_n))],...,S_1[-f(\mathcal O_X(D_1))]$ is a strong exceptional collection. 
\end{prop}
\begin{proof}
To declutter notation, set $v_i:=\mathcal O_X(D_i)$. By \cite{b}, 
\begin{equation} \label{eq: dual collection}
\{S_n[-f(v_n)],...,S_1[-f(v_1)]\}
\end{equation}
is an exceptional collection. It is strong if and only if 
    $$\Hom(S_i[-f(v_i)],S_j[-f(v_j)][l])= \Ext^{f(v_i)-f(v_j)+l}(S_i, S_j) = 0$$
    for all $l\ne 0$.
%     i.e. $\Ext^{l}(S_i[-f(v_i)],S_j[-f(v_j)])=0$. But,
%     $$\Ext^{l}(S_i[-f(v_i)],S_j[-f(v_j)])=\Ext^{f(v_i)-f(v_j)+l}(S_i, S_j)$$
% So, it is strong if and only if $\Ext^{f(v_i)-f(v_j)+l}(S_i, S_j)=0$ for all $l\ne 0$. 

Now since $\length(p)=f(h(p))-f(t(p))$ for $p\in \Mor(\mathcal{C}_S)$, by Proposition \ref{Ext groups}, 
$$\Ext^{f(v_i)-f(v_j)+l}(S_i, S_j)_{-n}=0$$
for all $n\ne f(v_i)-f(v_j)$. 
Therefore, 
$$\Ext^{f(v_i)-f(v_j)+l}(S_i, S_j)=\Ext^{f(v_i)-f(v_j)+l}(S_i, S_j)_{f(v_j)-f(v_i)}$$ 
Hence, \eqref{eq: dual collection} is strong if and only if $\Ext^{f(v_i)-f(v_j)+l}(S_i, S_j)_{f(v_j)-f(v_i)}=0$ for $l\ne 0$ which by Corollary \ref{Koszulity based on simples} is equivalent to Koszulity. 
\end{proof}
\begin{thm}
\label{thm: Koszul-Strong-CM}
Let $\mathcal O_X(D_1), ..., \mathcal O_X(D_n)$ be any collection of line bundles on a toric variety $X_{\Sigma}$.
 Let $A$ be the endomorphism algebra of $\bigoplus_{i=1}^n \mathcal O_X(D)$.  
 For $1 \leq i \leq n$, consider the poset $P_i$ of all monomials in the Cox ring of degree $D_j - D_i$. Then $A$ is Koszul iff.\ $P_i$ is locally Cohen-Macaulay for all $i$.
Furthermore, if the collection is a full strong exceptional collection such that the partial ordering by morphisms on objects is graded, then the above are equivalent to the existence of integers $d_1, ..., d_n$ such that the shifted dual exceptional collection $S_1[d_1], ..., S_n[d_n]$ is strong. 
\end{thm}
\begin{proof}
    The first part is Proposition \ref{prop: Koszulity and cm of monomial posets}. The second part is Proposition \ref{prop: Koszul iff strong}. 
\end{proof}
\begin{ex}
    Let $X_{\Sigma}=\mathbb{P}^1\times \mathbb P^1$. Consider the collection $\{v_1:=\mathcal{O}(0,0),v_2:=\mathcal{O}(1,0),v_3:=\mathcal O(0,1), v_4:=\mathcal{O}(1,1)\}$ and let $A$ be its HPA (See Example \ref{ex: P^1xP^1}).
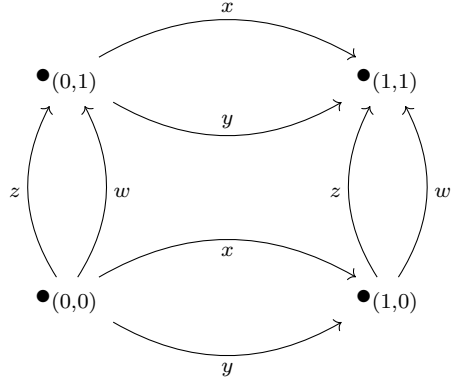
        \begin{center}
	       \begin{tikzcd}
                    \bullet_{(0,1)}\arrow[bend left]{rrr}{x}\arrow[bend right]{rrr}{y}
		          &&&\bullet_{(1,1)}\\\\\\
                    \bullet_{(0,0)}\arrow[swap, bend left]{rrr}{x}\arrow[swap, bend right]{rrr}{y}\arrow[bend left]{uuu}{z}\arrow[swap, bend right]{uuu}{w}
		          &&&\bullet_{(1,0)}\arrow[bend left]{uuu}{z}\arrow[swap, bend right]{uuu}{w}
	       \end{tikzcd}     
        \end{center} 
        In this example the collection $\{\mathcal{O}(0,0),\mathcal{O}(1,0),\mathcal O(0,1), \mathcal{O}(1,1)\}$ is a full strong exceptional collection. Moreover the partial ordering by morphisms on objects can be graded by
    $$f:\{\mathcal{O}(0,0),\mathcal{O}(1,0),\mathcal O(0,1), \mathcal{O}(1,1)\}\to \mathbb N$$
    $$(i,j)\mapsto i+j$$
    Furthermore this is obviously locally Cohen-Macaulay. Hence $A$ is Koszul. Thus by Theorem \ref{thm: Koszul-Strong-CM} $\{S_4[-2],S_3[-1],S_2[-1],S_1\}$ is a strong exceptional collection. 
\end{ex} 

\begin{ex}
Let $X_\Sigma = \mathbb P(a_0: \cdots : a_n)$ be weighted projective space and consider the collection $\mathcal O, ..., \mathcal O(\sum a_i - 1)$.  The arrows do not satisfy the condition of Theorem \ref{thm: Koszul-Strong-CM} unless $a_i =1$ for all $i$.  However, it is always locally Cohen-Macaulay (which follows from the locally Cohen-Macaulay property of $\mathbb N^{n+1}$).  Therefore, the endomorphism algebra of $\mathcal O, ..., \mathcal O(\sum a_i - 1)$ is Koszul but the dual collection cannot be made strong unless all the weights are 1.
\end{ex}

\printbibliography
\Addresses
%\nocite{*}
\end{document}